\documentclass{amsart}
\usepackage{yhmath}
\usepackage{subcaption}
\usepackage[colorlinks,linkcolor=blue,anchorcolor=blue,citecolor=blue]{hyperref}
\usepackage{amsmath}
\usepackage{amssymb,bbm}
\usepackage{amsfonts}
\usepackage{latexsym}
\usepackage{amsthm}
\usepackage{graphicx}
\usepackage{pgf}
\usepackage{tikz}
\usepackage[T1]{fontenc}
\usepackage[english]{babel}
\usepackage{listings}
\usepackage{xcolor,mathrsfs,url}
\usepackage{float}
\usepackage{enumitem}

\usepackage{color,soul}
\usepackage{comment}

\usepackage[a4paper, total={6.8in, 9in}]{geometry}

\newtheorem{theorem}{Theorem}[section]
\newtheorem{lemma}[theorem]{Lemma}
\newtheorem{corollary}[theorem]{Corollary}

\theoremstyle{definition}

\newtheorem{proposition}[theorem]{Proposition}
\newtheorem{drhp}[theorem]{$\bar\partial$-RH problem}

\newtheorem{rhp}[theorem]{RH problem}
\newtheorem{dbarproblem}[theorem]{$\bar\partial$-problem}
\newtheorem{remark}[theorem]{Remark}

\numberwithin{equation}{section}

\newcommand{\abs}[1]{\lvert#1\rvert}
\newcommand{\norm}[1]{\parallel#1\parallel}

\newcommand{\ddddd}{\mathrm{d}}

\newcommand{\pp}{\mathcal{P}}

\newcommand{\oo}{\mathcal{O}}
\newcommand{\cc}{\mathcal{C}}
\newcommand{\re}{\mathrm{Re}}
\newcommand{\I}{\mathrm{Im}}

\newcommand{\LL}{{\mathcal L}}
\newcommand{\ii}{\mathrm{i}}
\newcommand{\identity}{\mathrm{Id}}
\newcommand{\ylk}{\left(\begin{matrix}
		1\\0
	\end{matrix}\right)}
\newcommand{\ly}{\left(\begin{matrix}
		0\\1
	\end{matrix}\right)}

\begin{document}

\title[The Focusing Ablowitz-Ladik system]{Soliton resolution  and Painlev\'e Asymptotics for the focusing Ablowitz-Ladik system}

\author{Meisen Chen}
\address{School of Mathematics and Statistics, Fujian Normal University, Fuzhou, PR China}

\email{chenms@fjnu.edu.cn}

\author{Engui Fan}
\address{School of Mathematical Sciences,
	Fudan University, Shanghai, PR China}
\email{faneg@fudan.edu.cn}

\author{Zhaoyu Wang}
\address{School of Mathematical Sciences,
	Fudan University, Shanghai, PR China}
\email{wang\_zy@fudan.edu.cn}

%    General info
%\subjclass[2020]{Primary 54C40, 14E20; Secondary 46E25, 20C20}

\date{\today}

%\keywords{Global wellposedness, Fredholm theory, long-time asymptotics}

\begin{abstract}
{We investigate the soliton resolution and Painlev\'e asymptotics for the focusing Ablowitz-Ladik system with  the initial data in  a discrete weighted $\ell^2$ space.
First, we  establish the global well-posedness of this initial-value problem, which is  further reformulated as a  Riemann-Hilbert problem with higher-order poles.
	Using Fredholm theory, the Riemann-Hilbert problem with the jump contour consisting of three circles centered around the origin is uniquely solved.
	Then, by performing a $\bar\partial$-nonlinear steepest descent method to  the Riemann-Hilbert problem,   we  obtain  the    asymptotic approximation  to  the solution
   of  the focusing Ablowitz-Ladik system for large time in different space-time regions of the $(n,t)$-half plane.  In the sectors $\{(n,t): n /(2t) <-M_0 \}$ and $\{(n,t): n /(2t) >M_0 \}$, where $M_0$ is a positive constant, the leading order  asymptotics is dominated by the solitons; while in the  sector  $\{(n,t): |n /(2t) -1 <M_0^{-1} \}$,  the long-time asymptotics is influenced by both the  solitons and  the oscillations;  In the two transition zones  $\{(n,t): |n /(2t)+1|t^{2/3} <C \}$ and   $\{(n,t): |n /(2t)-1|t^{2/3} <C \}$ with $C$ being a positive constant,
  we find  the Painlev\'e-type  asymptotics  which can be  expressed in terms of the solution of the second Painlev\'e transcendents.
}\vspace{2mm}

\noindent {KEYWORDS:}  Focusing Ablowitz-Ladik system, higher-order soliton, Global well-posedness, Fredholm theory, long-time asymptotics,
soliton resolution,  Painlev\'e transcendents\vspace{2mm}

\noindent {MSC 2020:} 35P25, 35Q15, 35B40, 37K40
%keywords or MSC also are shown in the botton of the first page.

\end{abstract}

\maketitle

\tableofcontents

\section{Introduction}\label{sec1}

The Ablowitz-Ladik (AL) system  reads
\begin{align}\label{e1.1}
	\ii\frac{\ddddd}{\ddddd t}q_n=q_{n+1}-2q_n+q_{n-1}+\sigma\lvert q_n\rvert^2(q_{n+1}+q_{n-1}),
\end{align}
where  $\sigma=1$  is called the focusing and $\sigma=-1$  the defocusing AL system.
The AL system (\ref{e1.1}), first introduced by Ablowtiz and Ladik in the 1970s \cite{Ablowitz1975nonlinear,ablowitz1976nonlinear},
is    the completely integrable discretization of the nonlinear Schr\"odinger (NLS) equation
\begin{align}\label{e1.2}
	\ii u_t=u_{xx}+\sigma\abs{u}^2u.
\end{align}
The multi-Hamilton structure of the   AL system   was  constructed by Gekhtman and Nenciu \cite{gekhtman2009multi,mazzuca2023large,li2012periodic}.
By exploiting the connection between AL system  and orthogonal polynomials on the unit circle, Nenciu derived the Lax pair for the AL system   \cite{nenciu2005lax}.
 Miller et al obtained the finite genus solutions for the AL system  with an  algebro-geometric method \cite{miller1995finite}.
 Further Gesztesy et al   provide a detailed derivation of all complex-valued algebro-geometric finite-band solutions of the  AL hierarchy using polynomial recursive method \cite{Gesztesy}.
Xia and Fokas  solved the initial-boundary value problem of the AL system with  Fokas' unified  method \cite{xia2018initial}.
  Killip et al studied the continuum limits of the AL system \cite{killip2023continuum,killip2024modified}.
The study of Gibbs measures in the context of the AL  system has gained attention in recent years  \cite{angelopoulos2020invariant,grava2023generalized}.

The inverse scattering transform is a breakthrough development to  the integrable equations, which is first introduced by Gardner, Greene, Kruskal, and Miura to solve the Korteweg-de Vries equation \cite{gardner1967method}.
Using the inverse scattering transform, Ablowitz and Ladik investigated  the integrability of some discrete  equations, including AL system, Toda lattice, discrete sine-Gordon equation  \cite{Ablowitz1975nonlinear, ablowitz1976nonlinear}.
The inverse scattering transform of the AL systems \eqref{e1.1} in vanishing and non-vanishing background have been well investigated \cite{chen2021robust, Ablowitz2007inverse,ablowitz2004discrete,Ortiz2019inverse,Prinari}.
In general,  the higher-order solitons will  present with the discrete spectral data of multiple poles.
however, the higher-order solitons not always exist for all the integrable systems.
For example,  it has been acknowledged that the focusing NLS  equation has higher-order soliton solutions \cite{shabat1972exact}.
The asymptotic analysis of the $N$-soliton  for  the NLS equation  has been investigated as  $t\to+\infty$ \cite{schiebold2017asymptotics} and $N\to+\infty$ \cite{bilman2019large}.
There are no discrete spectrum for the defocusing NLS  equation with vanishing background,    only simple-pole   discrete spectrum
  for the defocusing NLS  equation with nonvanishing background \cite{cuccagna2016asymptotic}.
It is also known  that there is no soliton solution for the defocusing AL system \cite{ablowitz2004discrete,chen2022l2}.
In the Riemann-Hilbert (RH)  approach, the discrete spectrum is the pole of the RH problem in the RH approach.
By solving the RH problem with the poles, the soliton solutions of the integrable systems are obtained with  the reconstruction formula.
One crucial element required in the inverse scattering theory  is  the solbability of the RH problem.
In this article, we will   rigorously solve  the RH problem for the focusing AL system with multiple poles.

It is well acknowledged that on the complex plane, a matrix-valued RH problem is not always solved \cite{Bolibrukh1990}.
In 1989, based on the Fredholm theory, Zhou developed the vanishing lemma to solve   matrix-valued  RH problems on  the real line $\mathbb{R}$  \cite{zhou1989riemann}.
In Zhou's work, the jump matrix is skew-Hermitian outside $\mathbb{R}$: $v(z)=v(\bar z)^\dagger$,  and  positive definite along $\mathbb{R}$, that is, $\re\ v|_\mathbb{R}>\mathbf{0}$.
The vanishing lemma for the RH problem   ensures  that the Fredholm alternative is zero, thereby guaranteeing solvability of the matrix-valued RH problem.
In certain discrete integrable systems, the discrete and continuous spectrum are symmetric about the unit circle such as  the AL system  \cite{Ablowitz2007inverse,ablowitz2004discrete,Prinari,Ortiz2019inverse}, Toda system \cite{kruger2009long2}, discrete sine-Gordon equation \cite{chen2021riemann}.
However, Zhou's vanishing lemma does not apply to these discrete equations.
To guarantee the solvability of the corresponding RH problem for the focusing AL system, we reformulate it into a pure RH problem, which convert the residue condition for the higher order poles into jumps, from which  we develop the vanishing lemma for the focusing AL system  employing  Fredholm theory.

The long-time asymptotic behavior for the   NLS equation  \eqref{e1.2}  has been studied thoroughly by using Deift-Zhou  steepest descent method and $\bar\partial$-steepest descent method \cite{deift2003long,deift1993long,deift1994long2,deift1994long,dieng2008long,BM1,Miller1,Monvel2}.
The Deift-Zhou  steepest descent method was first developed by  Deift and Zhou
to rigorously obtain the long-time asymptotics behavior of the solution for the mKdV
equation by deforming contours to reduce the original RH problem to a model one whose
solution is calculated in terms of parabolic cylinder functions \cite{deift1993steepest}.
The $\bar\partial$-steepest descent method was first introduced for the asymptotic analysis of the orthogonal polynomials \cite{mclaughlin2006delta,mclaughlin2008steepest}.
Later  this method is applied to  investigate the long-time asymptotics, soliton resolution and  asymptotic stability of $N$-soliton solutions for
   for  integrable systems  \cite{Borghese2018long, jenkins2018soliton, cuccagna2016asymptotic,YF1,YF2,wang2023defocusing}.
  Recently,  we focused on discrete integrable systems  and  investigated  the inverse scattering theory for the defocusing AL system in the discrete weighted $l^2$ space \cite{chen2022l2}, and
   the long-time asymptotic behavior   for the solution  of the initial-value problem  on the background of lower regularity \cite{chen2024long}.
The results show that  the upper half plane $(n,t)$ is divided into three sectors: two of which are fast decaying sectors with the rate $\oo(t^{-1})$ while the other one is the oscillatory sector, where the long-time asymptotic behavior agrees with that for the defocusing  NLS  equation \cite{deift2003long}.

In this article,   we apply the Fredholm theory to  prove  the global well-posedness   of the initial-value problem for the focusing AL system
\begin{align}\label{e1}
	\begin{cases}
		&\ii\frac{\ddddd}{\ddddd t}q_n=q_{n+1}-2q_n+q_{n-1}+\lvert q_n\rvert^2(q_{n+1}+q_{n-1}),\\
		& q_n(t=0)=q_n(0)\in \ell^{2,1},
	\end{cases}
\end{align}
where $n\in\mathbb{Z}$ is the spatial variable and $t\in\mathbb{R}$ is the time variable.
Further  we investigate its long-time asymptotics of the initial-value problem (\ref{e1}) based on its corresponding RH problem.
  We   impose  the RH  problem without   spectral singularity, but  allowing the higher order poles.
In this context,   we reduce the RH problem to a new RH problem   \ref{r2.7}  with a jump contour consisting of three circles.
It is interesting that RH problem \ref{r2.7} is a pure RH problem with no discrete spectrum.
Utilizing the vanishing lemma, we develop the Fredholm theory to prove the solvability of the RH problem \ref{r2.7}.

\subsection{Main results}

It can be shown that the initial-value problem \eqref{e1}  is globally well-posed, and  the reflection coefficient $r(\lambda)$ belongs to $H^1$, which means that it is $\frac{1}{2}$-H\"older continuous on the unit circle.
In the residue condition of RH problem \ref{r1}, the order of the poles is set to be $\alpha_{l}$ for any $l=1,\dots,l_0$, and we remove the poles by some triangular rational functions, and equivalently transform RH problem \ref{r1} into RH problem \ref{r2.7} without poles.
The solvability of RH problem \ref{r1} is shown  by the Fredholm theory and the vanishing lemma.
Further we  study  the long-time asymptotics of $q_n(t)$    in the following different sectors of the $(n,t)$-half plane:
% we construct the long-time asymptotic formulas.
%Defining $\xi=\frac{n}{2t}$ and some $M_0>1$ and some positive constant $C$, according to %the distribution of $\xi$, we divide the upper-half plane $(n,t)$into these regions:
\begin{itemize}
	\item Sector I:  $\{(n,t):\xi>M_0\}$,
%	\item Sector II: $\{(n,t):\abs{\xi+1}t^{2/3}<C\}$,
	\item Sector II: $\{(n,t):\abs{\xi-1}<M_0^{-1}\}$,
%	\item Sector IV:  $\{(n,t):\abs{\xi-1}t^{2/3}<C\}$,
	\item Sector III:  $\{(n,t):\xi<-M_0 \}$,
\end{itemize}
where  $\xi=\frac{n}{2t}$, with  a positive constant $M_0>1$.
See FIGURE \ref{regions}. Besides, there are transition zones between the sectors.  The precise definitions of these transition zones are given as follows:
\begin{itemize}
		\item 1st Transition zone: $\{(n,t):\abs{\xi-1}t^{2/3}<C\}$ (between sector I and II),
		\item 2nd Transition zone:  $\{(n,t):\abs{\xi+1}t^{2/3}<C\}$  (between sector II and III),
\end{itemize}
where $C$ is any positive constant.

\begin{figure}
	\centering
	\includegraphics[width=0.7\textwidth]{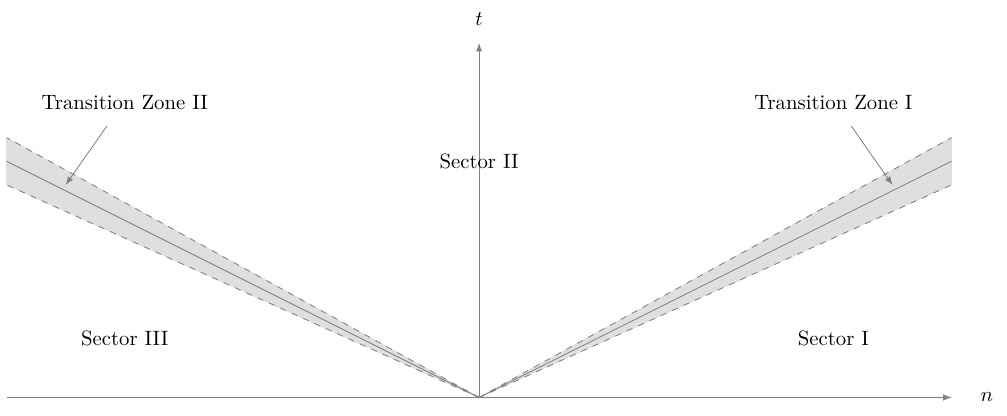}
	\caption{ The different asymptotic sectors of the $(n,t)$-half plane, where $\xi =\frac{n}{2t}$.}\label{regions}
\end{figure}
In particular, in sector II, the leading terms consist of solitons and the oscillatory part, while the remaining is dominated by $\oo(t^{-\frac{3}{4}})$.
In sector I and III, the leading terms consist of only solitons and the remaining is dominated by $\oo(t^{-1})$.
In 1st and 2nd transition zones,  the leading terms are expressed in terms of dedicated solutions of the  Painlev\'e II equation.
 The main results are stated as follows. %关于这句话painleve的表述我再看看
\begin{theorem}
The initial-value problem \eqref{e1} is uniquely solved, and the reflection coefficient $r(\lambda)$ belongs to $H^1$ when assuming the scattering coefficient $a(\lambda)$ non-vanishing along the unit circle $\{\lambda \in \mathbb{C}: \abs{\lambda}=1\}$.
	Let  $q_n(t)$ be   the solution of \eqref{e1},
	 then $q_n(t)$ admits the following asymptotics  in the $(n,t)$-half plane  as $t\to+\infty$.
	\begin{itemize}
		\item In sector II, %we have that as $t\to+\infty$,
		\begin{equation}\label{e5.3}
			\begin{split}
				q_{n}(t)&=T_0\left(q^{\mathcal{Z}_\xi}_{n}(t)+t^{-\frac{1}{2}}q^{Osc}(\xi)\right)+\oo(t^{-\frac{3}{4}}),
			\end{split}
		\end{equation}
		where  $T_0 = \prod_{\lambda_l\in\mathcal{Z}_\xi^-}\abs{\lambda_{l}}^{2\alpha_{l}}e^{-\int_{\arg S_2}^{\arg S_1}\ln(1+\abs{r(e^{\ii\theta})}^2)\frac{\ddddd \theta}{2\pi}}$, 
		and $q^{\mathcal{Z}_\xi}_n(t)$ is the solitonic part
		\begin{align}\label{qzxi}
			q^{\mathcal{Z}_{\xi}}_n(t)=M^{(\mathcal{Z}_\xi)}_{1,2}(0,n+1,t),
		\end{align}
		while $q^{Osc}(\xi)$ is the oscillatory part
		\begin{equation}\label{e5.5}
			\begin{split}
				&q^{Osc}(\xi)=(-T_1^2M^{(\mathcal{Z}_\xi)}_{1,2}(S_1)\overline{M^{(\mathcal{Z}_\xi)}_{2,1}(S_1)}\gamma_2(r(S_1))+T_1^{-2}\abs{M^{(\mathcal{Z}_\xi)}_{1,1}(S_1)}^2\gamma_1(r(S_1))\\
				&\quad-T_2^2M^{(\mathcal{Z}_\xi)}_{1,2}(S_2)\overline{M^{(\mathcal{Z}_\xi)}_{2,1}(S_2)}\gamma_1(\overline{r(S_2)})
				+T_2^{-2}\abs{M^{(\mathcal{Z}_\xi)}_{1,1}(S_2)}^2\gamma_2(\overline{r(S_2)}))\frac{1}{\sqrt{2}(1-\xi^2)^{\frac{1}{4}}}.
			\end{split}
		\end{equation}	
	Here, $M^{(\mathcal{Z}_\xi)}(\lambda)$ is  defined as \eqref{e5.43} and  $M^{(\mathcal{Z}_\xi)}_{i,j}(\lambda), \; i,j=1,2,$ denotes the $(i,j)$ elements of $M^{(\mathcal{Z}_\xi)}(\lambda)$.
		$S_1$ and $S_2$ are sationary phase points given by
		\begin{align*}
			&S_1=-\ii\xi-\sqrt{1-\xi^2},\quad S_2=-\ii\xi+\sqrt{1-\xi^2},\\
			&T_j=\prod_{\lambda_l\in\mathcal{Z}_\xi^-}\left(\frac{S_j-\lambda_l}{S_j-\bar\lambda_l^{-1}}\right)^{\alpha_l}\frac{(-2\sqrt{2}(1-\xi^2)^\frac{3}{4}t^{\frac{1}{2}})^{(-1)^{j-1}\ii\nu_j}(\ii S_j)^{(-1)^{j}\ii\nu_j}}{e^{\ii t[(-1)^{j}\sqrt{1-\xi^2}-\xi\arg S_j-1]-\tilde\alpha_j(S_j)}}, 	\	j=1,2,\\
			&\nu_j=-\frac{\ln(1+\abs{r(S_j)}^2)}{2\pi},\quad \tilde\alpha_j(\lambda)=\int_{S_2}^{S_1}\frac{\ln(1+\abs{r(\varsigma)}^2)-\ln(1+\abs{r(S_j)}^2)}{\varsigma-\lambda}\frac{\ddddd \varsigma}{2\pi\ii}.\\
			&\gamma_1=\gamma_1(\tau)=\frac{\sqrt{2\pi}e^{\ii\pi/4-\pi\nu(\tau)/2}}{\tau\Gamma(-\ii\nu(\tau))},\quad \gamma_2=\gamma_2(\tau)=\frac{\sqrt{2\pi}e^{-\ii\pi/4-\pi\nu(\tau)/2}}{\bar\tau\Gamma(\ii\nu(\tau))}.
		\end{align*}
		
		\item In sectors I and III, % we have that as $t\to+\infty$,
		\begin{align}\label{e1.15}
			q_n(t)=T_0q_n^{\mathcal{Z}_\xi}(t)+\oo(t^{-1}),
		\end{align}
		where
		$q^{\mathcal{Z}_\xi}_n(t)$ is the solitonic part given by \eqref{qzxi} and
		\begin{align*}
			T_0=\begin{cases}
				\prod_{\lambda_l\in\mathcal{Z}_\xi^-}\abs{\lambda_{l}}^{2\alpha_{l}}&\text{sector I},\\
				e^{-\int_{0}^{2\pi}\ln(1+\abs{r(e^{\ii\theta})}^2)\frac{\ddddd\theta}{2\pi}}\prod_{\lambda_l\in\mathcal{Z}_\xi^-}\abs{\lambda_{l}}^{2\alpha_{l}}&\text{sector III}.
			\end{cases}
		\end{align*}
		\item In 1st and 2nd transition zones, %we have that as $t\to+\infty$,
		\begin{equation}\label{qn}
			\begin{split}
				q_{n}(t)&= T_0 \left(q^{\mathcal{Z}_\xi}_{n}(t)+t^{-\frac{1}{3}}q^{P}(\xi)\right)+\oo(t^{-\frac{1}{2}}),
			\end{split}
		\end{equation}
		where	$$T_0 = \begin{cases}
			\prod_{\lambda_l\in\mathcal{Z}_\xi^-}\abs{\lambda_{l}}^{2\alpha_{l}}&\text{1st transition zone}, \\
			e^{-\int_{0}^{2\pi}\frac{\ln(1+\abs{r(e^{\ii\theta})}^2)\ddddd \theta}{2\pi}}	\prod_{\lambda_l\in\mathcal{Z}_\xi^-}\abs{\lambda_{l}}^{2\alpha_{l}}&\text{2nd transition zone},
		\end{cases}$$
	 $q^{\mathcal{Z}_\xi}_n(t)$ is the solitonic part  defined as \eqref{qzxi},
		and $q^{P}(\xi)$ is the interaction between the solitonic part and the solution of the Painlev\'e II equation, given by
		\begin{align}
			&q^{P}(\xi)=\begin{cases}
				  q^{\mathcal{Z},P}(-\ii)&\text{1st transition zone},\\
				q^{\mathcal{Z},P}(\ii) &\text{2nd transition zone},
			\end{cases} \nonumber
		\end{align}
		with
		\begin{align}
		 q^{\mathcal{Z},P}(\lambda)=& \ii  \int_{s}^{\infty} v^2(\varsigma) \ddddd \varsigma \left( M^{(\mathcal{Z}_\xi)}_{1,1}( \lambda) \overline{M^{(\mathcal{Z}_\xi)}_{2,1}( \lambda)}- M^{(\mathcal{Z}_\xi)}_{1,2}( \lambda)\overline{ M^{(\mathcal{Z}_\xi)}_{2,2}( \lambda)} \right) \nonumber\\
		& -v(s)e^{\ii \varphi}  M^{(\mathcal{Z}_\xi)}_{1,2}( \lambda)\overline{M^{(\mathcal{Z}_\xi)}_{2,1}( \lambda)}
		  -v(s)e^{-\ii \varphi}  M^{(\mathcal{Z}_\xi)}_{1,1}( \lambda)\overline{M^{(\mathcal{Z}_\xi)}_{2,2}( \lambda)},\nonumber\\
		\varphi=& \begin{cases}
\ii	\phi(-\ii ) +\arg  r(-\ii) + 2  \alpha_{l} \sum_{\lambda_{l}\in\mathcal{Z}_\xi^-} \arg(\ii+\lambda_{l})	&\text{1st transition zone},\\
			-\ii	\phi(\ii ) -\arg  r(\ii) - 2 \alpha_{l} \sum_{\lambda_{l}\in\mathcal{Z}_\xi^-} \arg(\ii-\lambda_{l}) +\oint_{\Sigma} \frac{\ln(1+\abs{r(\varsigma)}^2)}{\pi\varsigma} \ddddd \varsigma&\text{2nd transition zone}, 
		\end{cases} \nonumber
		\end{align}
		and $\phi(\lambda)$ being  the phase function in \eqref{e2.18}.
		Here, $M^{(\mathcal{Z}_\xi)}(\lambda)$ is defined as \eqref{e5.43}. % and  $M^{(\mathcal{Z}_\xi)}_{i,j}, \; i,j=1,2,$ denotes the $(i,j)$ elements of $M^{(\mathcal{Z}_\xi)}$
		In addition,
		$v(s)$ is the unique solution of the Painlev\'e II equation \eqref{p23} characterized by that as $s \to + \infty$,
		\begin{align}
			v(s) \sim \begin{cases}
				|r(-\ii)| \text{Ai}(s) &\text{1st transition zone},\\
				 |r(\ii)|\text{Ai}(s)&\text{2nd transition zone},
			\end{cases}
		\end{align}
		where $\text{Ai}(s)$ is the classical Airy function.

	\end{itemize}
	
\end{theorem}

%The Painlev\'e II transcendents in \eqref{qpnt} are real-valued, nonsingular on the real line and behave like $c \text{Ai}(s), -1\le c\le 1$, for large positive $s$. The one-parameter family of solutions are known as the Hastings-Mc

The Painlev\'e  function $v(s)$ belong to the one-parameter family of the second Painlev\'e transcendents first singled out by Segur and Ablowitz \cite{AS}. It is worthwhile to see that the Painlevé transcendents and their higher order
analogues play important roles in asymptotic studies of intergrable systems.
Among others,  we mention their appearances in the long-time asymptotics of Korteweg-de Vries, modified Korteweg-de Vries equations and its hierarchy \cite{Charlier2020,deift1993steepest, hz}, in the defocusing NLS equation under a nonzero background \cite{wang2023defocusing}, in the Camassa-Holm and modified Camassa-Holm equation \cite{mis,xu2024transient}, and so on.
Moreover,  it also appear in the case of the small dispersion limit of nonlinear integrable PDEs \cite{bertola2010universality, bertola2013universality, claeys2009universality, claeys2010painleve, lu2022universality}.

\subsection{Plan of the proof}

This article is organized as follows.
In Section \ref{s2}, we present the global well-posedness of the focusing AL system in the discrete weight $\ell^2$ space.
In Section \ref{s3}, we introduce the direct scattering and construct RH problem \ref{r1} for the initial-value problem.
In Section \ref{s4}, we investigate the solvability of RH problem \ref{r1}, by transforming it into a pure RH problem \ref{r2.7} with jump contour consisting of three circles,
which is uniquely solved by applying Fredholm theory.
In Section \ref{s5}, we investigate the long-time asymptotics in sector II.
In Section \ref{s6}, we consider the  long-time asymptotics in sectors I and III.
In Section \ref{s7}, we concentrate on the long-time asymptotics in 1st and 2nd transition zones.

List some notation exploited in this paper below:
\begin{itemize}
	\item The right-shifting operator $E$: $Ef(n)=f(n+1)$, the left-shifting operator $E^{-1}$: $E^{-1}f(n)=f(n-1)$.
	\item $\ell^1$: $\left\{\{a_n\}_{n=-\infty}^\infty:\sum_{-\infty}^\infty\lvert a_n\rvert<\infty\right\}$, where the norm reads $\|a_n\|_1=\sum_{-\infty}^\infty\lvert a_n\rvert$.
	\item $\ell^{k,1}$: $\left\{\{a_n\}_{n=-\infty}^\infty:\sum_{-\infty}^\infty(1+n^2)\lvert a_n\rvert^k<\infty\right\}$, $k=1,2$, where $\|a_n\|_{k,1}=\left(\sum_{-\infty}^\infty(1+n^2)\lvert a_n\rvert^k\right)^\frac{1}{k}$.
	\item $H^1$: $\left\{f(z)\in L^2(\Sigma):\partial_\theta f(e^{\ii\theta})\in L^2([0,2\pi])\right\}$, where $\Sigma=\{e^{\ii\theta}: \theta\in[0.2\pi]\}$ is the unit circle at the origin and the norm is $\norm{f}_{H^1}=\norm{f}_{L^2(\Sigma)}+\norm{\partial_\theta f}_{L^2(\Sigma)}$.
	\item The Pauli matrices:
	\begin{align}
		\sigma_1=\left(\begin{matrix}
			0&1\\1&0
		\end{matrix}\right),\quad\sigma_2\left(\begin{matrix}
			0&-\ii\\\ii&0
		\end{matrix}\right),\quad\sigma_3=\left(\begin{matrix}
			1&0\\0&-1
		\end{matrix}\right).
	\end{align}
	\item The phase function on $\lambda\in\mathbb{C}$:
	\begin{align}\label{e2.18}
		\phi(\lambda)=\phi(\lambda,n,t)=-\ii t(\lambda+\lambda^{-1}-2)+n\ln\lambda.
	\end{align}
	\item The discrete spectrum $\mathcal{Z}\cup\overline{\mathcal{Z}^{-1}}$:
	\begin{align}
		\mathcal{Z}=\{\lambda_1,\dots,\lambda_{l_0}\},
	\end{align}
	and
	\begin{align}
		\mathcal{Z}_\xi^-=\mathcal{Z}\cap\{\re\phi>0\},\quad\mathcal{Z}_\xi=\mathcal{Z}\cap\{\re\phi=0\}.
	\end{align}
	\item $V^\dagger$: Hermitian transpose of matrix $V$.
	
\end{itemize}

Our analysis exploits the Lax representation:
\begin{align}\label{e2}
	EX&=AX,\quad \frac{\ddddd X}{\ddddd t}=BX,\\
	A&=z^{\sigma_3}+Q,\quad B=\frac{\sigma_3}{\ii}\left(\frac{(z-z^{-1})^2}{2}-Q(E^{-1}Q)+z^{\sigma_3}Q-
	(E^{-1}Q)z^{\sigma_3}\right),
\end{align}
where $Q$ is a skew-symmetry matrix
\begin{align}
	Q=Q(n)=\left(\begin{matrix}
		0&q_n\\-\overline{q_n}&0
	\end{matrix}\right).
\end{align}
That a smooth function solves the initial-value problem (\ref{e1}) is equivalent to the compatibility condition:
\begin{align*}
	\frac{\ddddd A}{\ddddd t}=(EB)A-AB,
\end{align*}
which is proved to be the focusing AL system.
Introducing
\begin{align}\label{e1.14}
	c_n(t)=\prod_{k=n}^{\infty}(1+\abs{q_k(t)}^2),
\end{align}
it is readily seen that $c_{-\infty}=\lim_{n\to-\infty}c_n(t)$ is a conserved quantity.
By direct computations,
\begin{align}\label{e2.6}
	\frac{\ddddd}{\ddddd t} c_{-\infty}&=\frac{\ddddd}{\ddddd t}\prod_{n\in\mathbb{Z}}(1+\abs{q_n(t)}^2)=c_{-\infty}\sum_{n\in\mathbb{Z}}\frac{\frac{\ddddd}{\ddddd t}q_n(t)\overline{q_n(t)}+q_n(t)\overline{\frac{\ddddd}{\ddddd t}q_n(t)}}{1+\abs{q_n(t)}^2}
\end{align}
and substituting \eqref{e1} into \eqref{e2.6}, we obtain that
\begin{align*}
	\frac{\ddddd}{\ddddd t} c_{-\infty}\equiv 0,
\end{align*}
that is, $c_{-\infty}$ is independent on $t$.
Set
\begin{align}\label{e1.16}
	c_{-\infty}=c_{-\infty}(t)=\prod_{n\in\mathbb{Z}}(1+\abs{q_n(0)}^2).
\end{align}
Furthermorr, it is easily seen that for any $n\in\mathbb{Z}$ and $q_n(0)\in\mathbb{R}$,
\begin{align}
	0<\ln(1+\abs{q_n(0)}^2)<\abs{q_n(0)}^2,
\end{align}
and thus under the assumption of $q(0)$ in \eqref{e1}, it follows that $c_{-\infty}$ is dominated by $e^{\norm{q(0)}_{2}}$ and therefore a positive number.

\section{Global well-posedness}\label{s2}
In this section, we consider the solution of the initial-value problem \eqref{e1}.
We assert that when the initial potential $q(0)$ belongs to a weighted $\ell^2$ space, the conserved quantity $c_{-\infty}$ is a positive constant, and the solution is globally well-posed as shown in the proposition below.
\begin{proposition}\label{p2.1}
	When $q(0)\in \ell^{2,1}$, the initial-value problem \eqref{e1} admits a unique solution $q(t)\in\ell^{2,1}$ for $t\in\mathbb{R}$, and the solution is Lipschitz continuous about $t$.
\end{proposition}
\begin{proof}
	Recalling that $c_{-\infty}$ is constant and utilizing this fact, we solve the initial-value problem by Gronwell's equality.

	Since $c_{-\infty}$ is finite, we start by discussing the solution of \eqref{e1} in the ball:
	\begin{align*}
		B(q(0),c_{-\infty}):=\{a\in \ell^\infty:\norm{a-q(0)}_\infty<c_{-\infty}\}.
	\end{align*}
	By the definition of $c_{-\infty}$, it is easy to check that the initial potential satisfies $\norm{q(0)}_\infty=\sup_{n\in\mathbb{Z}}\abs{q_n(0)}<c_{-\infty}$.
	Since the right-hand side of \eqref{e1} is Lipschitz continuous, by a standard argument about ODEs in a Banach space, there exists $t_1>0$ such that for $t\in(-t_1,t_1)$, \eqref{e1} admits a unique solution in $B(q(0),c_{-\infty})$, and $t_1$ is determined by the conserved quantity $c_{-\infty}$.
	Moreover, noticing that $\norm{q(-\frac{t_1}{2})}_{\infty}$ and $\norm{q(\frac{t_1}{2})}_{\infty}$ are also less then $c_{-\infty}$, we extend the solution to $t\in(-\frac{t_1}{2},\frac{3t_1}{2})\cup(-\frac{3t_1}{2},\frac{t_1}{2})=(-\frac{3t_1}{2},\frac{3t_1}{2})$ by investigating the solution of \eqref{e1} in the ball $B(q(\frac{t_1}{2}),c_{-\infty})$ and $B(q(-\frac{t_1}{2}),c_{-\infty})$, respectively;
	repeating the procedure, we uniquely solve the solution $q(t)$ in $B(0,c_{-\infty})$ for $t\in\mathbb{R}$.
	
	Based on the unique solvability of \eqref{e1} in the ball $B(0,c_{-\infty})$, we come to solving \eqref{e1} in $\ell^{2,1}$, that is, $q(t)$ does not blow up in $\ell^{2,1}$ for $t\in\mathbb{R}$.
	By \eqref{e1},
	\begin{align*}
		\abs{\frac{\ddddd}{\ddddd t}\abs{q_n(t)}^2}\lesssim\abs{q_n(t)}^2+\abs{q_{n-1}(t)}^2+\abs{q_{n+1}(t)}^2.
	\end{align*}
	Summing up the both sides in the above inequality, we obtain that there exists a positive constant $C$ such that
	\begin{align}\label{e2.7}
		\norm{\frac{\ddddd}{\ddddd t}\abs{q(t)}^2}_{1,1}\le C\norm{q(t)}_{2,1}^2.
	\end{align}
	Integrate both sides of \eqref{e2.7} on the interval: $[0,t]$ for $t>0$ (or $[t,0]$ for $t<0$), it is readily seen that
	\begin{align}
		\norm{q(t)}_{2,1}-\norm{q(0)}_{2,1}\le C\int_{0}^t\norm{q(\tau)}_{2,1}\ddddd\tau,
	\end{align}
	therefore by Gronwell's inequality, $\norm{q(t)}_{2,1}\le\norm{q(0)}_{2,1}e^{Ct}$ that never blows up.
	Additionally, in the similar way, for any $\epsilon>0$, we obtain that
	\begin{align}
		\norm{q(t\pm\epsilon)-q(t)}_{2,1}\lesssim\pm\int_{0}^{\pm\epsilon}\norm{q(t+\tau)}_{2,1}\ddddd\tau,
	\end{align}
	that is, the solution $q(t)$ is Lipschitz continuous in $\ell^{2,1}$ about $t$.
	This proves the results.
\end{proof}

\section{Riemann-Hilbert problem}\label{s3}

In \cite{Ablowitz1975nonlinear,ablowitz1976nonlinear,ablowitz2004discrete}, the authors have presented the inverse scattering transform for the focusing AL system for the fast decaying initial potential.
In this section, we further study the scenario when the initial potential belongs to $\ell^{2,1}$, which refers to some results in \cite{Ablowitz1975nonlinear,ablowitz1976nonlinear,ablowitz2004discrete}, including the results for the Jost solution and the reflection coefficient, and based on these, we construct RH problem \ref{r1}.
Here, the jump contour is the unit circle $\Sigma=\{ \lambda:\abs{\lambda}=1\}$ oriented clockwise.
In the RH problem, we give the residue condition at the $\alpha_l$-order pole $\lambda_l$, $\alpha\in\mathbb{N}^+$.
Since the initial potential is in $l^{2,1}$, we further prove that the reflection coefficient $r(\lambda)$ belongs to a Hilbert space $H^1$.
Finally, by the RH problem, we obtain the reconstruction formula for the initial-value problem \eqref{e1}:
\begin{align}\label{e2.46}
	q_n(t)=[M(0,n+1,t)]_{1,2}.
\end{align}

\subsection{Jost solutions and modified Jost solutions}

Denote  $X^\pm(z,n,t)$ as the Jost solutions for the Lax pair \eqref{e2}, which satisfy that as $n\to\pm\infty$,
\begin{align*}
	X^\pm(z,n,t)\sim z^{n\sigma_3}e^{-\frac{\ii}{2}(z-z^{-1})^2t\sigma_3},
\end{align*}
and $S(z)$ as the scattering matrix such that
\begin{align}\label{e4}
	X^-(z,n,t)=X^+(z,n,t)S(z).
\end{align}
Since the matrix $Q(n)$ is skew-Hermitian, it follows that the Jost solutions admit the symmetric property:
\begin{align}\label{e2.10}
	X^\pm(z,n,t)=\overline{\sigma_2X^\pm(\bar z^{-1},n,t)\sigma_2}.
\end{align}
Denote $\lambda=z^2$, and naturally introduce the modified Jost solutions
\begin{align}\label{e5}
	Y^\pm(\lambda)=Y^\pm(\lambda,n,t)=\left(\begin{matrix}
		1&0\\0&z
	\end{matrix}\right)X^\pm(z,n,t)z^{-n\sigma_3}e^{\frac{\ii t}{2}(z-z^{-1})^2\sigma_3}\left(\begin{matrix}
		1&0\\0&z^{-1}
	\end{matrix}\right),
\end{align}
which are single-valued functions by substituting \eqref{e5} into \eqref{e2} and admit the symmetric property
\begin{align}\label{e3.6}
	Y^\pm(\lambda,n,t)=\overline{\sigma_2Y^\pm(\bar\lambda^{-1},n,t)\sigma_2}
\end{align}
in view of \eqref{e2.10}.
Recalling the assumption that $q(0)\in\ell^{2,1}\in\ell^1$, as shown in \cite{Ablowitz1975nonlinear,ablowitz1976nonlinear,ablowitz2004discrete}, we obtain that the matrix $(Y^+_1,Y^-_2)(\lambda)$ is holomorphic on $D_-=\{|\lambda|<1\}$ and continuously extended to $D_-\cup\Sigma$, $\Sigma=\{|\lambda|=1\}$, while $(Y^-_1,Y^+_2)(\lambda)$ is holomorphic on $D_+=\{|\lambda|>1\}$ and continuously extended to $D_+\cup\Sigma$.
Moreover, the modified Jost solutions satisfy asymptotic properties:
\begin{subequations}\label{e2.14}
	\begin{align}
		&(Y^+_1,Y^-_2)(\lambda,n,t)\sim\left(\begin{matrix}
			c_n^{-1}(t)+O(\lambda)&q_{n-1}(t)+O(\lambda)\\c_n^{-1}(t)\overline{q_n(t)}\lambda+O(\lambda^2)&1+O(\lambda)
		\end{matrix}\right),\quad\lambda\to0,\label{e2.14a}\\
		&(Y^-_1,Y^+_2)(\lambda,n,t)\sim\left(\begin{matrix}
			1+O(\frac{1}{\lambda})&-c_n^{-1}(t)q_n(t)\frac{1}{\lambda}+O(\frac{1}{\lambda^2})\\-\overline{q_{n-1}(t)}+O(\frac{1}{\lambda})&c_n^{-1}(t)+O(\frac{1}{\lambda})
		\end{matrix}\right),\quad\lambda\to\infty.
	\end{align}
\end{subequations}

\subsection{Scattering coefficients}

Introduce scattering coefficients,
\begin{align}\label{e2.15}
	a(\lambda)=c_n(t)\det(Y^-_1,Y^+_2)(\lambda,n,t),\quad b(\lambda)=e^{\phi(\lambda,n,t)}c_n(t)\det(Y^+_1,Y^-_1)(\lambda,n,t).
\end{align}
and the reflection coefficient
\begin{align}\label{e2.20}
r(\lambda)=\frac{b(\lambda)}{a(\lambda)},
\end{align}
and this part aims to proving that $r(\lambda)\in H^1$.

Factually,  in view of (\ref{e4}), \eqref{e5}, and \eqref{e2.15}, it follows that
\begin{align}\label{e2.16}
	a(\lambda)=S_{1,1}(z),\quad b(\lambda)=zS_{2,1}(z).
\end{align}
Since $Y^\pm(\lambda)$ is single-valued, we see that $a(\lambda)$, $b(\lambda)$ are also single-valued functions of $\lambda$.
By the analyticity of $Y^\pm(\lambda)$, it is easily seen that $a(\lambda)$ is holomorphic on $D_+$ and continuously extended to $D_+\cup\Sigma$;
in addition, $b(\lambda)$ is continuous on $\Sigma$.
By the symmetric property \eqref{e3.6} and the definition (\ref{e2.15}), it is readily seen that
\begin{align}
	&\breve{a}(\lambda):=\overline{a(\bar\lambda^{-1})}=c_n(t)\det(Y^+_1,Y^-_2)(\lambda,n,t),\label{e2.17}
\end{align}
and that $\breve{a}(\lambda)$ is holomorphic for $\lambda\in D_-$.
Furthermore, utilizing the asymptotic property \eqref{e2.14}, we obtain that
\begin{align*}
	&\lim_{\lambda\to\infty}a(\lambda)=1,\quad \breve{a}(0)=1.
\end{align*}

The assertion that when $q(0)\in \ell^{2,1}$, the corresponding reflection coefficient $r$ belongs to $H^1$ is checked in the following proposition.
\begin{proposition}\label{p2.2}
	Suppose that $q(0)\in \ell^{2,1}$ and there is no spectral singularity, that is, $a(\lambda)$ admits no zero on $\Sigma$.
	Let $r(\lambda)$ be the corresponding reflection coefficient as defined in \eqref{e2.20}.
	We assert that $r\in H^1$.
\end{proposition}
Before we prove Proposition \ref{p2.2}, transform the modified Jost solution for $t=0$ into
\begin{align}\label{e2.3}
	\tilde Y^\pm=\tilde Y^\pm(\lambda,n)=\lambda^{-n\sigma_3/2}Y^{\pm}(\lambda,n,0)\lambda^{n\sigma_3/2},
\end{align}
and it is a little abuse of notation that we simply denote $q_n=q_n(0)$ in the proof below.
Rewrite $a(\lambda)$, $b(\lambda)$ by substituting \eqref{e2.3} into \eqref{e2.15},
\begin{align}\label{e2.25}
	a(\lambda)=c_n(0)\det(\tilde Y^-_1,\tilde Y^+_2)(\lambda,n),\quad b(\lambda)=c_n(0)\det(\tilde Y^+_1,\tilde Y^-_1)(\lambda,n).
\end{align}
and it is readily seen that $\tilde Y^\pm$ satisfy the following Volterra summation equation:
\begin{align}\label{e2.4}
	\tilde Y^{\pm}=I+\mathcal{T}^\pm \tilde Y^{\pm},
\end{align}
where the linear integral operators $\mathcal{T}^\pm$ are defined for any $2\times1$ matrix function $Y=Y(n)$ as
\begin{subequations}\label{e2.24}
	\begin{align}
		&\mathcal{T}^{-}Y(\lambda,n)=\sum_{k=-\infty}^{n-1}\left(\begin{matrix}
			0&q_k\\-\overline{q_k}&0
		\end{matrix}\right)\lambda^{(k+1)\sigma_3}Y(k),\\
		&\mathcal{T}^{+}Y(\lambda,n)=-\sum_{k=n}^{+\infty}\left(\begin{matrix}
			0&q_k\\-\overline{q_k}&0
		\end{matrix}\right)\lambda^{(k+1)\sigma_3}Y(k).
	\end{align}
\end{subequations}
Recalling that $q\in\ell^{2,1}$, we obtain that the Volterra summation equations \eqref{e2.4} are uniquely solved by
\begin{align}\label{e2.2}
	\tilde Y^{\pm}=\sum_{k=0}^{+\infty}(\mathcal{T}^\pm)^kI.
\end{align}
In addition, by \eqref{e2.24}, it is readily seen that $\mathcal{T}^\pm$ admits the symmetry on the unit circle
\begin{align}\label{e2.27}
	\sigma_2\overline{\mathcal{T}^\pm \tilde Y}\sigma_2=\mathcal{T}^\pm(\overline{\sigma_2\tilde Y\sigma_2}).
\end{align}
\begin{proof}[Proof of Proposition \ref{p2.2}]
	Since $a(\lambda)$ does not vanish on the jump contour $\Sigma$, if we prove that
	\begin{align}\label{e2.23}
		a,b\in H^1,
	\end{align}
	then, the results follow.
	By \eqref{e2.25}, to prove \eqref{e2.23}, we only have to verify that
	\begin{align}\label{e2.29}
		\tilde Y^\pm(\cdot,0)\in H^1.
	\end{align}
	
	By direct computation, we obtain that for any $m\in\mathbb{Z}^+$,
	\begin{subequations}\label{e2.30}
		\begin{align}
			&(\mathcal{T}^-)^{2m-1}\ylk(\lambda,n)=(-1)^m\sum_{-\infty<k_{2l-1}<\cdots<k_1<n}\left(\begin{matrix}
				0\\\prod_{\text{Even }j}q_{k_j}\overline{\prod_{\text{Odd }j}q_{k_j}}\lambda^{A_{2m-1}}
			\end{matrix}\right),\label{e2.30a}\\
			&(\mathcal{T}^-)^{2m}\ylk(\lambda,n)=(-1)^m\sum_{-\infty<k_{2m}<\cdots<k_1<n}\left(\begin{matrix}
				\prod_{\text{Odd }j}q_{k_j}\overline{\prod_{\text{Even }j}q_{k_j}}\lambda^{1-A_{2m}}\\0
			\end{matrix}\right),\\
			&(\mathcal{T}^+)^{2m-1}\ylk(\lambda,n)=(-1)^m\sum_{n\le k_{1}\le\cdots\le k_{2m-1}<+\infty}\left(\begin{matrix}
				0\\\prod_{\text{Even }j}q_{k_j}\overline{\prod_{\text{Odd }j}q_{k_j}}\lambda^{A_{2m-1}}
			\end{matrix}\right),\\
			&(\mathcal{T}^+)^{2m}\ylk(\lambda,n)=(-1)^m\sum_{n\le k_{1}\le\cdots\le k_{2m}<+\infty}\left(\begin{matrix}
				\prod_{\text{Odd }j}q_{k_j}\overline{\prod_{\text{Even }j}q_{k_j}}\lambda^{1-A_{2m}}\\0
			\end{matrix}\right),
		\end{align}
	\end{subequations}
	where $A_m=\sum_{j=1}^m(-1)^{j-1}k_{j}+1$.
	Denoting $N(m,\lambda,n)$ as the non-zero entry of $(\mathcal{T}^-)^{2m-1}\ylk(\lambda,n)$, we obtain that for any $g(\lambda)\in L^2(\Sigma)$, $\norm{g}_{L^2(\Sigma)}=1$,
	\begin{equation}\label{e2.31}
		\begin{split}
			&\abs{\oint_{\Sigma}N(m,\lambda,n)g(\lambda)\frac{\ddddd\lambda}{2\pi\ii}}=\abs{\sum_{-\infty<k_{2m-1}<\cdots<k_1<n}(\prod_{\text{Even }j}q_{k_j})\overline{(\prod_{\text{Odd }j}q_{k_j})}\hat g(A_{2m-1})}\\
			&\quad\quad\le\sum_{-\infty<k_{2m-1}<\cdots<k_1<n}\abs{q_{k_1}\cdots q_{k_{2m-1}}\hat g(A_{2m-1})}\le\frac{\norm{q}_{1}^{2m-2}}{(2m-2)!}\norm{q}_2\norm{\hat g}_2\le \left(1+\frac{\pi^2}{3}\right)^{m-1}\frac{\norm{q}^{2m-1}_{2,1}}{(2m-2)!},
		\end{split}
	\end{equation}
	where $\hat g(n)$ is the Fourier coefficient of $g$, thus, by Riesz representation theorem,
	\begin{align}\label{e2.32}
		\norm{N(m,\cdot,n)}_{L^2(\Sigma)}\le\left(1+\frac{\pi^2}{3}\right)^{m-1}\frac{\norm{q}^{2m-1}_{2,1}}{(2m-2)!}.
	\end{align}
	We also have that for $\theta=\arg\lambda\in[0,2\pi)$,
	\begin{equation}\label{e2.44}
		\begin{split}
			&\partial_\theta N=\ii \left(N+\sum_{j=1}^{2m-1}(-1)^{m+j-1}N_j\right),\\
			&N_j=N_j(m,\lambda,n)=\sum_{-\infty<k_{2m-1}<\cdots<k_1<n}(\prod_{\text{Even }j'}q_{k_{j'}})\overline{(\prod_{\text{Odd }j'}q_{k_{j'}})}k_j\lambda^{A_{2m-1}}.
		\end{split}
	\end{equation}
	Using the same technique in \eqref{e2.31}, we have
	\begin{equation}\label{e2.43}
		\begin{split}
			\abs{\oint_{\Sigma}N_j(m,\lambda,n)g(\lambda)\frac{\ddddd\lambda}{2\pi\ii}}&\le \sum_{-\infty<k_{2m-1}<\cdots<k_1<n}\abs{(\prod_{j'\ne j}q_{k_{j'}})(k_jq_{k_j})\hat g(A_{2m-1})}\\
			&\le\frac{\norm{q}_{1}^{2m-1}}{(2m-2)!}\norm{q}_{2,1}\norm{\hat g}_2\le \left(1+\frac{\pi^2}{3}\right)^{m-1}\frac{\norm{q}^{2m-1}_{2,1}}{(2m-2)!},
		\end{split}
	\end{equation}
	thus, by \eqref{e2.44} and Riesz representation theorem, it is readily seen that
	\begin{align}\label{e2.35}
		\norm{\partial_\theta N(m,\cdot,n)}_{L^2(\Sigma)}\le 2m\left(1+\frac{\pi^2}{3}\right)^{m-1}\frac{\norm{q}^{2m-1}_{2,1}}{(2m-2)!}.
	\end{align}
	By \eqref{e2.32} and \eqref{e2.35}, it follows that
	\begin{align}
		\norm{N(m,\cdot,n)}_{H^1}\le(2m+1)\left(1+\frac{\pi^2}{3}\right)^{m-1}\frac{\norm{q}^{2m-1}_{2,1}}{(2m-2)!},
	\end{align}
	that is,
	\begin{align}
		\norm{(\mathcal{T}^-)^{2m-1}\ylk(\cdot,n)}_{H^1}\le(2m+1)\left(1+\frac{\pi^2}{3}\right)^{m-1}\frac{\norm{q}^{2m-1}_{2,1}}{(2m-2)!}.
	\end{align}
	Applying the technique for the estimate in \eqref{e2.30a}, we estimate the others in \eqref{e2.30} and obtain that
	\begin{align}\label{e2.37}
		\norm{(\mathcal{T}^\pm)^k\ylk(\cdot,n)}_{H^1}\le(k+2)\left(1+\frac{\pi^2}{3}\right)^{\frac{k-1}{2}}\frac{\norm{q}^{k}_{2,1}}{(k-1)!}.
	\end{align}
	Additionally, by \eqref{e2.27}, it is readily seen that for any $n\in\mathbb{Z}^+$,
	\begin{align}
		\sigma_2\overline{(\mathcal{T}^\pm)^kI}\sigma_2=\mathcal{T}^\pm(\sigma_2\overline{(\mathcal{T}^\pm)^{k-1}I}\sigma_2)=\cdots=(\mathcal{T}^\pm)^kI,
	\end{align}
	thus,
	\begin{align}
		(\mathcal{T}^\pm)^k\ly(\cdot,n)=\left(\begin{matrix}
			0&-1\\1&0
		\end{matrix}\right)\overline{(\mathcal{T}^\pm)^k\ylk(\cdot,n)}\in H^1,
	\end{align}
	and
	\begin{align}\label{e2.38}
		\norm{\mathcal{T}^\pm)^k\ly(\cdot,n)}_{H^1}\le(k+2)\left(1+\frac{\pi^2}{3}\right)^{\frac{k-1}{2}}\frac{\norm{q}^{k}_{2,1}}{(k-1)!}.
	\end{align}
	
	In view of \eqref{e2.2}, \eqref{e2.37}, and \eqref{e2.38}, applying dominated convergence theorem, we obtain \eqref{e2.29}.
	This proves the results.
\end{proof}

\subsection{Discrete spectrum}

Under the assumption of no spectral singularity, the discrete spectrum $\mathcal{Z}\cup\overline{\mathcal{Z}^{-1}}$ is finite, and the holomorphic function $a(\lambda)$ admits a zero at each point of $\mathcal{Z}$. Set the zero's order as $\alpha_{l}\in\mathbb{N}^+$ at $\lambda=\lambda_l$, $l=1,\dots,l_0$.
By the symmetric property \eqref{e2.17}, it is clear that $\breve{a}(\lambda)$ admits an $\alpha_{l}$-order zero at each $\lambda=\bar\lambda_{l}^{-1}\in\overline{\mathcal{Z}^{-1}}$, too.

Since $\lambda_{l}$ is the $\alpha_{l}$-order zero of $a(\lambda)$, by taking the $\alpha_l$-order derivative of $a(\lambda)$ in \eqref{e2.15}, it turns out that there exist constants: $\beta_{\lambda_{l},0},\beta_{\lambda_{l},1},\dots,\beta_{\lambda_{l},\alpha-1}$, such that $\beta_{\lambda_{l},0}\ne0$ and for any $\alpha=0,\dots,\alpha_{l}-1$,
\begin{align}\label{e19s}
	&\partial_\lambda^{\alpha}Y^-_1(\lambda_{l})=\sum_{\alpha'=0}^{\alpha}\left(\begin{matrix}
		\alpha\\\alpha'
	\end{matrix}\right)\partial_\lambda^{\alpha-\alpha'}Y^+_2(\lambda_{l})(\sum_{j=0}^{\alpha'}\left(\begin{matrix}
		\alpha'\\j
	\end{matrix}\right)\beta_{\lambda_{l},j}\partial_\lambda^{\alpha'-j}(e^{-\phi(\lambda_{l})})),
\end{align}
where  $\left(\begin{matrix}
	\alpha\\\alpha'
\end{matrix}\right)$ denotes the binomial coefficient.
Since $\breve a(\lambda)$ also admits an $\alpha_l$-order zero at $\lambda=\bar\lambda_l^{-1}\in\overline{\mathcal{Z}^{-1}}$, there exist constants $\beta_{\bar\lambda_{l}^{-1},0},\dots,\beta_{\bar\lambda_{l}^{-1},{\alpha_{l}-1}}$, such that
\begin{align}\label{e13}
	&\partial_\lambda^{\alpha}Y^-_2(\bar\lambda_{l}^{-1})=\sum_{\alpha'=0}^{\alpha}\left(\begin{matrix}
		\alpha\\\alpha'
	\end{matrix}\right)\partial_\lambda^{\alpha-\alpha'}Y^+_1(\bar\lambda_{l}^{-1})(\sum_{j=0}^{\alpha'}\left(\begin{matrix}
		\alpha'\\j
	\end{matrix}\right)\beta_{\bar\lambda_{l}^{-1},j}\partial_\lambda^{\alpha'-j}(e^{\phi(\bar\lambda_{l}^{-1})})),
\end{align}
especially, $\beta_{\bar\lambda_{l}^{-1},0}=-\overline{\beta_{\lambda_{l},0}}\ne0$.
Moreover, considering the symmetric property \eqref{e3.6} in \eqref{e19s} and \eqref{e13}, we have the symmetry for $\beta_{\lambda_l,j}$ and $\beta_{\bar\lambda_l^{-1},j}$, $j=1,...,\alpha_l-1$:
\begin{align}\label{e3.33}
	\beta_{\bar\lambda_l^{-1},j}=(-1)^{j+1}\sum_{j'=1}^{j}\left(\frac{j!}{j'!}\left(\begin{matrix}
		j-1\\j'-1
	\end{matrix}\right)\overline{\lambda_l^{j+j'}\beta_{\lambda_l,j'}}\right).
\end{align}

\subsection{Reconstruction formula}
In this part, we construct the Riemann-Hilbert problem and the reconstruction formula.
Set the $2\times2$ matrix-valued function
\begin{align}\label{e2.21}
	M(\lambda)=M(\lambda,n,t)=\begin{cases}
		\left(P\left(Y^+_1,\frac{Y^-_2}{\breve a}\right)\right)(\lambda,n,t),&\lambda\in D_-,\\
		\left(P\left(\frac{Y^-_1}{a},Y^+_2\right)\right)(\lambda,n,t),&\lambda\in D_+,
	\end{cases}
\end{align}
where $P=P(n,t)=\left(\begin{matrix}
	1&0\\c_n(t)\overline{q_{n-1}}(t)&c_n(t)
\end{matrix}\right)$.
We claim that $M(\lambda)$ solves RH problem \ref{r1}.
By the analyticity of $Y^\pm(\lambda)$, we obtain the analyticity of $M(\lambda)$.
By the asymptotic behavior \eqref{e2.14} of $Y^\pm(\lambda)$ at the infinity, we derive the normalization of $M(\lambda)$ at $\lambda\to\infty$.
Now, we focus on the residue condition.
Recalling that $\lambda=\lambda_{l}$ is an $\alpha_{l}$-order zero of $a(\lambda)$, it is naturally an $\alpha_{l}$-order pole of $M(\lambda,n,t)$.
If denote $\underset{\lambda=\lambda'}{\LL_{-\alpha}}M$ as the coefficient of $(\lambda-\lambda')^{-\alpha}$ in the Laurent's expansion of $M(\lambda)$ at $\lambda=\lambda'$, then by the argument of complex analysis, we obtain the formula to compute the coefficients
\begin{align}\label{e2.22}
	\underset{\lambda=\lambda_{l}}{\LL_{-\alpha}}M=\frac{1}{(\alpha_{l}-\alpha)!}\partial_\lambda^{\alpha_{l}-\alpha}[M(\lambda)(\lambda-\lambda_{l})^{\alpha_{l}}]\Big|_{\lambda=\lambda_{l}}.
\end{align}
Considering \eqref{e19s}, \eqref{e2.21}, and \eqref{e2.22}, we obtain the residue condition \eqref{e16}, and \eqref{e2.23b} is verified similarly.
The jump condition is the consequence of \eqref{e4}, \eqref{e5}, and \eqref{e2.21}.
\begin{rhp}\label{r1}
	\
	\begin{itemize}
		\item $M(\lambda)$ is holomorphic in $\mathbb{C}\setminus\Sigma$.
		\item As $\lambda\to\infty$, we have $M(\lambda)=I+\oo(\lambda^{-1})$.
		\item At the $\alpha_{l}$-order pole $\lambda=\lambda_{l}\in \mathcal{Z}$, we have
		\begin{subequations}
			\begin{align}\label{e16}
				\underset{\lambda=\lambda_{l}}{\LL_{-\alpha}}M=\sum_{j=0}^{\alpha_{l}-\alpha}\frac{\partial_\lambda^{\alpha_{l}-\alpha-j}M(\lambda_{l})}{(\alpha_{l}-\alpha-j)!}\left(\begin{matrix}
					0&0\\p_{\lambda_{l},j}e^{-\phi(\lambda_{l})}&0
				\end{matrix}\right),
			\end{align}
			where $\alpha=0,\dots,\alpha_{l}-1$ and $p_{\lambda_{l},j}$ is a polynomial of $(n,t)$ of order at most $j$
			\begin{align*}
				p_{\lambda_{l},j}
				=\sum_{j'=0}^{j}\sum_{j''=0}^{j'}\partial_\lambda^{j''}(\frac{(\cdot-\lambda_{l})^{\alpha_{l}}}{a})|_{\lambda=\lambda_{l}}\frac{\beta_{\lambda_{l},j'-j''}\partial_\lambda^{j-j'}\left(e^{-\phi}\right)(\lambda_{l})}{(j-j')!(j'-j'')!j''!e^{-\phi(\lambda_{l})}}.
			\end{align*}
			And at $\lambda=\bar\lambda_{l}^{-1}$, the residue condition is
			\begin{align}\label{e2.23b}
				\underset{\lambda=\bar\lambda_{l}^{-1}}{\LL_{-\alpha}}M=\sum_{j=0}^{\alpha_{l}-\alpha}\frac{\partial_\lambda^{\alpha_{l}-\alpha-j}M(\bar\lambda_{l}^{-1})}{(\alpha_{l}-\alpha-j)!}\left(\begin{matrix}
					0&p_{\bar\lambda_{l}^{-1},j}e^{\phi(\bar\lambda_{l}^{-1})}\\0&0
				\end{matrix}\right),
			\end{align}
			
		\end{subequations}
		where
		\begin{align*}
			p_{\bar\lambda_{l}^{-1},j}=\sum_{j'=0}^{j}\sum_{j''=0}^{j'}\partial_\lambda^{j''}(\frac{(\cdot-\bar\lambda_{l}^{-1})^{\alpha_{l}}}{\breve a})|_{\lambda=\bar\lambda_{l}^{-1}}\frac{\beta_{\bar\lambda_{l}^{-1},j'-j''}\partial_\lambda^{j-j'}\left(e^{\phi}\right)(\bar\lambda_{l}^{-1})}{(j-j')!(j'-j'')!j''!e^{\phi(\bar\lambda_{l}^{-1})}}
		\end{align*}
		\item If we set the boundary value of $M(\lambda)$ on $\lambda\in\Sigma$ as $M_\pm (\lambda)$, then
		\begin{align*}
			&M_+(\lambda)=M_-(\lambda)V(\lambda),\quad V(\lambda)=V(\lambda,n,t)=\left(\begin{matrix}
				1+\lvert r(\lambda)\rvert^2&\overline{r(\lambda)}e^{\phi(\lambda,n,t)}\\r(\lambda)e^{-\phi(\lambda,n,t)}&1
			\end{matrix}\right).
		\end{align*}
	\end{itemize}
\end{rhp}

By the asymptotic property \eqref{e2.14a} of $Y^\pm(\lambda)$ at $\lambda\to0$ and the formulization of $M(\lambda)$ in \eqref{e2.21}, we find that the solution of \eqref{e1} is recovered by the solution of RH problem \ref{r1}, which means the reconstruction formula \eqref{e2.46}.

\section{Fredholm alternative, pure RH problem and the solvability}\label{s4}

In this section, utilizing some triangular matrix function $\left(\begin{matrix}
	1&0\\fe^{-\phi}&1
\end{matrix}\right)$ and $\left(\begin{matrix}
	1&\breve fe^{\phi}\\0&1
\end{matrix}\right)$, we remove the poles in the original RH problem and transform it into a new RH problem without singularity.
In the new RH problem, we find that the jump matrix is self-adjoint, and as a consequence, using the Fredholm theory, we prove that the RH problem is uniquely solved.

\subsection{Remove the poles}
In this section, we construct some rational functions $f(\lambda)$, $\breve f(\lambda)$ are rational function such that the matrix-valued functions $M(\lambda)\left(\begin{matrix}
	1&0\\f(\lambda)e^{-\phi(\lambda)}&1
\end{matrix}\right)$, $M(\lambda)\left(\begin{matrix}
	1&\breve f(\lambda)e^{\phi(\lambda)}\\0&1
\end{matrix}\right)$ are analytic on $D_+$, $D_-\setminus\{0\}$, respectively.

Before proving the results for the region $D_+\cup D_-\setminus\{0\}$, we consider the scenario on the neighborhood of each $\lambda_{l}$ or $\bar\lambda_{l}^{-1}$ as the following proposition.

\begin{proposition}\label{p3}
There exist polynomials $f_{\lambda_{l}}(\lambda)$, $f_{\bar\lambda_{l}^{-1}}(\lambda)$ of $(\lambda-\lambda_{l})^{-1}$, $(\lambda-\bar\lambda_{l}^{-1})^{-1}$ with order not more than $\alpha_{l}$, respectively, such that for $l=1,\dots,l_0$ and $\alpha=1,\dots,\alpha_{l}$,
	\begin{align}\label{e14}
		\underset{\lambda=\lambda_{l}}{\LL_{-\alpha}}(M\left(\begin{matrix}
			1&0\\f_{\lambda_{l}}e^{-\phi}&1
		\end{matrix}\right))\equiv0, \quad \underset{\lambda=\bar\lambda_{l}^{-1}}{\LL_{-\alpha}}(M\left(\begin{matrix}
			1&f_{\bar\lambda_{l}^{-1}}e^{\phi}\\0&1
		\end{matrix}\right))\equiv0.
	\end{align}
\end{proposition}
\begin{proof}
	Set
	\begin{align}\label{e15}
		f_{\lambda_{l}}(\lambda)=\sum_{j=1}^{\alpha_{l}}f_{\lambda_{l},j}(\lambda-\lambda_{l})^{-j},
	\end{align}
	and rewrite the first equation in \eqref{e14},
	\begin{align}\label{e18}
		0\equiv\underset{\lambda=\lambda_{l}}{\LL_{-\alpha}}(M\left(\begin{matrix}
			1&0\\f_{\lambda_{l}}e^{-\phi}&1
		\end{matrix}\right))=\underset{\lambda=\lambda_{l}}{\LL_{-\alpha}}M+\underset{\lambda=\lambda_{l}}{\LL_{-\alpha}}(M\left(\begin{matrix}
			0&0\\f_{\lambda_{l}}e^{-\phi}&0
		\end{matrix}\right)).
	\end{align}
	Recalling that the second column of $M(\lambda)$ is holomorphic on $\lambda\in D_+$, we rewrite the second term on the right-hand side of \eqref{e18}:
	\begin{align}\label{e19}
		\underset{\lambda=\lambda_{l}}{\LL_{-\alpha}}(M\left(\begin{matrix}
			0&0\\f_{\lambda_{l}}e^{-\phi}&0
		\end{matrix}\right))=\sum_{j=0}^{\alpha_{l}-\alpha}\frac{\partial_\lambda^{\alpha_{l}-\alpha-j}M}{(\alpha_{l}-\alpha-j)!}\left(\begin{matrix}
			0&0\\\sum_{j'=0}^{j}\frac{f_{\lambda_{l},\alpha_{l}-j'}\partial_\lambda^{j-j'}e^{-\phi}}{(j-j')!}&0
		\end{matrix}\right)\rvert_{\lambda=\lambda_{l}}.
	\end{align}
	Substituting \eqref{e16} and \eqref{e19} into \eqref{e18}, we obtain that
	\begin{align}\label{e4.5}
		f_{\lambda_{l},j}=\sum_{j'=0}^{\alpha_{l}-j}\partial_\lambda^{j'}(\frac{(\lambda-\lambda_{l})^{\alpha_{l}}}{a(\lambda)})|_{\lambda=\lambda_{l}}\frac{\beta_{\lambda_{l},\alpha_{l}-j-j'}}{(\alpha_{l}-j-j')!j'!}
	\end{align}
	and \eqref{e15}-\eqref{e4.5} solves \eqref{e18}, that is,
	\begin{align}\label{e20}
		f_{\lambda_{l}}(\lambda)=\sum_{j=1}^{\alpha_{l}}\sum_{j'=0}^{\alpha_{l}-j}\partial_\lambda^{j'}(\frac{(\lambda-\lambda_{l})^{\alpha_{l}}}{a(\lambda)})|_{\lambda=\lambda_{l}}\frac{\beta_{\lambda_{l},\alpha_{l}-j-j'}(\lambda-\lambda_{l})^{-j}}{(\alpha_{l}-j-j')!j'!}.
	\end{align}
	We have proved the existence of $f_{\lambda_k}(\lambda)$, and the formulation of $f_{\bar\lambda_{l}^{-1}}(\lambda)$ is similar, that is,
	\begin{align}\label{e4.7}
		f_{\bar\lambda_{l}^{-1}}(\lambda)=\sum_{j=1}^{\alpha_{l}}f_{\bar\lambda_l^{-1},j}(\lambda-\bar\lambda_{l}^{-1})^{-j},\quad f_{\bar\lambda_l^{-1},j}=\sum_{j'=0}^{\alpha_{l}-j}\partial_\lambda^{j'}(\frac{(\lambda-\bar\lambda_{l}^{-1})^{\alpha_{l}}}{\breve a(\lambda)})|_{\lambda=\bar\lambda_{l}^{-1}}\frac{\beta_{\bar\lambda_{l}^{-1},\alpha_{l}-j-j'}}{(\alpha_{l}-j-j')!j'!}.
	\end{align}
	We complete the proof.
\end{proof}
As shown in Proposition \ref{p3}, we have proved that the pole of $M(\lambda)$ can be removed by the rational triangular function at each point of the discrete spectrum, which means that in RH problem \ref{r1}, the residue condition can be transformed into the jump.
Based on this, we further find a rational triangular function on $D_+$ or $D_-\setminus\{0\}$ to remove the poles on the discrete spectrum uniformly, which is helpful when constructing the pure RH problem.
To obtain the rational functions, we apply the theory of generalized Vandermonde matrix \cite{kalman1984generalized} on the following proposition.

\begin{proposition}\label{p4}
	Set $M(\lambda)$ as the solution of RH problem \ref{r1}.
	There exists rational functions $f(\lambda)$, $\breve f(\lambda)$ such that $M(\lambda)\left(\begin{matrix}
		1&0\\f(\lambda)e^{-\phi(\lambda)}&1
	\end{matrix}\right)$, $M(\lambda)\left(\begin{matrix}
		1&\breve f(\lambda)e^{\phi(\lambda)}\\0&1
	\end{matrix}\right)$ are holomorphic on $D_+$, $D_-\setminus\{0\}$, respectively.
	Moreover, there is a pair of $(f,\breve{f})$ satisfying
	\begin{align}
		\breve{f}(\lambda)=-\overline{f(\bar\lambda^{-1})}.
	\end{align}
\end{proposition}
\begin{proof}
	Notice that Proposition \ref{p3} is a special case of Proposition \ref{p4} with $l_0=1$. To prove the results, we introduce the generalized Vandermonde matrix \cite{kalman1984generalized}
	\begin{align}
		V(\lambda_1,\dots,\lambda_{l_0},\alpha_1,\dots,\alpha_{l_0})=\left(\begin{matrix}
			1&\lambda_1&\cdots&\lambda_1^{\alpha_1}&\cdots&\lambda_1^\tau\\
			0&1&\cdots&\alpha_1\lambda_1^{\alpha_1-1}&\cdots&\tau\lambda_1^{\tau-1}\\
			\cdots&\ddots&\ddots&\cdots&\cdots&\cdots\\
			0&\cdots&0&1&\cdots&\frac{\tau!}{\alpha_1!}\lambda_1^{\tau-\alpha_1}\\
			\cdots&\cdots&\cdots&\cdots&\cdots&\cdots\\
			1&\lambda_{l_0}&\cdots&\lambda_{l_0}^{\alpha_{l_0}}&\cdots&\lambda_{l_0}^\tau\\
			0&1&\cdots&\alpha_{l_0}\lambda_{l_0}^{\alpha_{l_0}-1}&\cdots&\tau\lambda_{l_0}^{\tau-1}\\
			\cdots&\ddots&\ddots&\cdots&\cdots&\cdots\\
			0&\cdots&0&1&\cdots&\frac{\tau!}{\alpha_{l_0}!}\lambda_{l_0}^{\tau-\alpha_{l_0}}
		\end{matrix}\right),
	\end{align}
	where $\tau=\sum_{l=1}^{l_0}\alpha_{l}$. The generalized Vandermonde matrix is invertible and its determinant is
	\begin{equation}
		\det V(\lambda_1,\dots,\lambda_{l_0},\alpha_1,\dots,\alpha_{l_0})=\prod_{j> k}(\lambda_{j}-\lambda_k)^{\alpha_j\alpha_k}.
	\end{equation}
	
	To construct $f(\lambda)$, it is sufficient to construct a polynomial $g(\lambda)$ of $\lambda$ such that
	\begin{align}\label{e2.40}
		f(\lambda)=g(\lambda)\prod_{l=1}^{l_0}f_{\lambda_{l}}(\lambda),\quad g(\lambda)=\sum_{j=0}^{\tau-1}g_j\lambda^j,
	\end{align}
	and for any $l=1,\dots,l_0$, $\alpha=1,\dots,\alpha_{l}$,
	\begin{align}\label{e2.41}
		\underset{\lambda=\lambda_{l}}{\LL_{-\alpha}}f(\lambda_{l})=f_{\lambda_{l},\alpha}.
	\end{align}
	Seeing the formulization of $f(\lambda)$ in \eqref{e2.40}, we learn that $f(\lambda)$ admits an $\alpha_{l}$-order pole at $\lambda=\lambda_{l}$.
	Calculating the Laurent's expansion of $f$ at $\lambda=\lambda_l$, we obtain that
	\begin{equation}\label{e25}
		\begin{split}
			\underset{\lambda=\lambda_{l}}{\LL_{-\alpha}}f&=\frac{1}{(\alpha_{l}-\alpha)!}\partial_\lambda^{\alpha_{l}-\alpha}((\cdot-\lambda_{l})^{-\alpha_{l}}f_{\lambda_{l}}g\prod_{l'\ne l}f_{\lambda_{l'}})(\lambda_{l})\\
			&=\sum_{j=0}^{\alpha_{l}-\alpha}\frac{1}{j!}f_{\lambda_{l},\alpha+j}\partial_\lambda^j(g\prod_{l'\ne l}f_{\lambda_{l'}})(\lambda_{l}).
		\end{split}
	\end{equation}
	By comparing \eqref{e25} and \eqref{e2.41}, it is sufficient to solve the following linear system with $\tau$ variables $g_0,\dots,g_{\tau-1}$:
	\begin{align}\label{e26}
		\partial_\lambda^j(g\prod_{l'\ne l}f_{\lambda_{l'}})(\lambda_{l})=\delta_{0,j},
		\quad\delta_{0,j}=\begin{cases}
			1&j=0\\
			0&j\ne0
		\end{cases},\quad j=0,\dots,\alpha_{l}-1.
	\end{align}
	Equivalently, we rewrite \eqref{e26} in the form of matrix:
	\begin{equation}\label{e27}
		\begin{split}
			&\mathrm{blkdiag}(\{J_{l}(\lambda_{l})\}_{l=1}^{l_0}) V(\lambda_1,\dots,\lambda_{l_0},\alpha_1,\dots,\alpha_{l_0})(g_0,\dots,g_{\tau-1})^{t}=(1,\overset{\alpha_1-1}{\overbrace{0,\dots,0}},\dots,1,\overset{\alpha_{l_0}-1}{\overbrace{0,\dots,0}})^t
		\end{split}
	\end{equation}
	where $\mathrm{blkdiag}(\cdot)$ denotes the block diagonal matrix, and $J_{l}$ is an $\alpha_{l}\times\alpha_{l}$ matrix function:
	\begin{align}
		J_{l}=\left(\begin{matrix}
			\prod_{l'\ne l}f_{\lambda_{l'}}&0&\cdots&0\\
			\partial_\lambda\prod_{l'\ne l}f_{\lambda_{l'}}&\prod_{l'\ne l}f_{\lambda_{l'}}&\ddots&\vdots\\
			\vdots&\ddots&\ddots&0\\
			\frac{\partial_\lambda^{\alpha_{l}}\prod_{l'\ne l}f_{\lambda_{l'}}}{\alpha_{l}!}&\cdots&\partial_\lambda\prod_{l'\ne l}f_{\lambda_{l'}}&\prod_{l'\ne l}f_{\lambda_{l'}}
		\end{matrix}\right).
	\end{align}
	We default $f_{\lambda_{l'}}(\lambda_{l})\ne0$ for all $l'\ne l$, or let $f_{\lambda_{l'}}$ plus a constant to ensure its non-zero property at $\lambda_{l}$;
	therefore, the block diagonal matrix in \eqref{e27} is invertible;
	in addition, by the fact that the generalized Vandermonde matrix is invertible, it is readily seen that the linear system \eqref{e27} is uniquely solved.
	We have confirmed the existence of $f$.
	
	We assert that the rational function
	\begin{align}\label{e4.12}
		\breve f(\lambda)=-\overline{f(\bar\lambda^{-1})}
	\end{align}
	comply with the assumption of the proposition.
	We first prove that for $\lambda_l\in\mathcal{Z}$, in the Laurent's series of
	\begin{align}\label{e4.19}
		\breve f_{\bar\lambda_l^{-1}}(\lambda):=-\overline{f_{\lambda_l}(\bar\lambda^{-1})},
	\end{align}
	at $\lambda=\bar\lambda_l^{-1}$, the coefficients of $\left(\lambda-\bar\lambda_l^{-1}\right)^{-j}$ consist with $f_{\bar\lambda_l^{-1},j}$, $j=1,\dots,\alpha_l$, that is,
	\begin{align}\label{e4.10}
		f_{\bar\lambda_l^{-1},j}=\frac{1}{(\alpha_l-j)!}\partial_\lambda^{\alpha_l-j}\left((\cdot-\bar\lambda_l^{-1})^{\alpha_l}\breve f_{\bar\lambda_l^{-1}}\right)(\bar\lambda_l^{-1}).
	\end{align}
	In view of \eqref{e20}, it follows that
	\begin{equation}\label{e4.18}
		\begin{split}
			&\frac{1}{(\alpha_l-j)!}\partial_\lambda^{\alpha_l-j}\left((\cdot-\bar\lambda_l^{-1})^{\alpha_l}\breve f_{\bar\lambda_l^{-1}}\right)(\bar\lambda_l^{-1})\\
			&\quad\quad=\sum_{j'=0}^{\alpha_l-j}\sum_{j''=0}^{\alpha_l-j-j'}\overline{\left((-1)^{\alpha_l+1-j'-j''}\left(\begin{matrix}
					\alpha_l-j'-j''\\j
				\end{matrix}\right)\frac{\lambda_l^{j'+j''-\alpha_l-j}\beta_{\lambda_l,j''}}{j''!j'!}\left(\partial_\lambda^{j'}\left(\frac{(\cdot-\lambda_l)}{a}\right)(\lambda_l)\right)\right)}.
		\end{split}
	\end{equation}
	In view of \eqref{e4.7} and \eqref{e3.33}, it is easily seen that
	\begin{equation}\label{e4.14}
		\begin{split}
			f_{\bar\lambda_l^{-1},j}=&\sum_{j'=0}^{\alpha_l-j-1}\sum_{j''=1}^{\alpha_l-j-j'}\left(\frac{(-1)^{\alpha_l-j-j'-1}}{j'!j''!}\left(\begin{matrix}
				\alpha_l-j-j'-1\\j''-1
			\end{matrix}\right)\overline{\lambda_l^{\alpha_l-j-j'+j''}\beta_{\lambda_l,j''}}\partial_\lambda^{j'}\left(\frac{(\cdot-\bar\lambda_l^{-1})}{\breve a}\right)(\bar\lambda_l^{-1})\right)\\
			&-\frac{1}{(\alpha-j)!}\overline{\beta_{\lambda_l,0}}\partial_\lambda^{\alpha-j}\left(\frac{(\cdot-\bar\lambda_l^{-1})}{\breve a}\right)(\bar\lambda_l^{-1}).
		\end{split}
	\end{equation}
	Since $\breve{a}(\lambda)=\overline{a(\bar\lambda^{-1})}$, it is readily seen that
	\begin{align}
		\partial_\lambda^{j}\left(\frac{(\cdot-\bar\lambda_l^{-1})}{\breve a}\right)(\bar\lambda_l^{-1})=\sum_{j'=0}^{j}\bar\lambda_l^{-2\alpha_l+j+j'}(-1)^{\alpha_l+j'}\frac{(\alpha_l-j')!}{(\alpha_l-j)!}\left(\begin{matrix}
			j\\j'
		\end{matrix}\right)\overline{\partial_\lambda^{j'}\left(\frac{(\cdot-\lambda_l)}{a}\right)(\lambda_l)}.
	\end{align}
	Then, by comparing the coefficients of $\beta_{\lambda_l,j}\partial_\lambda^{j'}\left(\frac{(\cdot-\lambda_l)}{a}\right)(\lambda_l)$ in \eqref{e4.18} and \eqref{e4.14}, it is equivalent to the identity
	\begin{align}
		\sum_{j'''=0}^{\alpha_l-j-j'-j''}(-1)^{\alpha_l-j-j'-j''-j'''}\left(\begin{matrix}
			\alpha-j'\\j'''
		\end{matrix}\right)\left(\begin{matrix}
			\alpha_l-j-j'-j'''-1\\j''-1
		\end{matrix}\right)=\left(\begin{matrix}
			\alpha_l-j'-j''\\j
		\end{matrix}\right)
	\end{align}
	which is the same as the well-acknowledged binomial identity% when we take $n=\alpha-j'$, $k=\alpha_l-j-j'-j''$, $j'=k-j'''$,
	\begin{align}
		\sum_{j'=0}^k(-1)^{j'}\left(\begin{matrix}
			n\\k-j'
		\end{matrix}\right)\left(\begin{matrix}
			j'+j''-1\\j''-1
		\end{matrix}\right)=\left(\begin{matrix}
			n-j''\\j
		\end{matrix}\right).
	\end{align}
	This proves \eqref{e4.10} and means that $\breve f_{\bar\lambda_l^{-1}}(\lambda)$ fulfills the requirements of $f_{\bar\lambda_l^{-1}}$ in Proposition \ref{p3}.
	Setting
	\begin{align}
		\breve{f}(\lambda)=\breve{g}(\lambda)\prod_{l=1}^{l_0}\breve{f}_{\bar\lambda_l^{-1}}(\lambda),\quad \breve{g}(\lambda)=(-1)^{l_0+1}\overline{g(\bar\lambda^{-1})},\lambda\in D_-,
	\end{align}
	it is follows that
	\begin{align}
		\breve{f}(\lambda)=-\overline{f(\bar\lambda^{-1})}.
	\end{align}
	By direct computation, it is readily seen that for any $l=1,\dots,l_0$ and $j\in\mathbb{N}^+$
	\begin{align}
		&(\breve{g}\prod_{l'\ne l}\breve{f})(\bar\lambda_l^{-1})=\overline{(g\prod_{l'\ne l}f)(\lambda_l)}=1,\\
		&\partial_\lambda^j(\breve{g}\prod_{l'\ne l}\breve{f})(\bar\lambda_l^{-1})=(-1)^j\sum_{j'=1}^{j}\frac{j!}{j'!}\left(\begin{matrix}
			j-1\\j'-1
		\end{matrix}\right)\lambda_l^{j+j'}\overline{\partial_\lambda^{j'}(g\prod_{l'\ne l}f)(\lambda_l)}=0,
	\end{align}
	which means that for any $\alpha=1,\dots,\alpha_l$,
	\begin{align}
		\underset{\lambda=\bar\lambda_{l}^{-1}}{\LL_{-\alpha}}f=\sum_{j=0}^{\alpha_{l}-\alpha}\frac{1}{j!}f_{\bar\lambda_{l}^{-1},\alpha+j}\partial_\lambda^j(\breve g\prod_{l'\ne l}\breve f_{\bar\lambda_{l'}^{-1}})(\bar\lambda_{l}^{-1})=f_{\bar\lambda_l^{-1},\alpha}.
	\end{align}
	These prove the results.
\end{proof}

\subsection{Solvability}
We have constructed the triangular matrix functions in Proposition \ref{p4} to remove poles of $M(\lambda)$, however, since $\phi(\lambda)$ admits an essential singularity at $\lambda\to\infty$, the newly obtained matrix function is not normalized at the infinity, i.e., $M(\lambda)\left(\begin{matrix}
	1&0\\f(\lambda)e^{-\phi(\lambda)}&1
\end{matrix}\right)$ diverges at $\lambda\to\infty$.
Moreover, because the essential singularity of $\phi(\lambda)$ at $\lambda=0$, $M(\lambda)\left(\begin{matrix}
	1&\breve f(\lambda)e^{\phi(\lambda)}\\0&1
\end{matrix}\right)$ also admits an essential singularity at $\lambda=0$.
To avoid these essential singularities, we choose $\tilde\Sigma=\Sigma\cup\tilde\Sigma_1\cup\Sigma_2$ depicted in FIGURE \ref{f1}, where $\tilde\Sigma_1:=\{\lambda\in\mathbb{C}:\abs{\lambda}=\sup_{l=1,\dots,l_0}\{\abs{\lambda_l}\}+1\}$ and $\tilde\Sigma_2=(\tilde\Sigma_1)^{-1}$.
Transforming $M(\lambda)$ into
\begin{align}\label{e3.23}
	\tilde M(\lambda):=\tilde M(\lambda,n,t)=M(\lambda)\begin{cases}
		\left(\begin{matrix}
			1&0\\f(\lambda)e^{-\phi(\lambda)}&1
		\end{matrix}\right)&\lambda\in\tilde D_+=\{\abs{\lambda}\in(1,\sup_{l=1,\dots,l_0}\{\abs{\lambda_l}\}+1)\},\\
		\left(\begin{matrix}
			1&\breve f(\lambda)e^{\phi(\lambda)}\\0&1
		\end{matrix}\right)&\lambda\in\tilde D_-=\{\abs{\lambda}^{-1}\in(1,\sup_{l=1,\dots,l_0}\{\abs{\lambda_l}\}+1)\},\\
		I &otherwise,
	\end{cases}
\end{align}
we obtain a matrix-valued function $\tilde M(\lambda)$ that is holomorphic on $\mathbb{C}\setminus  \tilde\Sigma $, where $\tilde\Sigma=\Sigma\cup\tilde\Sigma_1\cup\tilde\Sigma_2$, and it is easily seen that it admits the following RH problem.
\begin{rhp}\label{r2.7}
	\
	\begin{itemize}
		\item $\tilde M(\lambda)$ is analytic in $\mathbb{C}\setminus\tilde\Sigma$.
		\item As $\lambda\to\infty$, we have $\tilde M(\lambda)=I+\oo(\lambda^{-1})$.
		\item On $\lambda\in\tilde\Sigma$, we have
		\begin{align}
			&\tilde M_+(\lambda)=\tilde M_-(\lambda)\tilde V(\lambda),\quad\tilde V(\lambda)=\begin{cases}
				\left(\begin{matrix}
					1&0\\f(\lambda)e^{-\phi(\lambda)}&1
				\end{matrix}\right)&\lambda\in\tilde\Sigma_1,\\
				\left(\begin{matrix}
					1&-\breve f(\lambda)e^{\phi(\lambda)}\\0&1
				\end{matrix}\right)&\lambda\in\tilde\Sigma_2,\\
				\left(\begin{matrix}
					1+\abs{f(\lambda)+r(\lambda)}^2&\overline{(r(\lambda)+f(\lambda))}e^{\phi(\lambda)}\\(f(\lambda)+r(\lambda))e^{-\phi(\lambda)}&1
				\end{matrix}\right)&\lambda\in\Sigma.
			\end{cases}
		\end{align}

	\end{itemize}
\end{rhp}
In RH problem \ref{r2.7}, it is easily seen that the jump matrix admits an upper-lower triangular factorization for $\lambda\in\tilde\Sigma$:
\begin{align}
	\tilde V(\lambda)&=(I-w_-(\lambda))^{-1}(I+w_+(\lambda))\label{e3.24},\\
	w_-(\lambda)&=\begin{cases}
		\mathbf{0}&\lambda\in\tilde\Sigma_1,\\
		\left(\begin{matrix}
			0&\breve f(\lambda)e^{\phi(\lambda)}\\0&0
		\end{matrix}\right)&\lambda\in\tilde\Sigma_2,\\
		\left(\begin{matrix}
			0&-\overline{(r(\lambda)+f(\lambda))}e^{\phi(\lambda)}\\0&0
		\end{matrix}\right)&\lambda\in\Sigma,
	\end{cases}\label{e3.25}\\
	w_+(\lambda)&=\begin{cases}
		\left(\begin{matrix}
			0&0\\f(\lambda)e^{-\phi(\lambda)}&0
		\end{matrix}\right)&\lambda\in\tilde\Sigma_1,\\
		\mathbf{0}&\lambda\in\tilde\Sigma_2,\\
		\left(\begin{matrix}
			0&0\\(r(\lambda)+f(\lambda))e^{-\phi(\lambda)}&0
		\end{matrix}\right)&\lambda\in\Sigma.
	\end{cases}\label{e3.26}
\end{align}

Introduce the Cauchy-type integral operator:
\begin{align}
	\cc^\Gamma g(\lambda)=\oint_\Gamma \frac{g(\varsigma)}{\varsigma-\lambda}\frac{\ddddd \varsigma}{2\pi \ii},\quad g\in L^2(\Gamma),
\end{align}
where $\Gamma$ consists of finite oriented Jordan curves with finite intersection points.  %, in a slight abuse of notation, we denote $f|_{\sigma'}$ the restriction of $f$ on $\Sigma'$ or the zero extension of $f|_{\sigma'}$ on $\tilde\Sigma$, then we learn that
%\begin{align}\label{e3.28}
%	\cc^{\tilde\Sigma}f=\int_{\tilde\Sigma}\frac{f(s)}{
	%	\zeta-\cdot}\frac{\ddddd\zeta}{2\pi\ii}=\cc^{\tilde\Sigma_1}(f|_{\tilde\Sigma_1})+\cc^{\tilde\Sigma_2}(f|_{\tilde\Sigma_2})+\cc^{\Sigma}(f|_{\Sigma}).
%\end{align}
Denote $\cc_\pm^{\Gamma}g$ as the boundary values of $\cc^{\Gamma}g$.
It is well known that $\cc_\pm^{\Gamma}$ are bounded operators on $L^2(\Gamma)$.
Taking $\Gamma=\tilde\Sigma=\Sigma\cup\tilde\Sigma_1\cup\tilde\Sigma_2$ and $g\in L^2(\tilde\Sigma)$,
we define the bounded operator:
\begin{align}\label{e3.31}
	T^{\tilde\Sigma}f=\cc_-^{\tilde\Sigma}(\cc_+^{\tilde\Sigma}(fw)w_+)+\cc_+^{\tilde\Sigma}(\cc_-^{\tilde\Sigma}(fw)w_-),
\end{align}
where $w=w_++w_-$. By rationally approximating $w_+$, $w_-$ in the 1st, 2nd term on the right-hand side of \eqref{e3.31}, respectively, we obtain that both $\cc_-^{\tilde\Sigma}(\cc_+^{\tilde\Sigma}(\cdot w)w_+)$ and $\cc_+^{\tilde\Sigma}(\cc_-^{\tilde\Sigma}(\cdot w)w_-)$ are compact operators on $L^2(\tilde\Sigma)$.
Exactly, we simply rewrite \eqref{e3.31} as
\begin{align}
	T^{\tilde\Sigma}=(\cc_w)^2,
\end{align}
where
\begin{equation}\label{e3.22}
	\cc_wf=\cc_+^{\tilde\Sigma}(fw_-)+\cc_-^{\tilde\Sigma}(fw_+),
\end{equation}
and then by the compactness of $T^{\tilde\Sigma}$, $\identity-T^{\tilde\Sigma}=(\identity+\cc_w)(\identity-\cc_w)$ is Fredholm.
In addition, referring to Corollary \ref{c3.1} that the equation $\cc_w\mu=\mu$ admits no non-trivial solution, the Fredholm alternative is zero, and therefore, we obtain that both $\identity-T^{\tilde\Sigma}$ and $\identity-\cc_w$ are invertible with the property
\begin{align}
	(\identity-\cc_w)^{-1}=(\identity-T^{\tilde\Sigma})^{-1}(\identity+\cc_w).
\end{align}
Thus, RH problem \ref{r2.7} is uniquely solvable and the solution of RH problem \ref{r2.7} is
\begin{equation}
	\tilde M(\lambda)=I+\oint_{\tilde\Sigma}\frac{[(\identity-\cc_w)^{-1}Iw](\varsigma)}{\varsigma-\lambda}\frac{\ddddd \varsigma}{2\pi\ii},
\end{equation}
which is also well known as the Beals-Coifman solution.

\begin{figure}
	\centering
	\includegraphics[width=0.3\textwidth]{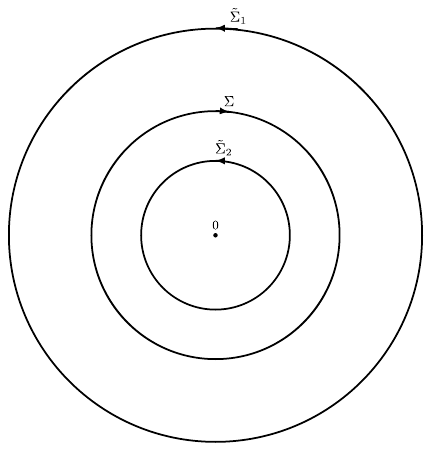}
	\caption{The jump contour $\tilde\Sigma$ consists of three circles centering at the origin, which is for RH problem \ref{r2.7}}\label{f1}
\end{figure}

\begin{lemma}[Vanishing Lemma]\label{l3.7}
	Set $\Gamma\subset\mathbb{C}$ as a curve consisting of finite Jordan curves and the unit circle with at most finite intersections, and assume that the union of the curves is invariant under the mapping $\lambda\mapsto\bar\lambda^{-1}$.
	Suppose a $2\times2$ matrix-valued function $V^{\Gamma}\in C^0(\Gamma)$  satisfies
	\begin{enumerate}
		\item $V^{\Gamma}(\lambda)$ admits the symmetric property on $\Gamma\setminus\Sigma$:
		\begin{align}\label{e3.32}
			V^{\Gamma}(\bar\lambda^{-1})=(V^{\Gamma}(\lambda))^\dagger.
		\end{align}
		\item For $\lambda\in\Sigma$,
		\begin{align}\label{e4.41}
			\re\ V^{\Gamma}(\lambda)>\mathbf{0}.
		\end{align}
		\item $V^{\Gamma}(\lambda)$ admits the factorization for $\lambda\in\Gamma$
		\begin{align}
			V^{\Gamma}(\lambda)=(I-w^{\Gamma}_-(\lambda))^{-1}(I+w^{\Gamma}_+(\lambda)),
		\end{align}
		where $w^\Gamma_\pm(\lambda)$ are some continuous functions.
	\end{enumerate}
	Define the Cauchy-type integral operator for $f\in L^2(\Gamma)$:
	\begin{align}
		\cc_{w^{\Gamma}}f=\cc^{\Gamma}_+(fw^{\Gamma}_-)+\cc^{\Gamma}_-(fw^{\Gamma}_+),
	\end{align}
	where
	\begin{align}
		w^{\Gamma}(\lambda)=w^{\Gamma}_+(\lambda)+w^{\Gamma}_-(\lambda),
	\end{align}
	then, we claim that there is no non-trivial solution $\mu\in L^2(\Gamma)$ for
	\begin{align}
		\cc_{w^\Gamma}\mu=\mu.
	\end{align}
	
\end{lemma}

\begin{proof}
	Suppose $\mu\in L^2(\Gamma)$ is a non-trivial solution.
	It is readily seen that $\cc^{\Gamma}(\mu w^\Gamma)(\lambda)$ is analytic outside $\Gamma$ and satisfies
	\begin{itemize}
		\item For $\lambda\in\Gamma$,
		\begin{align}\label{e3.55}
			\cc^{\Gamma}_+(\mu w^\Gamma)(\lambda)=\cc^{\Gamma}_-(\lambda)(\mu w^\Gamma)V^{\Gamma}(\lambda).
		\end{align}
		\item As $\lambda\to\infty$,
		\begin{align}\label{e3.56}
			\cc^{\Gamma}(\mu w^\Gamma)(\lambda)=\oo(\lambda^{-1}).
		\end{align}
	\end{itemize}
	The holomorphism is the direct result of $\mu, w^\Gamma\in L^2(\Gamma)$.
	For \eqref{e3.55}, it is well-known that when $\Gamma$ consists of finite Jordan curves,  $\cc^{\Gamma}_\pm$ are bounded on $L^2(\Gamma)$ and
	\begin{align}
		\cc_+^\Gamma-\cc_-^\Gamma=\identity,
	\end{align}
	therefore,
	\begin{subequations}
		\begin{align}
			&\cc^{\Gamma}_+(\mu w^\Gamma)=\cc_{w^\Gamma}\mu+(\cc^\Gamma_+-\cc^\Gamma_-)(\mu w^\Gamma_+)=\mu(I+w^{\Gamma}_+),\\
			&\cc^{\Gamma}_-(\mu w^\Gamma)=\cc_{w^\Gamma}\mu+(\cc^\Gamma_--\cc^\Gamma_+)(\mu w^\Gamma_-)=\mu(I-w^{\Gamma}_-),
		\end{align}
	\end{subequations}
	which is equivalent to \eqref{e3.55}.
	\eqref{e3.56} is the direct consequence of the compactness of $\Gamma$.
	
	Now, we investigate the analytic function on $\mathbb{C}\setminus\Gamma$:
	\begin{align}\label{e3.36}
		\Phi=\cc^{\Gamma}(\mu w^{\Gamma}).
	\end{align}
	By the above argument, it follows that $\Psi=\Psi(\lambda)=\Phi(\bar\lambda^{-1})^\dagger$ is holomorhpic on $\mathbb{C}\setminus\Gamma$ and satisfies that
	\begin{itemize}
		\item For $\lambda\in\Gamma$,
		\begin{align}\label{e3.59}
			\Psi_+(\lambda)=(V^\Gamma(\bar\lambda^{-1})^\dagger)^{-1}\Psi_-(\lambda).
		\end{align}
		\item As $\lambda\to\infty$,
		\begin{align}\label{e3.48}
			\Psi(\lambda)=\oo(1),
		\end{align}
	\end{itemize}
	Thus, by \eqref{e3.32}, \eqref{e3.55}, and \eqref{e3.59}, it follows that $\Phi\Psi$ admits no jump over $\Gamma\setminus\Sigma$, that is, for  $\lambda\in\Gamma\setminus\Sigma$,
	\begin{align}\label{e3.50}
		\Phi_+(\lambda)\Psi_+(\lambda)=\Phi_-(\lambda)V^{\Gamma}(\bar\lambda^{-1})^\dagger\Psi_+(\lambda)=\Phi_-(\lambda)\Psi_-(\lambda).
	\end{align}
	Recalling the asymptotic property \eqref{e3.56} of $\Phi(\lambda)$, we have that at $\lambda=0$,
	\begin{align}
		\Psi(\lambda)=\oo(\lambda),
	\end{align}
	therefore, $\Phi(\lambda)\Psi(\lambda)\lambda^{-1}$ is holomorphic on the unit disc $D_-$, and by Cauchy's Integral Theorem, the integral
	\begin{equation}\label{e3.51}
		\begin{split}
			\mathbf{0}&=\oint_\Sigma\Phi_-(\lambda)\Psi_-(\lambda)\lambda^{-1}\frac{\ddddd\lambda}{2\pi\ii}=\oint_\Sigma[\Phi_-(V^{\Gamma})^\dagger\Phi_-^\dagger](\lambda)\lambda^{-1}\frac{\ddddd\lambda}{2\pi\ii}=\int_0^{2\pi}[\Phi_-(V^{\Gamma})^\dagger\Phi_-^\dagger](e^{\ii\theta})\frac{\ddddd\theta}{2\pi}.
		\end{split}
	\end{equation}
	However, recalling \eqref{e4.41}, we obtain that the right-hand side of \eqref{e3.51} vanishes if and only if $\Phi(\lambda)$ tends to zero almost everywhere on the boundary of $D_-$, which in view of \eqref{e3.36} is equivalent to that $\mu(\lambda)$ is vanishing on the support of $w^{\Gamma}(\lambda)$.
	This controverts to the assumption.
	These prove the results.
\end{proof}

\begin{corollary}\label{c3.1}
	Set the Cauchy-type integral operator $\cc_w$ as defined in \eqref{e3.22}.
	There is no non-trivial solution in $\mu\in L^2(\tilde\Sigma)$ such that
	\begin{align}
		\cc_w\mu(\lambda)=\mu(\lambda),\quad \text{a.e. }\lambda\in\tilde\Sigma.
	\end{align}
	
\end{corollary}

\begin{proof}
	\eqref{e3.25} shows that $\tilde V(\lambda)$ admits a proper factorization.
	And since $r\in H^1$ and $f(\lambda)$ is a rational function, by \eqref{e3.25} and \eqref{e3.26}, $w_\pm(\lambda)$ are $\frac{1}{2}$-H\"older continuous on $\lambda\in\tilde\Sigma$.
	By the choice of $f(\lambda)$ and $\breve{f}(\lambda)$ in Proposition \ref{p4}, $\tilde V(\lambda)$ satisfies that on $\tilde\Sigma_1\cup\tilde\Sigma_2$,
	\begin{align}
		\tilde V(\bar\lambda^{-1})=\tilde V(\lambda)^\dagger.
	\end{align}
	Since $\tilde V(\lambda)$ is self-adjoint on $\lambda\in\Sigma$, it follows that
	\begin{align}
		\re\ \tilde V(\lambda)>\mathbf{0}.
	\end{align}
	Thus, the results is the consequence of Lemma \ref{l3.7}.
	We have proved the results.
\end{proof}

\begin{remark}\label{r3.7}
	We have proved the solvability of RH problem \ref{r2.7}, and it is readily seen that the solution has the following symmetric property that for $\lambda\in\mathbb{C}\setminus\tilde\Sigma$
	\begin{align}\label{e4.60}
		\tilde M(0)^{-1}\tilde M(\bar\lambda^{-1})=\sigma_2\overline{\tilde M(\lambda)}\sigma_2.
	\end{align}
	And by \eqref{e2.46} and \eqref{e3.23}, the solution of the initial-value problem is written by $\tilde M(\lambda)$ in the form,
	\begin{equation}
		q_n(t)=[\tilde M(0,n+1,t)]_{1,2}.
	\end{equation}
	Especially, under the reflectionless condition, that is, $r\equiv0$, the RH problem is still uniquely solved, which is equivalent to the solvability of RH problem \ref{r1} with $r\equiv0$.
	We denote the solution only with the discrete spectral data $(\mathcal{Z},\pp)$ as
	$M^{(\mathcal{Z},\pp)}(\lambda)=M^{(\mathcal{Z},\pp)}(\lambda,n,t)$.
	In this vein, it is readily seen the identity from \eqref{e4.60} that for $\lambda\in\Sigma$,
	\begin{align}
		M^{(\mathcal{Z},\pp)}(0)^{-1}M^{(\mathcal{Z},\pp)}(\lambda)=\sigma_2\overline{M^{(\mathcal{Z},\pp)}(\lambda)}\sigma_2.
	\end{align}
	Denote the soliton solution as
	\begin{align}
		q_n^{\mathcal{Z}}(t)=[M^{(\mathcal{Z},\pp)}(0,n+1,t)]_{1,2}.
	\end{align}
	
	It is readily seen that $M^{(\mathcal{Z},\pp)}(\lambda)$ is unique.
	In fact, we assume that there is another solution $M^{(\mathcal{Z},\pp)'}(\lambda)$ solves RH problem \ref{r1} for $r\equiv0$, then we have that by Proposition \ref{p4}, $M^{(\mathcal{Z},\pp)}(\lambda)(M^{(\mathcal{Z},\pp)'})^{-1}(\lambda)$ is entire function, and by Liouville's Theorem, $M^{(\mathcal{Z},\pp)}(M^{(\mathcal{Z},\pp)'})^{-1}\equiv I$ on the whole plane, that is, \begin{align}
		M^{(\mathcal{Z},\pp)}\equiv M^{(\mathcal{Z},\pp)'}.
	\end{align}
	
\end{remark}

\section{Soliton resolution  in sector II}\label{s5}
Previously, we have presented the direct scattering transformation and constructed the RH problem for the initial-value problem that is uniquely solvable.
The aim of this section is to carry out the analysis in sector II.
In this vein, the phase function admits two different saddle points on the unit circle: $S_1$ and $S_2$,
and the Taylor's expansion of $\phi(\lambda)$ at $S_j$ is
\begin{equation}\label{e5.2}
	\begin{split}
		\phi(\lambda)-\phi(S_j)=&(-1)^{j-1}\ii t\sqrt{1-\xi^2}S_j^{-2}(\lambda-S_j)^2+\oo(\abs{\lambda-S_j}^3)
	\end{split},\quad
	\lambda\to S_j,\quad j=1,2,
\end{equation}
where 
\begin{equation}
	S_1=-\ii\xi-\sqrt{1-\xi^2},\quad S_2=-\ii\xi+\sqrt{1-\xi^2}.\label{S1S2}
\end{equation}
Due to this fact, in the asymptotic formula, the leading terms contain soliton and oscillatory parts, and the remaining is dominated by $\oo(t^{-\frac{3}{4}})$.
\begin{figure}
	\centering\includegraphics[width=0.5\textwidth]{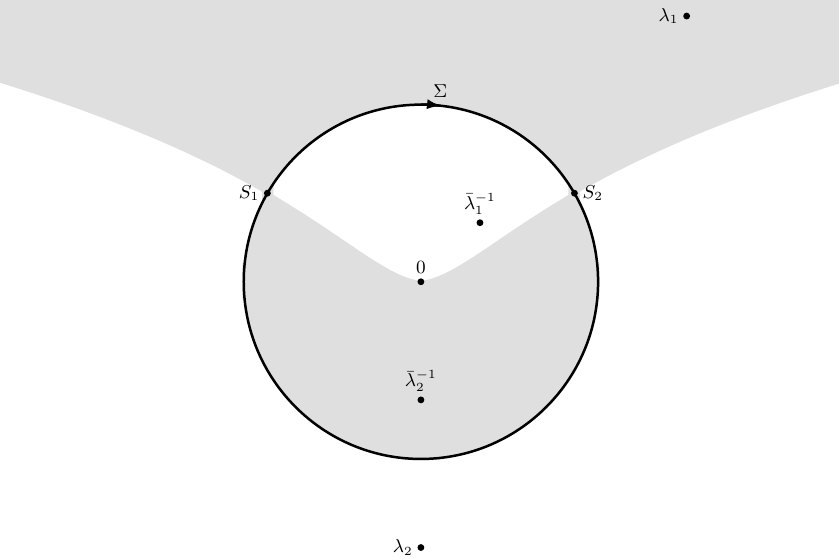}
	\caption{The signature table for $\mathrm{Re}\phi$, where $\Sigma$ is the jump contour, $\mathrm{Re}\phi>0$ on the grayed region, and $\mathrm{Re}\phi<0$ otherwise.}\label{f2}
\end{figure}

At the beginning, we introduce the scalar function that admits a jump on the lower-half unit arc $\wideparen{S_1S_2}=\{e^{\ii\theta}:\theta\in(\arg S_1,3\pi-\arg S_1)\}$,%upper 换成lower
\begin{align}\label{e3.1}
	T(\lambda)=T(\lambda,\xi)=e^{-\int_{S_2}^{S_1}\frac{\ln(1+\abs{r(\varsigma)}^2)}{\varsigma-\lambda}\frac{\ddddd \varsigma}{2\pi\ii}}\prod_{\lambda_l\in \mathcal{Z}_\xi^-}\left(\frac{\lambda-\lambda_l}{\lambda-\bar\lambda_l^{-1}}\right)^{\alpha_l},\quad \xi\in(-1,1),
\end{align}
where $\mathcal{Z}_\xi^-=\mathcal{Z}\cap\{\mathrm{Re}\phi<0\}$ and the integral is along the unit arc from $S_2$ to $S_1$.
By this scalar function, the aim is to construct a new RH problem, the solution is $M^{(1)}(\lambda)=M(\lambda)T^{\sigma_3}(\lambda)$, with spectral data satisfying that the jump matrix admits a proper factorization and the exponential functions $e^\phi$, $e^{-\phi}$ in the residue condition do not blow up.
We assert that $T(\lambda)$ admits properties shown in the following proposition.
\begin{proposition}\label{p4.1}
	\eqref{e3.1} satisfies the following properties:
	\begin{enumerate}[label=(\alph*)]
		\item $T(\lambda)$ is a meromorphic function for $\lambda\in\mathbb{C}\setminus\wideparen{S_1S_2}$, and admits $\alpha_j$-order pole at $\lambda=\lambda_l\in\mathcal{Z}$ and at $\lambda=\bar\lambda_l^{-1}\in\overline{\mathcal{Z}^{-1}}$.
		\item For $\lambda\in\mathbb{C}\setminus(\wideparen{S_1S_2}\cup\mathcal{Z}\cup\overline{\mathcal{Z}^{-1}})$, $T(\lambda)$ admits the symmetric property
		\begin{align}
			T(\bar\lambda^{-1})=T(0)\overline{T^{-1}(\lambda)}.
		\end{align}
		\item For $\lambda\in\wideparen{S_1S_2}$, the boundary values $T_\pm(\lambda)$ admit the identity
		\begin{align}\label{e5.9}
			T_+(\lambda)=T_-(\lambda)(1+\abs{r(\lambda)}^2)^{-1}.
		\end{align}
		\item As $\lambda\to\infty$, we have
		\begin{align}
			T(\lambda)\sim 1+\oo(\lambda^{-1}).
		\end{align}
		\item Near the stationary phase points $S_j$, $T(\lambda)$ admits the asymptotic properties near $\lambda=S_j,j=1,2$,
		\begin{align}\label{e3.5}
			T(\lambda)=\prod_{\lambda_l\in \mathcal{Z}_\xi^-}\left(\frac{S_j-\lambda_l}{S_j-\bar\lambda_l^{-1}}\right)^{\alpha_l}\left(\frac{\lambda-S_2}{\lambda-S_1}\right)^{\ii\nu_j}e^{\tilde\alpha_j(S_j)}+\oo(\abs{\lambda-S_j}^{\frac{1}{2}}).
		\end{align}
	\end{enumerate}
\end{proposition}
\begin{proof}
	(a) is directly verified by \eqref{e3.1}.
	(b) is the results from computing $\overline{T(\lambda)}T(\bar\lambda^{-1})$ for $\lambda\in\mathbb{C}\setminus(\wideparen{S_1S_2}\cup\mathcal{Z}\cup\overline{\mathcal{Z}^{-1}})$.
	For (c), we take the natural logarithm on both sides of \eqref{e5.9}, and it is equivalent to
	\begin{align}
		\int_{S_2}^{S_1}\frac{\ln(1+\abs{r(\varsigma)}^2)}{\varsigma-\lambda_+}\frac{\ddddd  \varsigma}{2\pi\ii}-\int_{S_2}^{S_1}\frac{\ln(1+\abs{r(\varsigma)}^2)}{\varsigma-\lambda_-}\frac{\ddddd  \varsigma}{2\pi\ii}=\ln(1+\abs{r(\lambda)}^2),\quad \lambda\in\wideparen{S_1S_2},
	\end{align}
	which is correct by Sokhotski-Plemelj formula.
	Since $r\in H^1$, (d) is the direct consequence of the definition \eqref{e3.1}.
	
	The remaining is (e).
	To check (e), we refer to a conformal mapping \eqref{e5.10} that transforms the unit circle to the real line.
	Take
	\begin{equation}
		\delta(\lambda)=\delta(\lambda,\xi)=T(\lambda)\prod_{\lambda_l\in \mathcal{Z}_\xi^-}\left(\frac{\lambda-\lambda_l}{\lambda-\bar\lambda_l^{-1}}\right)^{-\alpha_l}=e^{-\int_{S_2}^{S_1}\frac{\ln(1+\abs{r(\varsigma)}^2)}{\varsigma-\lambda}\frac{\ddddd \varsigma}{2\pi\ii}}.
	\end{equation}
	With the conformal mapping: $\lambda\mapsto\lambda'$, $\varsigma\mapsto \varsigma'$,
	\begin{align}\label{e5.10}
		\lambda=\sqrt{S_1S_2}\frac{\sqrt{S_1}+\sqrt{S_2}\lambda'}{\sqrt{S_2}+\sqrt{S_1}\lambda'},\quad \varsigma=\sqrt{S_1S_2}\frac{\sqrt{S_1}+\sqrt{S_2}\varsigma'}{\sqrt{S_2}+\sqrt{S_1}\varsigma'},
	\end{align}
	we obtain that
	\begin{align}\label{e17}
		\int_{S_2}^{S_1}\frac{f(\varsigma)-f(S_1)}{(2\pi\ii)(\lambda-\varsigma)}\mathrm{d}\varsigma=\frac{\sqrt{S_2}+\sqrt{S_1}\lambda'}{2\pi\ii}\int_{-\infty}^{0}\frac{f(\varsigma)-f(S_1)}{(\sqrt{S_2}+\sqrt{S_1}\varsigma')(\varsigma'-\lambda')}\ddddd \varsigma',
	\end{align}
	where $f(\varsigma)=\ln(1+\lvert r(\varsigma)\rvert^2)$.
	Since $r(\varsigma)\in H^1\subset L^\infty(\Sigma)$, it follows that
	\begin{align}
		F(\varsigma')=\begin{cases}
			\frac{f(\varsigma(\varsigma'))-f(S_1)}{\sqrt{S_2}+\sqrt{S_1}\varsigma'}&\varsigma'\le0,\\
			0&\varsigma'>0,
		\end{cases}
	\end{align}
	belongs to $H^1(\mathbb{R})$ and $F(\varsigma'=0)=0$;
	therefore, we apply Lemma 23.3 of \cite{beals1988direct} on $F(\varsigma')\in H^1(\mathbb{R})$, and get that for any $\lambda\in\mathbb{C}\setminus\mathbb{R}$,
	\begin{align}
		&\lvert \int_{-\infty}^{+\infty}\frac{F(\varsigma')}{(2\pi\ii)(\varsigma'-\lambda')}\mathrm{d}s'-\int_{-\infty}^{+\infty}\frac{F(\varsigma')}{2\pi\ii \varsigma'}\mathrm{d}\varsigma'\rvert\lesssim\parallel F\parallel_{H^1}\lvert \lambda'\rvert^{\frac{1}{2}},
	\end{align}
	that is,
	\begin{align}\label{e3.12}
		\delta(\lambda)\sim\left(\frac{\lambda-S_2}{\lambda-S_1}\right)^{\ii\nu_1}e^{\tilde\alpha_1(S_1)}+\oo(\abs{\lambda-S_j}^{\frac{1}{2}}),
	\end{align}
	which is an equivalence of \eqref{e3.5} for $j=1$, while the proof for $j=2$ is parallel.
	These complete the proof.
\end{proof}
\subsection{Transform to a $\bar\partial$-RH problem}\label{s1m1}
As it usually does, we transform $M(\lambda)$ into
\begin{align}\label{e4.13}
	M^{(1)}(\lambda)=M(\lambda)T^{\sigma_3}(\lambda).
\end{align}

\begin{rhp}\label{r4.2}
	\
	\begin{itemize}
		\item $M^{(1)}(\lambda)$ is meromorphic in $\mathbb{C}\setminus\Sigma$.
		\item As $\lambda\to\infty$,
			$M^{(1)}(\lambda)=I+\oo(\lambda^{-1})$
		\item For $\lambda\in\Sigma$, we have $M^{(1)}_+(\lambda)=M^{(1)}_-(\lambda)V^{(1)}(\lambda)$, where
		\begin{align*}
			V^{(1)}(\lambda)=\begin{cases}
				\left(\begin{matrix}
					1+\abs{r(\lambda)}^2&\overline{r}(\lambda)T^{-2}(\lambda)e^{\phi(\lambda)}\\r(\lambda)T^2(\lambda)e^{-\phi(\lambda)}&1
				\end{matrix}\right)&\lambda\in\Sigma-\wideparen{S_1S_2},\\
				\left(\begin{matrix}
					1&\frac{\bar r(\lambda)}{1+\abs{r(\lambda)}^2}T_+^{-2}(\lambda)e^{\phi(\lambda)}\\\frac{r(\lambda)}{1+\abs{r(\lambda)}^2}T_-^2(\lambda)e^{-\phi(\lambda)}&1+\abs{r(\lambda)}^2
				\end{matrix}\right)&\lambda\in\wideparen{S_1S_2}.
			\end{cases}
		\end{align*}
		\item On the poles $\mathcal{Z}\cup\overline{\mathcal{Z}^{-1}}$, we have
		\begin{subequations}\label{e4.17s}
			\begin{align}
				&\underset{\lambda=\lambda_{l}}{\LL_{-\alpha}}M^{(1)}=\begin{cases}
					\sum_{j=0}^{\alpha_{l}-\alpha}\frac{\partial_\lambda^{\alpha_{l}-\alpha-j}M^{(1)}(\lambda_{l})}{(\alpha_{l}-\alpha-j)!}\left(\begin{matrix}
						0&p^T_{\lambda_{l},j}e^{\phi(\lambda_{l})}\\0&0
					\end{matrix}\right) &\lambda_j\in \mathcal{Z}_\xi^-,\\
					\sum_{j=0}^{\alpha_{l}-\alpha}\frac{\partial_\lambda^{\alpha_{l}-\alpha-j}M^{(1)}(\lambda_{l})}{(\alpha_{l}-\alpha-j)!}\left(\begin{matrix}
						0&0\\p^T_{\lambda_{l},j}e^{-\phi(\lambda_{l})}&0
					\end{matrix}\right) &\lambda_j\in \mathcal{Z}\setminus\mathcal{Z}_\xi^-,
				\end{cases}\label{e4.17a}\\
				&\underset{\lambda=\bar\lambda_{l}^{-1}}{\LL_{-\alpha}}M^{(1)}=\begin{cases}
					\sum_{j=0}^{\alpha_{l}-\alpha}\frac{\partial_\lambda^{\alpha_{l}-\alpha-j}M^{(1)}(\bar\lambda_{l}^{-1})}{(\alpha_{l}-\alpha-j)!}\left(\begin{matrix}
						0&0\\p^T_{\bar\lambda_{l}^{-1},j}e^{-\phi(\bar\lambda_{l}^{-1})}&0
					\end{matrix}\right)&\lambda_j\in\mathcal{Z_\xi^+},\\
					\sum_{j=0}^{\alpha_{l}-\alpha}\frac{\partial_\lambda^{\alpha_{l}-\alpha-j}M^{(1)}(\bar\lambda_{l}^{-1})}{(\alpha_{l}-\alpha-j)!}\left(\begin{matrix}
						0&p^T_{\bar\lambda_{l}^{-1},j}e^{\phi(\bar\lambda_{l}^{-1})}\\0&0
					\end{matrix}\right)&\lambda_j\in\mathcal{Z}\setminus\mathcal{Z_\xi^+},
				\end{cases}\label{e4.17b}
			\end{align}
			where $p^T_{\lambda_{l},j}$ and $p^T_{\bar\lambda_{l}^{-1},j}$ are some polynomials of $(n,t)$ of order at most $j$.
		\end{subequations}
		
	\end{itemize}
\end{rhp}
Using proposition \ref{p4.1}, we assert that $M^{(1)}(\lambda)$ solves the RH problem above.
That \eqref{e4.13} satisfies the analyticity, normalization, and jump condition in RH problem \ref{r4.2} is easily verified by RH problem \ref{r1} and the first three items in Proposition \ref{p4.1}.
The remaining to check is the residue condition on $\mathcal{Z}$.
On the one hand for $\lambda_{l}\in\mathcal{Z}\setminus\mathcal{Z}_\xi^-$,
since $T(\lambda)$ is analytic at $\lambda_{l}$, we calculate the quantity
\begin{equation}\label{e4.15}
	\begin{split}
		\underset{\lambda=\lambda_{l}}{\LL_{-\alpha}}M^{(1)}_1&=\frac{\partial_\lambda^{\alpha_{l}-\alpha}((\cdot-\lambda_{l})^{\alpha_{l}}M_1T)(\lambda_{l})}{(\alpha_{l}-\alpha)!}\\
		&=\sum_{j=0}^{\alpha_{l}-\alpha}\frac{\partial_\lambda^{\alpha_{l}-\alpha-j}((\cdot-\lambda_{l})^{\alpha_{l}}M_1)(\lambda_{l})}{(\alpha_{l}-\alpha-j)!}\frac{\partial_\lambda^jT(\lambda_{l})}{j!}
		=\sum_{j=0}^{\alpha_{l}-\alpha}\frac{\partial_\lambda^jT(\lambda_{l})}{j!}\underset{\lambda=\lambda_{l}}{\LL_{-\alpha-j}}M_1.
	\end{split}
\end{equation}
By \eqref{e4.13} and Leibniz rule, it follows that
\begin{align}\label{e4.16}
	\frac{\partial_\lambda^\alpha M_2(\lambda_{l})}{\alpha!}=\sum_{j=0}^\alpha\frac{\partial_\lambda^{\alpha-j} M_2^{(1)}(\lambda_{l})}{(\alpha-j)!}\frac{\partial_\lambda^{j}T(\lambda_{l})}{j!}.
\end{align}
Thus, considering \eqref{e16}, \eqref{e4.15}, and \eqref{e4.16}, we obtain that there are polynomials $p^T_{\lambda_{l},j}$ of $(n,t)$ of order at most $j$, $j=0,\dots,\alpha_{l}-1$, such that
\begin{align}\label{e4.11}
	\underset{\lambda=\lambda_{l}}{\LL_{-\alpha}}M^{(1)}_1=\sum_{j=0}^{\alpha_{l}-\alpha}\frac{\partial_\lambda^{\alpha_{l}-\alpha-j}M^{(1)}_2(\lambda_{l})}{(\alpha_{l}-\alpha-j)!}p^T_{\lambda_{l},j}e^{-\phi(\lambda_{l})},\quad\alpha=1,\dots,\alpha_{l}.
\end{align}
On the other hand for $\lambda_{l}\in\mathcal{Z}_\xi^-$, in view of the transformation \eqref{e4.13}, it is readily seen that at any $\lambda_{l}\in\mathcal{Z}_\xi^-$, the pole of $M^{(1)}$ is on the second column while the first column is holomorphic.
Calculate
\begin{equation}\label{e4.23}
	\begin{split}
		\underset{\lambda=\lambda_{l}}{\LL_{-\alpha}}M^{(1)}_2=\sum_{j=0}^{\alpha_{l}-\alpha}\underset{\lambda=\lambda_{l}}{\LL_{-\alpha-j}}T^{-1}\frac{\partial_\lambda^jM_2(\lambda_{l})}{j!}.
	\end{split}
\end{equation}
Utilizing Leibniz rule, we also have that for any natural number $\alpha<\alpha_{l}$,
\begin{equation}\label{e4.24}
	\begin{split}
		\frac{\partial_\lambda^\alpha M^{(1)}_1(\lambda_{l})}{\alpha!}&=(\alpha!)^{-1}\partial_\lambda^\alpha \left((\cdot-\lambda_{l})^{\alpha_{l}}M_1\frac{T}{(\cdot-\lambda_{l})^{\alpha_{l}}}\right)(\lambda_{l})\\
		&=\sum_{j=0}^\alpha\frac{\partial_\lambda^{\alpha-j}[(\cdot-\lambda_{l})^{\alpha_{l}}M_1](\lambda_{l})}{(\alpha-j)!}\frac{\partial_\lambda^{j}[(\cdot-\lambda_{l})^{-\alpha_{l}}T]}{j!}
		=\sum_{j=0}^\alpha\underset{\lambda=\lambda_{l}}{\mathcal{L}_{\alpha-j-\alpha_{l}}}M_1\frac{\partial_\lambda^{j}[(\cdot-\lambda_{l})^{-\alpha_{l}}T]}{j!}.
	\end{split}
\end{equation}
In view of \eqref{e16}, \eqref{e4.23}, and \eqref{e4.24}, it is readily seen that there are polynomials $p^T_{\lambda_{l},j}$ of $(n,t)$ of order at most $j$ such that for any $\lambda_{l}\in\mathcal{Z}_\xi^-$,%这里似乎没错，但还是检查下
\begin{align}\label{e4.2}
	\underset{\lambda=\lambda_{l}}{\LL_{-\alpha}}M^{(1)}_2=\sum_{j=0}^{\alpha_{l}-\alpha}\frac{\partial_\lambda^{\alpha_{l}-\alpha-j}M^{(1)}_1(\bar\lambda_{l}^{-1})}{(\alpha_{l}-\alpha-j)!}p^T_{\lambda_{l},j}e^{\phi(\lambda_{l})},\quad\alpha=1,\dots,\alpha_{l}.
\end{align}
Combining \eqref{e4.11} and \eqref{e4.2}, we obtain the residue condition \eqref{e4.17a}, and we also verify \eqref{e4.17b} in the similar way.
Thus, we complete the assertion.

In the below, applying the $\bar\partial$-RH transformation, we split the jump contour $\Sigma$ along the arcs $\wideparen{S_1S_2}$ and $\Sigma\setminus\wideparen{S_1S_2}$ into $\Sigma^{(2)}$.
We properly choose the oriented closed contour $\Sigma^{(2)}$ as shown in FIGURE \ref{f3} that it consists of line segments $S_j(1+e^{\pm\ii\pi/4}(-\epsilon_0,\epsilon_0))$ and arcs centering at the origin.
With the unit circle $\Sigma$, it also divides the complex plane into six sections: $\Omega_1,\dots,\Omega_4$, and we choose $\epsilon_0$ so small enough that $\bigcup_{k=1}^4\Omega_k$ is away from the discrete spectrum as shown in FIGURE \ref{f3}.

\begin{figure}
	\centering\includegraphics[width=0.4\textwidth]{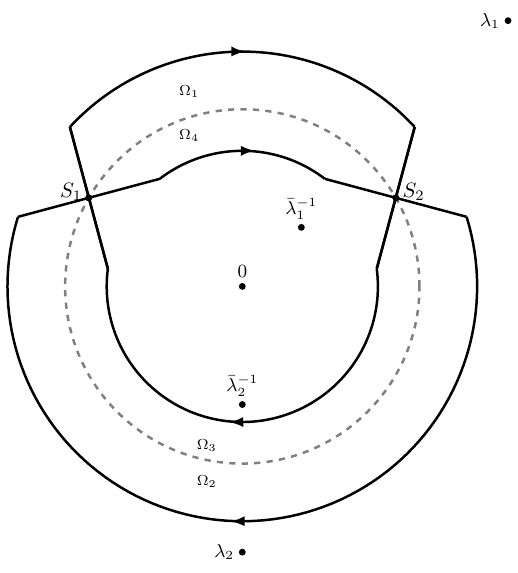}
	\caption{The jump contour $\Sigma^{(2)}=\Sigma^{(2)}_1\cup\Sigma^{(2)}_2\cup\Sigma^{(2)}_3\cup\Sigma^{(2)}_4$ and regions $\Omega_1,\dots,\Omega_4$, where $\Sigma^{(2)}_i=\partial\Omega_i\cap D_+$ for $i=1,2$ and $\Sigma^{(2)}_i=\partial\Omega_i\cap D_-$ for $i=3,4$.}\label{f3}
\end{figure}

Introduce the continuous functions on $\mathbb{C}\setminus\{0\}$
\begin{subequations}\label{e4.17}
	\begin{align}
		&R_1(\lambda)=r(\lambda/\abs{\lambda}),\quad R_4(\lambda)=\overline{R_1(\bar\lambda^{-1})},\\
		&R_2(\lambda)=\left(\frac{\bar r}{1+\abs{r}^2}\right)(\lambda/\abs{\lambda}),\quad R_3(\lambda)=\overline{R_2(\bar\lambda^{-1})},
	\end{align}
\end{subequations}
and the $2\times2$ matrix-valued function
\begin{align}\label{e4.25}
	M^{(2)}(\lambda)=M^{(1)}(\lambda)\mathcal{R}(\lambda),\quad \mathcal{R}(\lambda)=\begin{cases}
		\left(\begin{matrix}
			1&0\\-R_1(\lambda)T^2(\lambda)e^{-\phi(\lambda)}&1
		\end{matrix}\right)&\lambda\in\Omega_1,\\
		\left(\begin{matrix}
			1&-R_2(\lambda)T^{-2}(\lambda)e^{\phi(\lambda)}\\0&1
		\end{matrix}\right)&\lambda\in\Omega_2,\\
		\left(\begin{matrix}
			1&0\\R_3(\lambda)T^2(\lambda)e^{-\phi(\lambda)}&1
		\end{matrix}\right)&\lambda\in\Omega_3,\\
		\left(\begin{matrix}
			1&R_4(\lambda)T^{-2}(\lambda)e^{\phi(\lambda)}\\0&1
		\end{matrix}\right)&\lambda\in\Omega_4,\\
	I&otherwise.
	\end{cases}
\end{align}
Utilizing RH problem \ref{r4.2} and the transformation \eqref{e4.25}, ones easily check that $M^{(2)}(\lambda)$ uniquely solves the following $\bar\partial$-RH problem with the jump contour $\Sigma^{(2)}$ depicted in FIGURE \ref{f3}.
\begin{drhp}\label{db4.3}
	\
	\begin{itemize}
		\item $M^{(2)}(\lambda)$ is continuous in $\mathbb{C}\setminus\{\Sigma^{(2)}\cup\mathcal{Z}\cup\overline{\mathcal{Z}^{-1}}\}$.
		\item As $\lambda\to\infty$, we have $M(\lambda)=I+\oo(\lambda^{-1})$.
		\item On $\Sigma^{(2)}$, we have $M^{(2)}_+(\lambda)=M^{(2)}_-(\lambda)V^{(2)}(\lambda)$, where
		\begin{align}\label{e4.21}
			V^{(2)}(\lambda)=\begin{cases}
				\left(\begin{matrix}
					1&0\\R_1(\lambda)T^2(\lambda)e^{-\phi(\lambda)}&1
				\end{matrix}\right)&\lambda\in\Sigma^{(2)}_1=\partial\Omega_1\cap D_+,\\
				\left(\begin{matrix}
					1&R_2(\lambda)T^{-2}(\lambda)e^{\phi(\lambda)}\\0&1
				\end{matrix}\right)&\lambda\in\Sigma^{(2)}_2=\partial\Omega_2\cap D_+,\\
				\left(\begin{matrix}
					1&0\\R_3(\lambda)T^2(\lambda)e^{-\phi(\lambda)}&1
				\end{matrix}\right)&\lambda\in\Sigma^{(2)}_3=\partial\Omega_3\cap D_-,\\
				\left(\begin{matrix}
					1&R_4(\lambda)T^{-2}(\lambda)e^{\phi(\lambda)}\\0&1
				\end{matrix}\right)&\lambda\in\Sigma^{(2)}_4=\partial\Omega_4\cap D_-.
			\end{cases}
		\end{align}
		\item The residue condition is the same as that of RH problem \ref{r4.2} with $M^{(1)}(\lambda)$ replaced by $M^{(2)}(\lambda)$.
		\item For $\mathbb{C}\setminus\{\Sigma^{(2)}\cup\mathcal{Z}\cup\overline{\mathcal{Z}^{-1}}\}$, we have $\bar\partial M^{(2)}(\lambda)=M^{(2)}(\lambda)\bar\partial\mathcal{R}(\lambda)$.
	\end{itemize}
\end{drhp}

Rewrite $\lambda\in\mathbb{C}$ in the polar coordinates: $\lambda=\rho e^{\ii\theta}$.
Notice that for fixed $\theta$, the function $R_j(\rho e^{\ii\theta})$ are constant along the ray $\mathbb{R}^+e^{\ii\theta}$, thus utilizing the formula $\bar\partial=e^{i\theta}(\partial_\rho+\ii\rho^{-1}\partial_\theta)/2$, we calculate
\begin{subequations}\label{e4.3}
	\begin{align}
		\bar\partial R_1(\lambda)&=\frac{\ii}{2}\rho^{-1}e^{\ii\theta}\partial_\theta \left(r(e^{\ii\theta})\right)\label{e4.3a},\\
		\bar\partial R_2(\lambda)&=\frac{\ii}{2}\rho^{-1}e^{\ii\theta}\frac{\overline{\partial_\theta\left(r(e^{\ii\theta})\right)}-\overline{r(e^{\ii\theta})^2}\partial_\theta\left(r(e^{\ii\theta})\right)}{(1+\abs{r(e^{\ii\theta})}^2)^2}.
	\end{align}
\end{subequations}
Through \eqref{e4.3}, it is readily seen that the $L^1$-norm of $\bar\partial R_j$ on any circle $\{\lambda:\abs{\lambda}=C>0\}$ is well dominated by the norm of $r\in H^1$, which is applied to those estimates in Section \ref{s4.3.2}.

\subsection{Factorization of the $\bar\partial$-RH problem}
In this part, we decompose the solution for $\bar\partial$-RH problem into a product of solutions for an $RH$ problem and for a $\bar\partial$ problem, that is,
\begin{align}\label{e5.32}
	M^{(2)}(\lambda)=M^{(2)}_D(\lambda)M^{(2)}_{RH}(\lambda).
\end{align}

\begin{rhp}\label{r4.4}
	$M^{(2)}_{RH}(\lambda)$ solves the $\bar\partial$-RH problem \ref{db4.3} with $\bar\partial\mathcal{R}\equiv 0$.
\end{rhp}
\begin{dbarproblem}\label{d4.5}
	\
	\begin{itemize}
		\item $M^{(2)}_{D}(\lambda)$ is continuous in $\mathbb{C}$.
		\item As $\lambda\to\infty$, we have $M^{(2)}_{D}(\lambda)=I+\oo(\lambda^{-1})$.
		\item For $\lambda\in\mathbb{C}$, we have $\bar\partial M^{(2)}_D(\lambda)=M^{(2)}_D(\lambda)\tilde{\mathcal{R}}(\lambda)$, $\tilde{\mathcal{R}}(\lambda)=M^{(2)}_{RH}(\lambda)\bar\partial\mathcal{R}(\lambda)(M^{(2)}_{RH}(\lambda))^{-1}$.
	\end{itemize}
	
\end{dbarproblem}

\subsubsection{Reflection coefficient vanishes}\label{s531}

In this part, we discuss the RH problem \ref{r4.2} when the reflection coefficient vanishing and denote its solution as $M^{(\mathcal{Z})}(\lambda)=M^{(\mathcal{Z})}(\lambda,n,t)$.
\begin{rhp}\label{r5.6}
	$M^{(\mathcal{Z})}(\lambda)$ solves RH problem \ref{r4.2} with $r\equiv0$, that is, $V^{(1)}\equiv I$.
\end{rhp}
For any $\lambda_{l}\in\mathcal{Z}_\xi^-$, by the method to prove Proposition \ref{p3}, there are two rational functions $f_{\lambda_j}^{\xi}(\lambda)$, $f_{\bar\lambda_j^{-1}}^{\xi}(\lambda)$ such that $M^{(\mathcal{Z})}(\lambda)\left(\begin{matrix}
	1&f_{\lambda_{l}}^{\xi}(\lambda)e^{\phi(\lambda)}\\0&1
\end{matrix}\right)$, $M^{(\mathcal{Z})}(\lambda)\left(\begin{matrix}
	1&0\\f_{\bar\lambda_{l}^{-1}}^{\xi}(\lambda)e^{-\phi(\lambda)}&1
\end{matrix}\right)$ are holomorphic at $\lambda=\lambda_{l}$, $\lambda=\bar\lambda_{l}^{-1}$, respectively,
while for $\lambda_{l}\in\mathcal{Z}\setminus\left(\mathcal{Z}_\xi^-\cup\mathcal{Z}_\xi\right)$, there are rational functions $f_{\lambda_j}^{\xi}(\lambda)$, $f_{\bar\lambda_j^{-1}}^{\xi}(\lambda)$ such that $M^{(\mathcal{Z})}(\lambda)\left(\begin{matrix}
	1&0\\f_{\lambda_{l}}^{\xi}(\lambda)e^{-\phi(\lambda)}&1
\end{matrix}\right)$, $M^{(\mathcal{Z})}(\lambda)\left(\begin{matrix}
	1&f_{\bar\lambda_{l}^{-1}}^{\xi}(\lambda)e^{\phi(\lambda)}\\0&1
\end{matrix}\right)$ are holomorphic at $\lambda=\lambda_{l}$, $\lambda=\bar\lambda_{l}^{-1}$, respectively.
In view of these facts, choosing a small enough positive constant, we construct the matrix-valued function
\begin{align}\label{e5.34}
	M^{(3)}(\lambda)=M^{(\mathcal{Z})}(\lambda)\begin{cases}
		\left(\begin{matrix}
			1&f_{\lambda_{l}}^{\xi}(\lambda)e^{\phi(\lambda)}\\0&1
		\end{matrix}\right)&\abs{\lambda-\lambda_{l}}<\epsilon_1,\quad \lambda_{l}\in\mathcal{Z}_\xi^-,\\
		\left(\begin{matrix}
			1&0\\f_{\bar\lambda_{l}^{-1}}^{\xi}(\lambda)e^{-\phi(\lambda)}&1
		\end{matrix}\right)&\abs{\lambda-\bar\lambda_{l}^{-1}}<\epsilon_1,\quad \lambda_{l}\in\mathcal{Z}_\xi^-,\\
		\left(\begin{matrix}
			1&0\\f_{\lambda_{l}}^{\xi}(\lambda)e^{-\phi(\lambda)}&1
		\end{matrix}\right)&\abs{\lambda-\lambda_{l}}<\epsilon_1,\quad \lambda_{l}\in\mathcal{Z}\setminus(\mathcal{Z}_\xi^-\cup\mathcal{Z}_\xi),\\
		\left(\begin{matrix}
			1&f_{\bar\lambda_{l}^{-1}}^{\xi}(\lambda)e^{\phi(\lambda)}\\0&1
		\end{matrix}\right)&\abs{\lambda-\bar\lambda_{l}^{-1}}<\epsilon_1,\quad \lambda_{l}\in\mathcal{Z}\setminus(\mathcal{Z}_\xi^-\cup\mathcal{Z}_\xi),\\
		I&otherwise.
	\end{cases}
\end{align}
It is clear from \eqref{e5.34} that $M^{(3)}(\lambda)$ is the solution of an RH problem that is jump condition is over $\Sigma^{(3)}=\bigcup_{\lambda\in\mathcal{Z}
	\setminus\mathcal{Z}_\xi}\{\abs{\lambda-\lambda_{l}}=\epsilon_1\text{ or }\abs{\lambda-\bar\lambda_{l}^{-1}}=\epsilon_1\}$, and seeing the signature table of $\re\phi(\lambda)$ in FIGURE \ref{f2}, if $\epsilon_1$ is small enough, there exists a positive constant $C$ such that
\begin{align}\label{e5.35}
	\norm{V^{(3)}-I}_{L^\infty(\Sigma^{(3)})}\le e^{-Ct}.
\end{align}
Explicitly, the RH problem reads as follows.
\begin{rhp}
	\
	\begin{itemize}
		\item $M^{(3)}(\lambda)$ is holomorphic in $\mathbb{C}
		\setminus(\Sigma^{(3)}\cup\mathcal{Z}_\xi\cup\overline{\mathcal{Z}_\xi^{-1}})$.
		\item As $\lambda\to\infty$, $M^{(3)}(\lambda)I+\oo(\lambda^{-1})$.
		\item On $\Sigma^{(3)}$, $M^{(3)}_+(\lambda)=M^{(3)}_-(\lambda)V^{(3)}(\lambda)$, where
		\begin{align}
			V^{(3)}(\lambda)=\begin{cases}
				\left(\begin{matrix}
					1&f_{\lambda_{l}}^{\xi}(\lambda)e^{\phi(\lambda)}\\0&1
				\end{matrix}\right)&\abs{\lambda-\lambda_{l}}=\epsilon_1,\ \lambda_{l}\in\mathcal{Z}_\xi^-,\\
				\left(\begin{matrix}
					1&0\\f_{\bar\lambda_{l}^{-1}}^{\xi}(\lambda)e^{-\phi(\lambda)}&1
				\end{matrix}\right)&\abs{\lambda-\bar\lambda_{l}^{-1}}=\epsilon_1,\ \lambda_{l}\in\mathcal{Z}_\xi^-,\\
				\left(\begin{matrix}
					1&0\\f_{\lambda_{l}}^{\xi}(\lambda)e^{-\phi(\lambda)}&1
				\end{matrix}\right)&\abs{\lambda-\lambda_{l}}=\epsilon_1,\ \lambda_{l}\in\mathcal{Z}\setminus(\mathcal{Z}_\xi^-\cup\mathcal{Z}_\xi),\\
				\left(\begin{matrix}
					1&f_{\bar\lambda_{l}^{-1}}^{\xi}(\lambda)e^{\phi(\lambda)}\\0&1
				\end{matrix}\right)&\abs{\lambda-\bar\lambda_{l}^{-1}}=\epsilon_1,\ \lambda_{l}\in\mathcal{Z}\setminus(\mathcal{Z}_\xi^-\cup\mathcal{Z}_\xi).
			\end{cases}
		\end{align}
		\item When $\mathcal{Z}_\xi\ne\emptyset$, the residue condition corresponds that of RH problem \ref{r4.2} on $\mathcal{Z}_\xi$.
	\end{itemize}
\end{rhp}

Set
\begin{align}\label{e5.43}
	M^{(\mathcal{Z}_\xi)}=\begin{cases}
		I&\mathcal{Z}_\xi=\emptyset,\\
		M^{(\mathcal{Z}_\xi,\pp^T_\xi)}&\mathcal{Z}_\xi\ne\emptyset,
	\end{cases}
\end{align}
where $M^{(\mathcal{Z}_\xi,\pp^T_\xi)}$ is as defined in Remark \ref{r3.7}, and the discrete spectral data $(\mathcal{Z}_\xi,\pp^T_\xi)$ is the restriction of RH problem \ref{r4.2}'s discrete spectral data $(\mathcal{Z},\pp^T)$ on $\mathcal{Z}_\xi$.
Recalling the discussion in Remark \ref{r3.7}, we have that $M^{(\mathcal{Z}_\xi)}$ uniquely exists.
Referring to Proposition \ref{p3}, for any $\lambda_{l}\in\mathcal{Z}_\xi$, there is a rational function $f_{\lambda_{l}}(\lambda)$ such that both $M^{(3)}(\lambda)\left(\begin{matrix}
		1&0\\f_{\lambda_{l}}(\lambda)e^{-\phi(\lambda)}&1
	\end{matrix}\right)$ and $M^{(\mathcal{Z}_\xi)}(\lambda)\left(\begin{matrix}
		1&0\\f_{\lambda_{l}}e^{-\phi(\lambda)}&1
	\end{matrix}\right)$
are holomorphic at $\lambda=\lambda_{l}$, and then it follows that
\begin{align}\label{e5.39}
	\tilde M^{(3)}(\lambda)=M^{(3)}(\lambda)\left(M^{(\mathcal{Z}_\xi)}(\lambda)\right)^{-1}
\end{align}
is analytic at $\lambda=\lambda_{l}$.
Similarly, it is also analytic at $\lambda=\bar\lambda_{l}^{-1}$.
Based on the above arguments, it is clear that \eqref{e5.39} solves a small-norm RH problem without poles
\begin{rhp}
	\
	\begin{itemize}
		\item $\tilde M^{(3)}(\lambda)$ is holomorphic in $\lambda\in\mathbb{C}\setminus\Sigma^{(3)}$. % 应该是$\mathbb{C} $\backslash$ \Sigma^{(3)}$吧？
		\item As $\lambda\to\infty$, we have $\tilde M^{(3)}(\lambda)=I+\oo(\lambda^{-1})$.
		\item On $\Sigma^{(3)}$, we have $\tilde M^{(3)}_+(\lambda)=\tilde M^{(3)}_-(\lambda)\tilde V^{(3)}(\lambda)$, where $\tilde V^{(3)}(\lambda)=M^{(\mathcal{Z}_\xi)}(\lambda)V^{(3)}(\lambda)\left(M^{(\mathcal{Z}_\xi)}(\lambda)\right)^{-1}$.
	\end{itemize}
\end{rhp}
By \eqref{e5.35} and the boundedness of $M^{(\mathcal{Z}_\xi)}(\lambda)$ on $\Sigma^{(3)}$, for $t\gg0$, the jump matrix $\tilde V^{(3)}(\lambda)-I$ decays exponentially,
\begin{align}\label{e5.41}
	\norm{\tilde V^{(3)}-I}_{L^\infty(\Sigma^{(3)})}\lesssim e^{-Ct}.
\end{align}
Finally, by the small norm theory for the RH problems shown in \cite{deift1999strong,kamvissis2003semiclassical}, it follows from \eqref{e5.34} and \eqref{e5.39} that
\begin{align}\label{e5.46}
	&\quad\quad\quad M^{(\mathcal{Z})}(0)=\tilde M^{(3)}(0)M^{(\mathcal{Z}_\xi)}(0)=\left(I+\oo(e^{-Ct})\right)M^{(\mathcal{Z}_\xi)}(0),\\
	&M^{(\mathcal{Z})}(S_j)=\tilde M^{(3)}(S_j)M^{(\mathcal{Z}_\xi)}(S_j)=\left(I+\oo(e^{-Ct})\right)M^{(\mathcal{Z}_\xi)}(S_j),\quad j=1,2.\label{e5.47}
\end{align}

\subsubsection{Solvability of $M^{(2)}_{RH}(\lambda)$}

In this part, we discuss the solvability of $M^{(2)}_{RH}(\lambda)$.
For $j=1,2$, transform $\lambda\mapsto\zeta$,
\begin{align}\label{e4.30}
	&\lambda=\beta_j\zeta+S_j,\quad\beta_j=(-1)^j\frac{\ii}{\sqrt{2}}(1-\xi^2)^{-\frac{1}{4}}t^{-\frac{1}{2}}S_j.
\end{align}
Under the scaling \eqref{e4.30}, it is readily seen that $R_1(\lambda)$, $R_2(\lambda)$, and $T(\lambda)$ admit the following asymptotic properties around $S_1$ and $S_2$.
\begin{proposition}\label{p4.9}
	If we set the transformation for $j=1,2$: $\lambda=\beta_j\zeta+S_j$, where $\beta_j$ are defined as \eqref{e4.30} and $S_j$ are the stationary phase points,
	then when $\xi\in(-1,1)$, we have the following asymptotics that as $\lambda\to S_j$,
	\begin{align}
		&R_1(\lambda)\sim r(S_j)+\oo(t^{-\frac{1}{4}}\abs{\zeta}^{\frac{1}{2}}),\quad R_2(\lambda)\sim \frac{\overline{r(S_j)}}{1+\abs{r(S_j)}^2}+\oo(t^{-\frac{1}{4}}\abs{\zeta}^{\frac{1}{2}}),\label{e4.31}\\
		&T^2(\lambda)e^{-\phi(\lambda)}\sim T_j^2\zeta^{(-1)^{j}2\ii\nu_j}e^{-(-1)^j\ii\frac{\zeta^2}{2}}+\oo(t^{-\frac{1}{4}}\abs{\zeta}^{\frac{1}{2}}e^{-(-1)^j\ii\frac{\zeta^2}{2}}),\label{e4.32}
	\end{align}
where 
\begin{align}\label{e5.58}
	T_j=\prod_{\lambda_l\in\mathcal{Z}_\xi^-}\left(\frac{S_j-\lambda_l}{S_j-\bar\lambda_l^{-1}}\right)^{\alpha_l}\frac{(-2\sqrt{2}(1-\xi^2)^\frac{3}{4}t^{\frac{1}{2}})^{(-1)^{j-1}\ii\nu_j}(\ii S_j)^{(-1)^{j}\ii\nu_j}}{e^{\ii t[(-1)^{j}\sqrt{1-\xi^2}-\xi\arg S_j-1]-\tilde\alpha_j(S_j)}}.
\end{align}
\end{proposition}
\begin{proof}
	On the one hand, recalling that $r\in H^1$, we have that $r$ is $\frac{1}{2}$-H\"older continous and for $\lambda$ on the neighborhood of $S_j$,
	\begin{equation}
		\begin{split}
			R_1(\lambda)-r(S_j)=r(\frac{\beta_j\zeta+S_j}{\abs{\beta_j\zeta+S_j}})-r(S_j)\lesssim\abs{\frac{\beta_j\zeta+S_j}{\abs{\beta_j\zeta+S_j}}-S_j}^{\frac{1}{2}}\lesssim\abs{\beta_j\zeta}^{\frac{1}{2}}\lesssim t^{-\frac{1}{4}}\abs{\zeta}^{\frac{1}{2}},
		\end{split}
	\end{equation}
	and we proved the result for $R_1$ in \eqref{e4.31}.
	In addition, when $r\in H^1$, it is readily seen that $\frac{r}{1+\abs{r}^2}\in H^1$, therefore, we similarly obtain that for $R_2(\lambda)$ in \eqref{e4.31} and complete the proof of \eqref{e4.31}.
	On the other hand, the result \eqref{e4.32} is the consequence of \eqref{e5.2}, \eqref{e3.5}, \eqref{e4.30}, and \eqref{e5.58}.
	This completes the proof.
\end{proof}
Based on these asymptotics, we construct the model RH problem \ref{r4.6} around the neighborhoods of $S_1$, $S_2$, respectively, and the jump contours $\Sigma^{(2,j)}$ are depicted in FIGURE \ref{f4}.
\begin{rhp}\label{r4.6}
	\
	\begin{itemize}
		\item $M^{(2,j)}(\lambda)$ is analytic in $\mathbb{C}\setminus\Sigma^{(2,j)}$.
		\item $M^{(2,j)}(\lambda)=I+\oo(\lambda^{-1})$ as $\lambda\to\infty$.
		\item $M^{(2,j)}(\lambda)$ admits the jump condition on $\lambda\in\Sigma^{(2,j)}=\bigcup_{k=1}^4\Sigma^{(2,j)}_k$, $\Sigma^{2,j}_k=S_j(1+e^{\frac{\ii\pi}{4}(2k+4j-7)})$,
		\begin{align}
			(M^{(2,j)})_+(\lambda)=(M^{(2,j)})_-(\lambda)V^{(2,j)}(\lambda),
		\end{align}
		where
		\begin{subequations}\label{e4.35}
			\begin{align}
				&V^{(2,1)}(\lambda=\beta_1\zeta+S_1)=\begin{cases}
					\left(\begin{matrix}
						1&0\\T_1^2r(S_1)\zeta^{-2\ii\nu_1}e^{\frac{\ii\zeta^2}{2}}&1
					\end{matrix}\right)&\lambda\in\Sigma^{(2,1)}_1,\\
					\left(\begin{matrix}
						1&T_1^{-2}\frac{\overline{r(S_1)}}{1+\abs{r(S_1)}^2}\zeta^{2\ii\nu_1}e^{-\frac{\ii\zeta^2}{2}}\\0&1
					\end{matrix}\right)&\lambda\in\Sigma^{(2,1)}_2,\\
					\left(\begin{matrix}
						1&0\\T_1^{2}\frac{r(S_1)}{1+\abs{r(S_1)}^2}\zeta^{-2\ii\nu_1}e^{\frac{\ii\zeta^2}{2}}&1
					\end{matrix}\right)&\lambda\in\Sigma^{(2,1)}_3,\\
					\left(\begin{matrix}
						1&T_1^{-2}\overline{r(S_1)}\zeta^{2\ii\nu_1}e^{-\frac{\ii\zeta^2}{2}}\\0&1
					\end{matrix}\right)&\lambda\in\Sigma^{(2,1)}_4,
				\end{cases}\\
				&V^{(2,2)}(\lambda=\beta_2\zeta+S_2)=\begin{cases}
					\left(\begin{matrix}
						1&-T_2^{-2}\overline{r(S_2)}\zeta^{-2\ii\nu_2}e^{\frac{\ii\zeta^2}{2}}\\0&1
					\end{matrix}\right)&\lambda\in\Sigma^{(2,2)}_1,\\
					\left(\begin{matrix}
						1&0\\T_2^{2}\frac{-r(S_2)}{1+\abs{r(S_2)}^2}\zeta^{2\ii\nu_2}e^{-\frac{\ii\zeta^2}{2}}&1
					\end{matrix}\right)&\lambda\in\Sigma^{(2,2)}_2,\\
					\left(\begin{matrix}
						1&T_2^{-2}\frac{-\overline{r(S_2)}}{1+\abs{r(S_2)}^2}\zeta^{-2\ii\nu_2}e^{\frac{\ii\zeta^2}{2}}\\0&1
					\end{matrix}\right)&\lambda\in\Sigma^{(2,2)}_3,\\
					\left(\begin{matrix}
						1&0\\-T_2^2r(S_2)\zeta^{2\ii\nu_2}e^{-\frac{\ii\zeta^2}{2}}&1
					\end{matrix}\right)&\lambda\in\Sigma^{(2,2)}_4.
				\end{cases}
			\end{align}
		\end{subequations}
	\end{itemize}
\end{rhp}
Choosing a small enough constant $\epsilon_2\in(0,\epsilon_0)$ and small discs $U_j=\{\lambda\in\mathbb{C}:\abs{\lambda-S_j}<\epsilon_2\}$,
we rewrite $M^{(2)}_{RH}(\lambda)$ in the form
\begin{align}\label{e4.22}
	M^{(2)}_{RH}(\lambda)=\begin{cases}
		E(\lambda)M^{(\mathcal{Z})}(\lambda),&\lambda\in\mathbb{C}\setminus(\Sigma^{(2)}\cup U_1\cup U_2),\\
		E(\lambda)M^{(\mathcal{Z})}(\lambda)M^{(2,j)}(\lambda),&\lambda\in U_j\setminus\Sigma^{(2)}, \quad j=1,2,
	\end{cases}
\end{align}
where $M^{(\mathcal{Z})}(\lambda)$ is the solution of RH problem \ref{r5.6}, $M^{(2,j)}(\lambda)$ is the solution of RH problem \ref{r4.6}, and $E(\lambda)$ admits the RH problem \ref{r4.5} with the jump contour $\Sigma^E=\partial U_1\cup\partial U_2\cup(\Sigma^{(2)})$ depicted as FIGURE \ref{f5}.
\begin{rhp}\label{r4.5}
	\
	\begin{itemize}
		\item $E(\lambda)$ is analytic in $\mathbb{C}\setminus\Sigma^E$.
		\item As $\lambda\to\infty$, $E(\lambda)=I+\oo(\lambda^{-1})$.
		\item On $\Sigma^E$, we have
		\begin{align}
			E_+(\lambda)=E_-(\lambda)V^E(\lambda),
		\end{align}
		where
		\begin{align}\label{e4.33}
			V^E=\begin{cases}
				M^{(\mathcal{Z})}V^{(2)}(M^{(\mathcal{Z})})^{-1},&\lambda\in\Sigma^{(2)}\setminus(\partial U_1\cup\partial U_2),\\
				M^{(\mathcal{Z})}M^{(2,j)}V^{(2)}(V^{(2,j)})^{-1}(M^{(\mathcal{Z})}M^{(2,j)})^{-1},&\lambda\in\Sigma^{(2)}\cap U_j,\\
				M^{(\mathcal{Z})}M^{(2,j)}(M^{(\mathcal{Z})})^{-1},&\lambda\in\partial U_1\cup\partial U_2.
			\end{cases}
		\end{align}
	\end{itemize}
\end{rhp}
In \eqref{e4.22}, recalling the discussion in Section \ref{s531}, we see that $M^{(\mathcal{Z})}$ is uniquely solved for $t\gg0$, and also, $M^{(2,j)}$ can be represented by
\begin{subequations}\label{e4.26}
	\begin{align}
		&M^{(2,1)}(\lambda)=T_1^{-\sigma_3}M^{(PC)}(\zeta,r(S_1))T_1^{\sigma_3},\\
		&M^{(2,2)}(\lambda)=\sigma_1T_2^{\sigma_3}M^{(PC)}(\zeta,-\overline{r(S_2)})T_2^{-\sigma_3}\sigma_1,
	\end{align}
\end{subequations}
where $M^{(PC)}$ is well-acknowledged as the model RH problem \ref{rhpc}.
Thus, to see the solvability of $M^{(2)}_{RH}$, we only have to look into the part of $E$, i.e., the solvability of RH problem \ref{r4.5}.

\begin{figure}
	\centering
	\begin{subfigure}{0.45\textwidth}
		\centering\includegraphics{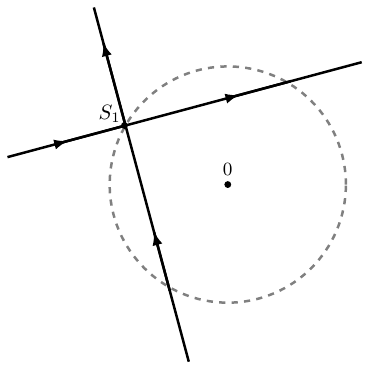}
		\subcaption*{$\Sigma^{(2,1)}$}
	\end{subfigure}
	\begin{subfigure}{0.45\textwidth}
		\centering\includegraphics{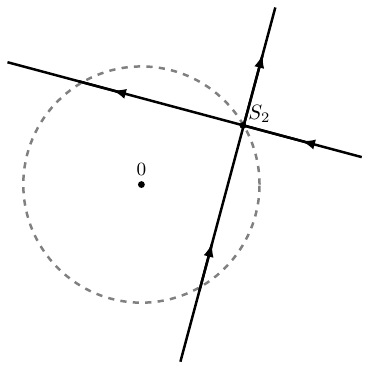}
		\subcaption*{$\Sigma^{(2,2)}$}
	\end{subfigure}
	\caption{Jump contours  $\Sigma^{(2,j)}$, $j=1,2$.}\label{f4}
	%	\caption{Jump contours for RH problem \ref{r4.6}: $\Sigma^{(2,j)}$.}\label{f4}
\end{figure}
\begin{figure}
	\centering\includegraphics[width=0.3\textwidth]{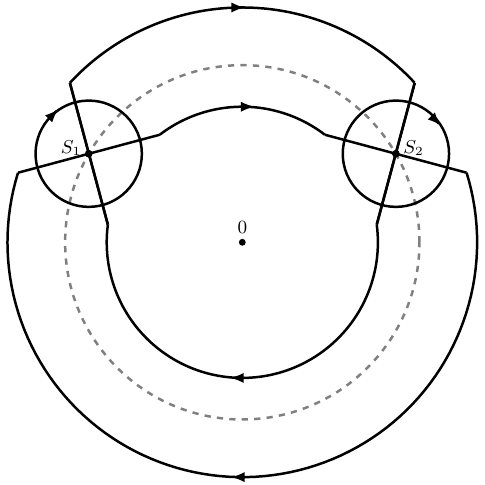}
	\caption{The jump contour $\Sigma^E=\Sigma^{(2)}\cup\partial U_1\cup\partial U_2$, where $\partial U_1$ and $\partial U_2$ are oriented clockwise.}\label{f5}
\end{figure}

We claim that RH problem \ref{r4.5} is a small-norm RH problem for $t\gg0$.
By \eqref{e3.5}, \eqref{e4.17}, \eqref{e4.21}, and the signature table of $\re \phi(\lambda)$, it is readily seen that for some positive constant $C$
\begin{align}\label{e4.29}
	\abs{V^{(2)}(\lambda)-I}\le e^{-Ct},\quad \lambda\in\Sigma^{(2)}\setminus(U_1\cup U_2).
\end{align}
Considering \eqref{e4.21}, \eqref{e4.35}, and Proposition \ref{p4.9}, we obtain that
\begin{equation}\label{e4.28}
	\begin{split}
		&\abs{V^{(2)}(\lambda)(V^{(2,j)}(\lambda))^{-1}-I}\lesssim\abs{V^{(2)}(\lambda)-V^{(2,j)}(\lambda)}\lesssim t^{-\frac{1}{4}}\abs{\zeta}^{\frac{1}{2}}e^{-\frac{\abs{\zeta}^2}{2}},
	\end{split}
	\quad\lambda\in \Sigma^{(2)}\cap U_j.
\end{equation}
In view of \eqref{e4.26}, \eqref{e4.27}, and \eqref{e4.30}, it follows that
\begin{align}\label{e4.34}
	\abs{M^{(2,j)}(\lambda)-I}\lesssim t^{-\frac{1}{2}},\quad \lambda\in\partial U_j.
\end{align}
Considering \eqref{e4.33}, \eqref{e4.29}, \eqref{e4.28}, \eqref{e4.34}, and the boundedness of $M^{(\mathcal{Z})}(\lambda)$ on $\Sigma^E$, we conclude that
\begin{align}\label{e4.36}
	\abs{V^E(\lambda)-I}\lesssim t^{-\frac{1}{4}}.
\end{align}
Since $\cc_-^{\Sigma^E}$ is a bounded operator on $L^2(\Sigma^E)$, by \eqref{e4.36}, it is readily seen that the Cauchy-type integral operator
\begin{align}
	\cc_{V^E-I}:\ g\mapsto \cc_{V^E-I}g=\cc_-(g(V^E-I))
\end{align}
is consequently bounded on $L^2(\Sigma^E)$, and
\begin{align}\label{e4.37}
	\norm{\cc_{V^E-I}}_{L^2_{op}(\Sigma^E)}\lesssim t^{-\frac{1}{4}}.
\end{align}
\eqref{e4.37} indicates that as $t\to+\infty$, the operator decays, and for sufficiently large $t$, $(\identity-\cc_{V^E-I})^{-1}$ exists and is bounded on $L^2(\Sigma^E)$.
By RH problem \ref{r4.5}, it follows that the solution can be written in the integral form:
\begin{align}\label{e5.65}
	E(\lambda)=I+\oint_{\Sigma^E}\frac{[(\identity-\cc_{V^E-I})^{-1}I](\varsigma)(V^E(\varsigma)-I)}{\varsigma-\lambda}\frac{\ddddd \varsigma}{2\pi \ii}.
\end{align}

Now, we focus on the estimate of $E(0)$, and assert that as $t\to+\infty$,
\begin{equation}\label{e5.66}
	\begin{split}
		&E(0)-I=t^{-\frac{1}{2}}\frac{M^{(\mathcal{Z_\xi})}(S_1)}{\sqrt{2}(1-\xi^2)^{\frac{1}{4}}}\left(\begin{matrix}
			0&\frac{\gamma_1(r(S_1))}{T_1^2}\\-\gamma_2(r(S_1))T_1^2&0
		\end{matrix}\right)\left(M^{(\mathcal{Z_\xi})}(S_1)\right)^{-1}\\
		&\quad\quad\quad+t^{-\frac{1}{2}}\frac{M^{(\mathcal{Z_\xi})}(S_2)}{\sqrt{2}(1-\xi^2)^{\frac{1}{4}}}\left(\begin{matrix}
			0&-\frac{\gamma_2(\overline{r(S_2)})}{T_2^2}\\\gamma_1(\overline{r(S_2)})T_2^2&0
		\end{matrix}\right)\left(M^{(\mathcal{Z_\xi})}(S_2)\right)^{-1}+\oo(t^{-\frac{3}{4}}).
	\end{split}	
\end{equation}
Taking the value of \eqref{e5.65} at $\lambda=0$, we have
\begin{equation}\label{e5.67}
	\begin{split}
		&E(0)=I+\oint_{\partial U_1\cup\partial U_2}\varsigma^{-1}[(\identity-\cc_{V^E-I})^{-1}\cc_{V^E-I}I](\varsigma)(V^E(\varsigma)-I)\frac{\ddddd \varsigma}{2\pi\ii}\\
		&\quad+\oint_{\Sigma^{(2)}}\varsigma^{-1}[(\identity-\cc_{V^E-I})^{-1}I](\varsigma)(V^E(\varsigma)-I)\frac{\ddddd \varsigma}{2\pi\ii}+\oint_{\partial U_1\cup\partial U_2}\varsigma^{-1}(V^E(\varsigma)-I)\frac{\ddddd \varsigma}{2\pi\ii}.
	\end{split}
\end{equation}
In \eqref{e5.67}, the first integral is dominated by $\oo(t^{-\frac{3}{4}})$ because of \eqref{e4.34} and \eqref{e4.37}, and in view of \eqref{e4.29} and \eqref{e4.28}, it follows that the second integral is also constrained by $\oo(t^{-\frac{3}{4}})$,
that is, as $t\to+\infty$,
\begin{equation}\label{e5.6}
	\begin{split}
		&E(0)=I+\oint_{\partial U_1\cup\partial U_2}\varsigma^{-1}(V^E(\varsigma)-I)\frac{\ddddd \varsigma}{2\pi\ii}+\oo(t^{-\frac{3}{4}}).
	\end{split}
\end{equation}
Using \eqref{e5.47}, \eqref{e4.33}, \eqref{e4.26}, \eqref{e4.27}, and Cauchy's integral formula, we formulate the integral term in \eqref{e5.6}, and consequently obtain \eqref{e5.66}.

We have proved that $M^{(2)}_{RH}(\lambda)$ is uniquely solved,
then we also claim that it is invertible.
In fact, it is readily seen that $\det M^{(2)}_{RH}(\lambda)$ has no jump over $\mathbb{C}$ and tends to 1 at the infinity.
In addition, utilizing the technique to prove Proposition \ref{p3}, we can remove the pole of $M^{(2)}_{RH}(\lambda)$ at $\lambda=\lambda_{l}$ by a triangular function with constant determinant,
that is, the singularities of $\det M^{(2)}_{RH}(\lambda)$ at $\lambda=\lambda_{l}$ are removable.
In the similar way, its singularities at $\lambda=\bar\lambda_{l}^{-1}$ are also removable.
By Liouville's Theorem, $\det M^{(2)}_{RH}(\lambda)$ is the constant 1 over $\mathbb{C}$, that is, $M^{(2)}_{RH}(\lambda)$ is invertible.

Defining the function
\begin{align}\label{e5.70}
	M^{(2)}_D(\lambda)=M^{(2)}(\lambda)\left(M^{(2)}_{RH}(\lambda)\right)^{-1}.
\end{align}
Since $M^{(2)}_{RH}(\lambda)$ is invertible, the definition \eqref{e5.70} is well-posed,
and comparing $\bar\partial$-RH problem \ref{db4.3} to RH problem \ref{r4.4}, we find that $M^{(2)}_D(\lambda)$ is the solution of $\bar\partial$ problem \ref{d4.5}.

\subsubsection{Estimate on the $\bar\partial$ problem}\label{s4.3.2}
Since we have already proved the existence of $M^{(2)}_{RH}(\lambda)$, in this part, we concentrate on some estimates of $M^{(2)}_D(\lambda)$.

$\bar\partial$ problem \ref{db4.3} is equivalent to the identity,
\begin{align}\label{e4.50}
	(\identity-S)M^{(2)}_D=I,\quad Sf(\lambda)=\iint_{\mathbb{C}}\frac{f(\varsigma)\tilde{\mathcal{R}}(\varsigma)}{\varsigma-\lambda}\frac{\ddddd L(\varsigma)}{\pi},
\end{align}
where $L(s)$ is the  Lebesgue's measure on the complex plane.
The linear integral operator $S$ is proved bounded on $L^{\infty}(\mathbb{C})$ in Proposition \ref{p4.10}, and the bound $\norm{S}_{L^\infty_{op}(\mathbb{C})}$ decays as $t\to\infty$.
Therefore, \eqref{e4.50} is uniquely solved in $L^\infty(\mathbb{C})$ for $t\gg0$.

\begin{proposition}\label{p4.10}
	Define the integral operator $S$ as \eqref{e4.50}.
	If $r\in H^1$, then $S$ is bounded on $L^\infty(\mathbb{C})$, and as $t\gg0$, the bound is dominated by $t^{-\frac{1}{4}}$, that is,
	\begin{align}
		\norm{S}_{L^\infty_{op}(\mathbb{C})}\lesssim t^{-\frac{1}{4}}.
	\end{align}
\end{proposition}
\begin{proof}
	Recalling that $\tilde{\mathcal{R}}(\lambda)$ is supported on $\bigcup_{k=1}^4\Omega_k$ and that the discrete spectrum of RH problem \ref{r4.4} is away from $\bigcup_{k=1}^4\Omega_k$, we obtain that $M^{(2)}_{RH}(\lambda)$, $(M^{(2)}_{RH}(\lambda))^{-1}$, $T(\lambda)$, and $T^{-1}(\lambda)$ are all bounded on $\bigcup_{k=1}^4\Omega_k$, and consequently,
	\begin{equation}\label{e4.52}
		\begin{split}
			&\abs{Sf(\lambda)}=\abs{\iint_{\mathbb{C}}\frac{f(\varsigma)\tilde{\mathcal{R}}(\varsigma)}{\varsigma-\lambda}\frac{\ddddd L(\varsigma)}{\pi}}=\abs{\iint_{\bigcup_{k=1}^4\Omega_k}\frac{f(\varsigma)\tilde{\mathcal{R}}(\varsigma)}{\varsigma-\lambda}\frac{\ddddd L(\varsigma)}{\pi}}\\
			&\quad \lesssim\norm{f}_{L^\infty(\mathbb{C})}\sum_{k=1}^4\iint_{\Omega_k}\abs{\frac{\bar\partial\mathcal{R}(\varsigma)}{\varsigma-\lambda}}\frac{\ddddd L(\varsigma)}{\pi}
			=\norm{f}_{L^\infty(\mathbb{C})}\sum_{k=1}^4\iint_{\Omega_k}\abs{\frac{[e^{(-1)^{k}\phi}\bar\partial R_k](\varsigma)}{\varsigma-\lambda}}\frac{\ddddd L(\varsigma)}{\pi}.
		\end{split}
	\end{equation}
	
	For the region $\Omega_1$, we write $\varsigma\in\Omega_1$ in the polar coordinates: $\varsigma=\rho e^{\ii\theta}$, $(\rho,\theta)\in(1,\rho_0)\times(\pi-\theta_\rho,\theta_\rho)$, where
	\begin{align}
		\rho_0=\abs{1+\epsilon_0e^{\ii\pi/4}}\quad \rho e^{\ii\theta_\rho}\in S_1(1+e^{-\ii\pi/4}(0,\epsilon_0))\quad \rho e^{\ii(\pi-\theta_\rho)}\in S_2(1+e^{\ii\pi/4}(0,\epsilon_0)).
	\end{align}
	It is readily seen from the Taylor's expansion of \eqref{e2.18} at $\lambda=S_1$ that as $\rho e^{\ii\theta_\rho}\to S_1$,
	\begin{equation}\label{e4.53}
		\begin{split}
			\re\phi(\rho e^{\ii\theta_\rho})
			=2t\sqrt{1-\xi^2}(\rho-1)^2+\oo(\rho-1)^3.
		\end{split}
	\end{equation}
	Utilizing \eqref{e4.3a} and the boundedness of $T$ on $\Omega_1$, we estimate the integral on $\Omega_1$ for $\lambda\ne0$,
	\begin{equation}\label{e4.54}
		\begin{split}
			&\iint_{\Omega_1}\abs{\frac{[\bar\partial R_1e^{-\phi}](\varsigma)}{\varsigma-\lambda}}\frac{\ddddd L(\varsigma)}{\pi} =\int_{1}^{\rho_0}\int_{\pi-\theta_\rho}^{\theta_\rho}\abs{\frac{\rho^{-1}\partial_\theta(r(e^{\ii\theta}))e^{-\phi(\rho e^{\ii\theta})}}{2\pi(\rho e^{\ii\theta}-\lambda)}}\rho\ddddd\theta\ddddd\rho\\
			&\quad \le\int_{1}^{\rho_0}\int_{-\frac{\pi}{2}}^{\frac{3\pi}{2}}\frac{\abs{\partial_\theta (r(e^{\ii\theta}))}e^{-\re\phi(\rho e^{\ii\theta_\rho})}}{2\pi\abs{\rho e^{\ii\theta}-\lambda}}\ddddd\theta\ddddd\rho\\
			&\quad \le\int_{1}^{\rho_0}\ddddd\rho\frac{e^{-\re\phi(\rho e^{\ii\theta_\rho})}}{2\pi}\norm{r}_{H^1}\left(\int_{-\pi}^{\pi}\abs{\rho e^{\ii\theta}-\lambda}^{-2}\ddddd\theta\right)^\frac{1}{2}\\
			&\quad\lesssim\int_{1}^{\rho_0}\ddddd\rho\frac{e^{-\re\phi(\rho e^{\ii\theta_\rho})}}{2\pi\abs{\rho-\abs{\lambda}}^{1/2}}.
		\end{split}
	\end{equation}
	For small enough $\epsilon_0$, it is readily seen from \eqref{e4.53} and \eqref{e4.54} that
	\begin{align}\label{e4.55}
		\iint_{\Omega_1}\abs{\frac{[\bar\partial R_1e^{-\phi}](\varsigma)}{\varsigma-\lambda}}\frac{\ddddd L(\varsigma)}{\pi}\le\int_{0}^{\rho_0}\ddddd\rho\frac{e^{-t\sqrt{1-\xi^2}(\rho-1)^2}}{2\pi\abs{\rho-\abs{\lambda}}^{1/2}}\lesssim t^{-\frac{1}{4}}.
	\end{align}
	
	For regions $\Omega_2, \Omega_3, \Omega_4$, letting $\lambda\ne0$, it admits the same estimates as \eqref{e4.55}, that is, for sufficiently small $\epsilon_0$ and $k=1,2,3,4$,
	\begin{align}\label{e4.56}
		\iint_{\Omega_k}\abs{\frac{[e^{(-1)^{k}\phi}\bar\partial R_k](\varsigma)}{\varsigma-\lambda}}\frac{\ddddd L(\varsigma)}{\pi}\lesssim t^{-\frac{1}{4}}.
	\end{align}
	Substituting \eqref{e4.56} into \eqref{e4.52}, we finally prove the results.
\end{proof}

Now, we focus on the estimate of $M^{(2)}_D(0)$,
\begin{equation}\label{e5.18}
	\begin{split}
		M^{(2)}_D(0)=I+\iint_{\mathbb{C}}[M^{(2)}_D\tilde{\mathcal{R}}](\varsigma)\varsigma^{-1}\frac{\ddddd L(\varsigma)}{\pi}.
	\end{split}
\end{equation}
With the notation $\varsigma=\rho e^{\ii\theta}\in\Omega_1$ and $\rho e^{\ii\theta_\rho}\in S_1(1+e^{-\ii\pi/4}(0,\epsilon_0))$, it is readily seen that, for any fixed $\xi\in(-1,1)$, there is a positive constant $\tilde\epsilon(\xi)$ such that
\begin{equation}\label{e5.82}
	\begin{split}
		&\sin\theta-\sin\theta_\rho\ge\begin{cases}
			-\tilde\epsilon(\xi)(\theta-\theta_\rho)&\theta\in\left(\frac{\pi}{2},\theta_\rho\right),\\
			\tilde\epsilon(\xi)(\theta-\pi+\theta_\rho)&\theta\in\left(\pi-\theta_\rho,\frac{\pi}{2}\right).
		\end{cases}
	\end{split}
\end{equation}
Especially, ones can choose $\tilde\epsilon(\xi)=\cos\theta_\rho$ when $\xi$ is near $1$.
Utilizing the techniques applied in Proposition \ref{p4.10} and the estimate \eqref{e5.82}, we obtain that the following integral on $\Omega_1$ is dominated by $t^{-\frac{3}{4}}$, that is,
\begin{equation}\label{e5.8}
	\begin{split}
		&\abs{\iint_{\Omega_1}[M^{(2)}_D\tilde{\mathcal{R}}](\varsigma)\varsigma^{-1}\frac{\ddddd L(\varsigma)}{\pi}}\lesssim\iint_{\Omega_1}\abs{\bar\partial R_1(\varsigma)e^{-\phi(\varsigma)}\varsigma^{-1}}\frac{\ddddd L(\varsigma)}{\pi}\\
		&\quad=\int_1^{\rho_0}\frac{e^{-\re\phi(\rho e^{\ii\theta_\rho})}}{2\pi\rho}\ddddd\rho\int_{\pi-\theta_\rho}^{\theta_\rho}\abs{\partial_\theta(r(e^{\ii\theta}))}e^{-\re\phi(\rho e^{\ii\theta})+\re\phi(\rho e^{\ii\theta_\rho})}\ddddd\theta\\
		&\quad\le\norm{r}_{H^1}\int_1^{\rho_0}\frac{ e^{-\re\phi(\rho e^{\ii\theta_\rho})}}{2\pi\rho}\ddddd\rho
		\left(\int_{\frac{\pi}{2}}^{\theta_\rho}e^{2t(\rho-\rho^{-1})(\theta-\theta_\rho)\tilde\epsilon(\xi)}\ddddd\theta+\int_{\pi-\theta_\rho}^{\frac{\pi}{2}}e^{-2t(\rho-\rho^{-1})(\theta-\pi+\theta_\rho)\tilde\epsilon(\xi)}\ddddd\theta\right)^{\frac{1}{2}}\\
		&\quad\lesssim\int_1^{\rho_0}\frac{e^{-t\sqrt{1-\xi^2}(\rho-1)^2}}{2\pi\rho}t^{-\frac{1}{2}}(\rho-\rho^{-1})^{-\frac{1}{2}}\ddddd\rho\lesssim t^{-\frac{3}{4}}.
	\end{split}
\end{equation}
The result of \eqref{e5.8} still holds for $\Omega_2$, $\Omega_3$, and $\Omega_4$, that is,
\begin{align}\label{e5.84}
	\abs{\iint_{\Omega_k}[M^{(2)}_D\tilde{\mathcal{R}}](\varsigma)\varsigma^{-1}\frac{\ddddd L(\varsigma)}{\pi}}\lesssim t^{-\frac{3}{4}},\quad k=1,2,3,4.
\end{align}
Substituting \eqref{e5.84} into \eqref{e5.18}, we obtain that as $t\to+\infty$,
\begin{align}\label{e5.73}
	M^{(2)}_D(0)= I+\oo(t^{-\frac{3}{4}}).
\end{align}

\subsection{Asymptotic formula}

Tracking back to the transformations \eqref{e4.13}, \eqref{e4.25}, \eqref{e5.32}, and \eqref{e4.22}, we obtain that in the neighborhood of $\lambda=0$, $M(\lambda)$ is analytic and
\begin{align}
	M=M^{(2)}_DEM^{(\mathcal{Z})}T^{-\sigma_3}.
\end{align}
Then, it follows by the reconstruction formula \eqref{e2.46} that
\begin{align}\label{e5.74}
	q_{n-1}(t)=T_0\left(M^{(2)}_DEM^{(\mathcal{Z})}\right)_{1,2}(0,n,t),
\end{align}
where
\begin{align}
	T_0=e^{-\int_{\arg S_2}^{\arg S_1}\ln(1+\abs{r(e^{\ii\theta})}^2)\frac{\ddddd \theta}{2\pi}}\prod_{\lambda_l\in \mathcal{Z}_\xi^-}\abs{\lambda_l}^{2\alpha_l}.
\end{align}
In view of the estimates \eqref{e5.46}, \eqref{e5.66}, and \eqref{e5.73}, we obtain the asymptotic formula \eqref{e5.3} for \eqref{e5.74}.

\section{Soliton resolution  in sectors I and III}\label{s6}
The aim of this section is to carry out the analysis in sectors I and III.
In this vein, the stationary phase points are off the unit circle.
Due to this fact, the leading term in the asymptotic formula is determined by solitons only, and we call sectors I and III as the soliton sectors with the remaining constrained by $\oo(t^{-1})$.
Set the scalar function $T(\lambda)$
\begin{align}\label{e6.1}
	T(\lambda)=T(\lambda,\xi)=\begin{cases}
		\prod_{\lambda_{l}\in\mathcal{Z}_\xi^-}\left(\frac{\lambda-\lambda_{l}}{\lambda-\bar\lambda_{l}^{-1}}\right)^{\alpha_{l}}&\text{sector I},\\
		e^{-\oint_{\Sigma}\frac{\ln{1+\abs{r(\varsigma)}^2}}{\varsigma-\lambda}\frac{\ddddd \varsigma}{2\pi\ii}}\prod_{\lambda_{l}\in\mathcal{Z}_\xi^-}\left(\frac{\lambda-\lambda_{l}}{\lambda-\bar\lambda_{l}^{-1}}\right)^{\alpha_{l}}&\text{sector III}.
	\end{cases}
\end{align}
Since the analysis in sectors I and III are fairly similar, we concentrate only on the formulation for the former one.  % when $\xi>1$.

\begin{figure}
	\centering
	\begin{subfigure}{0.35\textwidth}
		\centering\includegraphics{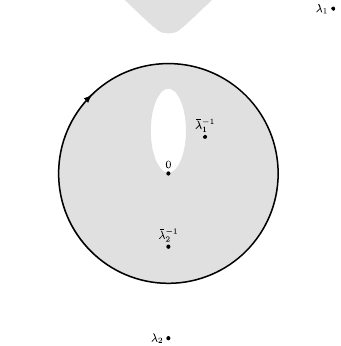}
		\subcaption{}
	\end{subfigure}
	\begin{subfigure}{0.35\textwidth}
		\centering\includegraphics{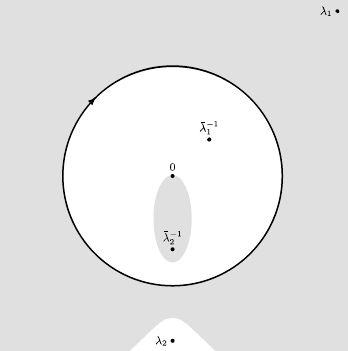}
		\subcaption{}
	\end{subfigure}
	\begin{subfigure}{0.35\textwidth}
		\centering\includegraphics{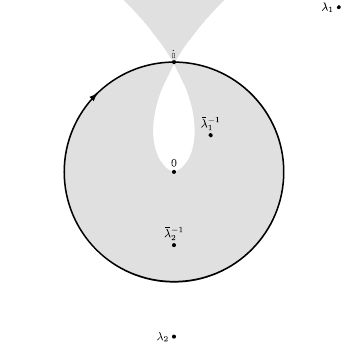}
		\subcaption{}
	\end{subfigure}
	\begin{subfigure}{0.35\textwidth}
		\centering\includegraphics{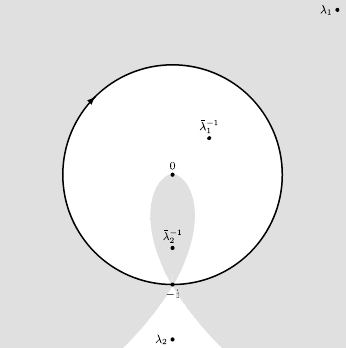}
		\subcaption{}
	\end{subfigure}
	\caption{Signature tables of $\re\phi$ for (A) $\xi<-1$, (B) $\xi>1$, (C) $\xi=-1$, and (D) $\xi=1$, where $\re\phi>0$ in the grayed regions, and $\re\phi<0$ otherwise.}\label{f6}
\end{figure}
\begin{figure}
	\centering
	\includegraphics[width=0.4\linewidth]{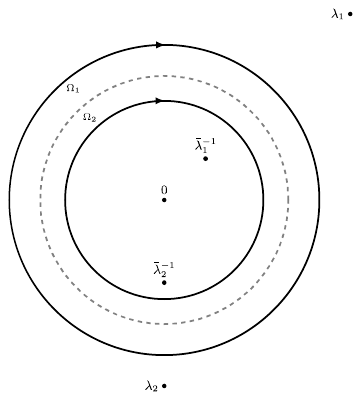}
	\caption{The jump contour $\Sigma^{(2)}=\Sigma_{\epsilon_0}\cup\Sigma_{\epsilon_0}^{-1}$ of $\bar\partial$-RH problem \ref{db13} along with the regions $\Omega_1$ and $\Omega_{2}$.}
	\label{f7}
\end{figure}

\subsection{Transform to a $\bar\partial$-RH problem}
Transform $M(\lambda)$ into
\begin{align}\label{e6.3}
	M^{(1)}(\lambda)=M(\lambda)T^{\sigma_3}(\lambda).
\end{align}
It is clear from RH problem \ref{r1} and \eqref{e6.1} that $M(\lambda)$ admits the RH problem below.
\begin{rhp}\label{r6.1}
	\
	\begin{itemize}
		\item $M^{(1)}(\lambda)$ is meromorphic in $\mathbb{C}\setminus\Sigma$.
		\item As $\lambda\to\infty$, we have
		$M^{(1)}(\lambda)= I+\oo(\lambda^{-1})$
		\item For $\lambda\in\Sigma$, we have $M^{(1)}_+(\lambda)=M^{(1)}_-(\lambda)V^{(1)}(\lambda)$, where
		\begin{align*}
			V^{(1)}(\lambda)=
				\left(\begin{matrix}
					1+\abs{r(\lambda)}^2&\overline{r(\lambda)}T^{-2}(\lambda)e^{\phi(\lambda)}\\r(\lambda)T^2(\lambda)e^{-\phi(\lambda)}&1
				\end{matrix}\right).
		\end{align*}
	\item The residue condition resembles \eqref{e4.17s}, which is expressed in terms of the scalar function
	$T(\lambda)$ in \eqref{e6.1}.
	\end{itemize}
\end{rhp}
Set the positive constant
\begin{align}
	\epsilon_0=\frac{\min\{\abs{\lambda_1}-1,\dots,\abs{\lambda_{l_0}}-1,\sqrt{\xi^2-1}+\xi-1\}}{2}.
\end{align}
In view of the signature table shown in FIGURE \ref{f6} and \eqref{e2.18}, it is clear that
\begin{align}\label{e6.2}
	&\Sigma_{\epsilon_0}=\{\abs{\lambda}=1+\epsilon_0\}\subset\{\re\phi>0\},\quad
	\Sigma_{\epsilon_0}^{-1}\subset\{\re\phi<0\},
	\end{align}
and that there is no discrete spectrum in the region between $\Sigma$ and $\Sigma_{\epsilon_0}$.
Consequently, it follows from the compactness of $\Sigma_{\epsilon_0}\cup\Sigma_{\epsilon_0}^{-1}$ that there is a positive constant $C$ such that
\begin{align}\label{e6.4}
	\inf_{\lambda\in\Sigma_{\epsilon_0}}\re\phi(\lambda,n,t)>Ct,\quad \sup_{\lambda\in\Sigma_{\epsilon_0}^{-1}}\re\phi(\lambda,n,t)<-Ct.
\end{align}
Define
\begin{align*}
	R_1(\lambda)=r(\lambda/\abs{\lambda}),\quad R_2(\lambda)=\overline{R_1(\bar\lambda^{-1})},
\end{align*}
and
\begin{equation}\label{e6.6}
	M^{(2)}(\lambda)=M^{(1)}(\lambda)\mathcal{R}(\lambda),\quad \mathcal{R}(\lambda)=\begin{cases}
		\left(\begin{matrix}
			1&0\\-R_1(\lambda)T^2(\lambda)e^{-\phi(\lambda)}&1
		\end{matrix}\right)&\lambda\in\Omega_1=\{\abs{\lambda}\in(1,1+\epsilon_0)\},\\
		\left(\begin{matrix}
			1&R_2(\lambda)T^{-2}(\lambda)e^{\phi(\lambda)}\\0&1
		\end{matrix}\right)&\lambda\in\Omega_2=\{\abs{\lambda}\in((1+\epsilon_0)^{-1},1)\},\\
		I&otherwise.
	\end{cases}
\end{equation}
It is clear that $M^{(2)}$ uniquely solves the following $\bar\partial$-RH problem, and the jump contour is $\Sigma^{(2)}=\Sigma_{\epsilon_0}\cup\Sigma_{\epsilon_0}^{-1}$ oriented as depicted in FIGURE \ref{f7}.
\begin{drhp}\label{db13}
	\
	\begin{itemize}
		\item $M^{(2)}(\lambda)$ is continuous in $\mathbb{C}\setminus\{\Sigma^{(2)}\cup\mathcal{Z}\cup\overline{\mathcal{Z}^{-1}}\}$.
		\item As $\lambda\to\infty$, we have $M^{(2)}(\lambda)=I+\oo(\lambda^{-1})$.
		\item For $\lambda\in\Sigma^{(2)}$, we have
		\begin{align}
			M^{(2)}_+(\lambda)=M^{(2)}_-(\lambda)V^{(2)}(\lambda),
		\end{align}
		where
		\begin{align}\label{e6.5}
			V^{(2)}(\lambda)=\begin{cases}
				\left(\begin{matrix}
					1&0\\R_1(\lambda)T^2(\lambda)e^{-\phi(\lambda)}&1
				\end{matrix}\right)&\lambda\in\Sigma_{\epsilon_0},\\
				\left(\begin{matrix}
					1&R_2(\lambda)T^{-2}(\lambda)e^{\phi(\lambda)}\\0&1
				\end{matrix}\right)&\lambda\in\Sigma_{\epsilon_0}^{-1}.
			\end{cases}
		\end{align}
		\item The residue condition is the same as that of RH problem \ref{r4.2} with $M^{(1)}(\lambda)$ replaced by $M^{(2)}(\lambda)$.
		\item For $\lambda\in\mathbb{C}$, we have $\bar\partial M^{(2)}(\lambda)=M^{(2)}(\lambda)\bar\partial\mathcal{R}(\lambda)$.
	\end{itemize}
\end{drhp}

\subsection{Factorization of the $\bar\partial$-RH problem}

As we have done in the case $\xi\in(-1,1)$, factorize the solution $M^{(2)}(\lambda)$ into the product of solutions of a $\bar\partial$ problem and an RH problem,
\begin{align}\label{e6.11}
	M^{(2)}(\lambda)=M^{(2)}_D(\lambda)M^{(2)}_{RH}(\lambda).
\end{align}

\begin{rhp}\label{r6.2}
	$M^{(2)}_{RH}(\lambda)$ solves $\bar\partial$-RH problem \ref{db13} with $\bar\partial M^{(2)}_{RH}(\lambda)\equiv0$ over the complex plane $\mathbb{C}$.
\end{rhp}
\begin{dbarproblem}\label{db6.3}
	\
	\begin{itemize}
		\item $M^{(2)}_D(\lambda)$ is continuous in $\mathbb{C}$.
		\item As $\lambda\to\infty$, we have $M^{(2)}_D(\lambda)\sim I+\oo(\lambda^{-1})$.
		\item For $\lambda\in\mathbb{C}$, we have $\bar\partial M^{(2)}_D(\lambda)=M^{(2)}_D(\lambda)\tilde{\mathcal{R}}(\lambda)$, where $\tilde{\mathcal{R}}(\lambda)=M^{(2)}_{RH}(\lambda)\bar\partial\mathcal{R}(\lambda)(M^{(2)}_{RH}(\lambda))^{-1}$.
		
	\end{itemize}
\end{dbarproblem}

\subsubsection{Solvability of $M^{(2)}_{RH}(\lambda)$}
In this part, we discuss the solvability of $M^{(2)}_{RH}(\lambda)$.
$M^{(2)}_{RH}(\lambda)$ is well-defined as that in Section \ref{s5}.
Exactly, with a little abuse of notation, for $\xi>1$, denote $M^{(\mathcal{Z})}(\lambda)=M^{(\mathcal{Z})}(\lambda,n,t)$ as the solution of RH problem \ref{r6.2} with $V^{(2)}(\lambda)\equiv I$, where $(\mathcal{Z},\pp^T)$ is the discrete spectral data.
By the techniques of small norm RH problem theory in Section \ref{s531}, it is clear that for some positive constant $C$,
\begin{align}\label{e6.12}
	M^{(\mathcal{Z})}(0)=(I+\oo(e^{-Ct}))M^{(\mathcal{Z}_\xi)}(0),
\end{align}
where $M^{(\mathcal{Z}_\xi)}(\lambda)=M^{(\mathcal{Z}_\xi)}(\lambda,n,t)$ is defined for $\xi<-1$ as
	\begin{align}
		M^{(\mathcal{Z}_\xi)}(\lambda)=\begin{cases}
			I&\mathcal{Z}_\xi=\emptyset,\\
			M^{(\mathcal{Z}_\xi,\pp^T_\xi)}(\lambda)&\mathcal{Z}_\xi\ne\emptyset,
		\end{cases}\quad\mathcal{Z}_\xi=\mathcal{Z}\cap\{\re\phi=0\},
	\end{align}
and $M^{(\mathcal{Z}_\xi,\pp^T_\xi)}(\lambda)=M^{(\mathcal{Z}_\xi,\pp^T_\xi)}(\lambda,n,t)$ is as defined in Remark \ref{r3.7} with the discrete spectral data $(\mathcal{Z}_\xi,\pp^T_\xi)$ that is the restriction of $(\mathcal{Z},\pp^T)$ on $\mathcal{Z}_\xi$.
Utilizing the techniques in Section \ref{s5}, we obtain that $M^{(2)}_{RH}(\lambda)$ admits the form
\begin{align}\label{e6.14}
	M^{(2)}_{RH}(\lambda)=E(\lambda)M^{(\mathcal{Z})}(\lambda),
\end{align}
where $E(\lambda)$ is the solution of the small-norm RH problem:
\begin{rhp}\label{r6.4}
	\
	\begin{itemize}
		\item $E(\lambda)$ is analytic in $\mathbb{C}\setminus\Sigma^{(2)}$.
		\item As $\lambda\to\infty$, we have $E(\lambda)\sim I+\oo(\lambda^{-1})$.
		\item On $\Sigma^{(2)}$, we have $E_+(\lambda)=E_-(\lambda)V^E(\lambda)$, where $V^E(\lambda)=  M^{(\mathcal{Z})}(\lambda)V^{(2)}(\lambda)\left(  M^{(\mathcal{Z})}(\lambda)\right)^{-1}$.
	\end{itemize}
\end{rhp}
\noindent By the boundedness of $  M^{(\mathcal{Z})}(\lambda)$ and $\left(  M^{(\mathcal{Z})}(\lambda)\right)^{-1}$ on $\Sigma^{(2)}$, recalling \eqref{e6.4} and \eqref{e6.5}, ones obtain that as $t\to+\infty$,
\begin{align}\label{e6.13}
	\norm{V^E-I}_{L^\infty(\Sigma^{(2)})}\lesssim e^{-Ct}.
\end{align}
With the small norm theory for RH problems, it is readily seen from \eqref{e6.13} that RH problem \ref{r6.4} is uniquely solved and as $t\to+\infty$,
\begin{align}\label{e6.16}
	E(0)=I+\oo(e^{-Ct}).
\end{align}

Thus, we have proved the unique solvability of $M^{(2)}_{RH}$, and by \eqref{e6.14}, it is also invertible.
It follows that $M^{(2)}_D$ is also well-defined
\begin{align}
	M^{(2)}_D=M^{(2)}\left(M^{(2)}_{RH}\right)^{-1},
\end{align}
and in the following we discuss the solution $M^{(2)}_D$.

\subsubsection{Estimate on the $\bar\partial$ problem}
In this part, we concentrate on some estimates of $M^{(2)}_D(\lambda)$.
We assert that when $t\to+\infty$, $M^{(2)}_D(0)$ decays to the identity matrix like that it have done for $\xi\in(-1,1)$, but the decaying rate is $\oo(t^{-1})$, that is,
\begin{align}\label{e6.19}
	M^{(2)}_D(0)= I+\oo(t^{-1}).
\end{align}
$\bar\partial$ problem \ref{db6.3} is equivalent to the identity,
\begin{align}\label{e6.20}
	M^{(2)}_D(\lambda)=((\identity-S)^{-1}I)(\lambda),\quad Sf(\lambda)=\iint_{\mathbb{C}}\frac{[f\tilde{\mathcal{R}}](\varsigma)}{\varsigma-\lambda}\frac{\ddddd L(\varsigma)}{\pi}.
\end{align}
For the integral operator $S$, we also have the proposition.
\begin{proposition}\label{p6.10}
	Define the integral operator $S$ as \eqref{e6.20}.
	If $r\in H^1$, then $S$ is bounded on $L^\infty(\mathbb{C})$, and as $t\gg0$, the bound is dominated by $t^{-\frac{1}{2}}$, that is,
	\begin{align}
		\norm{S}_{L^\infty_{op}(\mathbb{C})}\lesssim t^{-\frac{1}{2}}.
	\end{align}
\end{proposition}
\begin{proof}
	For $\lambda\ne0$ and $f\in L^\infty(\mathbb{C})$, by the boundedness of $M^{(\mathcal{Z})}(\lambda)$, $(M^{(\mathcal{Z})}(\lambda))^{-1}$, and $T(\lambda)$ on $\Sigma^{(2)}$, it is similar to the estimates in \eqref{e4.54} and \eqref{e5.8} that
\begin{equation}\label{e6.23}
\begin{split}
\abs{\iint_{\Omega_1}\frac{[f\tilde{\mathcal{R}}](\varsigma)}{\varsigma-\lambda}\frac{\ddddd \varsigma}{\pi}}&\lesssim\norm{f}_{L^\infty(\mathbb{C})}\iint_{\Omega_1}\frac{\abs{\bar\partial R_1e^{-\phi}}(\varsigma)}{\abs{\varsigma-\lambda}}\frac{\ddddd L(\varsigma)}{\pi}\\
&=\norm{f}_{L^\infty(\mathbb{C})}\int_{1}^{1+\epsilon_0}\int_{-\pi}^{\pi}\frac{\abs{\partial_\theta \left(r(e^{\ii\theta})\right)}e^{-t\left((\rho-\rho^{-1})\sin\theta+2\xi\ln\rho\right)}}{2\pi\abs{\rho e^{\ii\theta}-\lambda}}\ddddd\theta\ddddd\rho\\
&\le\norm{f}_{L^\infty(\mathbb{C})}\norm{r}_{H^1}\int_{1}^{1+\epsilon_0}e^{-t\left((\rho-\rho^{-1})+2\xi\ln\rho\right)}\frac{\ddddd\rho}{2\pi}\left(\int_{-\pi}^{\pi}\abs{\rho e^{\ii\theta}-\lambda}^{-2}\ddddd\theta\right)^{\frac{1}{2}}\\
&\lesssim\norm{f}_{L^\infty(\mathbb{C})}\int_{1}^{1+\epsilon_0}\frac{e^{-2t(\rho-1)\left(\xi\frac{\ln(1+\epsilon_0)}{\epsilon_0}+\frac{1}{2}\right)}}{\abs{\rho-\abs{\lambda}}^{\frac{1}{2}}}\ddddd\rho\lesssim t^{-\frac{1}{2}}\norm{f}_{L^\infty(\mathbb{C})}.
\end{split}
\end{equation}
	and similarly,
	\begin{align}
		\abs{\iint_{\Omega_2}\frac{[f\tilde{\mathcal{R}}](\varsigma)}{\varsigma-\lambda}\frac{\ddddd \varsigma}{\pi}}\lesssim t^{-\frac{1}{2}}\norm{f}_{L^\infty(\mathbb{C})},
	\end{align}
	thus, we obtain that
	\begin{align}\label{e6.21}
		\norm{S}_{L^\infty_{op}}\lesssim t^{-\frac{1}{2}},
	\end{align}
	that is, the integral operator is bounded on $L^\infty(\mathbb{C})$ and the bound decays to zero for $t\gg0$.
\end{proof}

Now, we focus on the estimate of $M^{(2)}_D(0)$,
\begin{equation}\label{e6.28}
	\begin{split}
		M^{(2)}_D(0)=I+\iint_{\mathbb{C}}[M^{(2)}_D\tilde{\mathcal{R}}](\varsigma)\varsigma^{-1}\frac{\ddddd L(\varsigma)}{\pi}.
	\end{split}
\end{equation}
Taking $f(\varsigma)=M^{(2)}_D(\varsigma)$ and $\lambda=0$ in \eqref{e6.23}, we obtain that
\begin{align}
	\iint_{\Omega_1}[M^{(2)}_D\tilde{\mathcal{R}}](\varsigma)\varsigma^{-1}\frac{\ddddd L(\varsigma)}{\pi}\lesssim t^{-1},
\end{align}
and utilizing the same techniques, we similarly estimate the integral on $\mathbb{C}$
\begin{align}
	\iint_{C}[M^{(2)}_D\tilde{\mathcal{R}}](\varsigma)\varsigma^{-1}\frac{\ddddd L(\varsigma)}{\pi}\lesssim t^{-1},
\end{align}
which means that
\begin{align}\label{e6.26}
	M^{(2)}_D(0)=I+\oo(t^{-1}).
\end{align}

\subsection{Asymptotic formula}
Tracking back to the transformations \eqref{e6.3}, \eqref{e6.6}, \eqref{e6.11}, and \eqref{e6.14}, it is readily seen that in the neighborhood of $\lambda=0$, $M(\lambda)$ is analytic and reads
\begin{align}
	M=M^{(2)}_DEM^{(\xi)}T^{-\sigma_3}.
\end{align}
Then, it follows by the reconstruction formula \eqref{e2.46}  that
\begin{align}\label{e6.22}
	q_n(t)=T_0\left(M^{(2)}_DEM^{(\xi)}\right)_{1,2}(0,n+1,t),
\end{align}
where
\begin{align*}
	T_0=
		\prod_{\lambda_l\in\mathcal{Z^-_\xi}}\abs{\lambda_l}^{2\alpha_l}.
\end{align*}
In view of \eqref{e6.12}, \eqref{e6.16}, and \eqref{e6.26}, we obtain the asymptotic formula \eqref{e1.15} for \eqref{e6.22} in sector I.
Similarly, we obtain the asymptotic formula in sector II with
\begin{align}
	T_0=e^{-\int_{0}^{2\pi}\ln(1+\abs{r(e^{\ii\theta})}^2)\frac{\ddddd\theta}{2\pi}}\prod_{\lambda_l\in\mathcal{Z^-_\xi}}\abs{\lambda_l}^{2\alpha_l}.
\end{align}

\section{Painlev\'e-type asymptotics in 1st and 2nd transition zones}\label{s7}

The aim of this section is to carry out the Painlev\'e-type asymptotics  in 1st and 2nd transition zones.
In this vein, the stationary phase points are fairly near $\pm\ii$, that is, as $\xi \to \pm 1$, we have
\begin{align*}
	S_1, S_2 \to  \ii,\, \text{as} \ \xi \to -1;\quad S_1, S_2 \to  -\ii,\, \text{as} \ \xi \to 1.
\end{align*}
Due to this fact, the leading term includes the solitons and the Painlev\'e-type asymptotics, and the remaining is constrained by $\oo(t^{-\frac{1}{2}})$.
The corresponding scalar function $T(\lambda)$ reads
\begin{align}\label{T}
	T(\lambda)=T(\lambda,\xi)=\begin{cases}
		\prod_{\lambda_{l}\in\mathcal{Z}_\xi^-}\left(\frac{\lambda-\lambda_{l}}{\lambda-\bar\lambda_{l}^{-1}}\right)^{\alpha_{l}}&\text{1st transition zone},\\
	e^{-\oint_{\Sigma}\frac{\ln{1+\abs{r(\varsigma)}^2}}{\varsigma-\lambda}\frac{\ddddd s}{2\pi \ii}}	\prod_{\lambda_{l}\in\mathcal{Z}_\xi^-}\left(\frac{\lambda-\lambda_{l}}{\lambda-\bar\lambda_{l}^{-1}}\right)^{\alpha_{l}}	&\text{2nd transition zone}.
	\end{cases}
\end{align}
Since the analysis in 1st transition zone and 2nd transition zone are similar, we focus only on formulating the former one.

\subsection{Transform to a $\bar\partial$-RH problem}
Following the similar analysis in Subsection \ref{s1m1}, we take a transformation:%这是的T的上标-号已经改了，后面的还没改
\begin{align}\label{t1}
	M^{(1)}(\lambda)=M(\lambda)T^{\sigma_3}(\lambda),
\end{align}
which is the solution for the following RH problem.
\begin{rhp}\label{m1}
	\
	\begin{itemize}
		\item $M^{(1)}(\lambda)$ is meromorphic in $\mathbb{C}\setminus\Sigma$.
		\item As $\lambda\to\infty$, we have $M^{(1)}(\lambda)= I+\oo(\lambda^{-1})$.
		\item For $\lambda\in\Sigma$,  we have $M^{(1)}_+(\lambda)=M^{(1)}_-(\lambda)V^{(1)}(\lambda),$
		where
		\begin{align*}
			V^{(1)}(\lambda)=
			\left(\begin{matrix}
				1+\abs{r(\lambda)}^2&\bar{r}(\lambda)T^{-2}(\lambda)e^{\phi(\lambda)}\\r(\lambda)T^{2}(\lambda)e^{-\phi(\lambda)}&1
			\end{matrix}\right).
		\end{align*}
		\item The residue condition resembles \eqref{e4.17s}, which is expressed in terms of the scalar function
		$T(\lambda)$ in \eqref{T}.
		
	\end{itemize}
\end{rhp}

To obtain a mixed $\bar\partial$-RH problem,  we introduce the continuous function
\begin{align}\label{R1}
	&R_1(\lambda)=
	r(\lambda/\abs{\lambda}),\quad R_2(\lambda)=\overline{R_1(\bar\lambda^{-1})},
\end{align}
and transform $M^{(1)}(\lambda)$ into
\begin{align}\label{m2m1}
	M^{(2)}(\lambda)=M^{(1)}(\lambda)\mathcal{R}(\lambda),\quad\mathcal{R}(\lambda)=\begin{cases}
		\left(\begin{matrix}
			1&0\\-R_1(\lambda)T^{2}(\lambda)e^{-\phi(\lambda)}&1
		\end{matrix}\right)&\lambda\in\Omega_1, \\
		\left(\begin{matrix}
			1&R_2(\lambda)T^{-2}(\lambda)e^{\phi(\lambda)}\\0&1
		\end{matrix}\right)&\lambda\in\Omega_2, \\
		I&\text{elsewhere}.\\
	\end{cases}
\end{align}
Then, it is easily seen from RH problem \ref{m1} and \eqref{m2m1} that $M^{(2)}(\lambda)$ satisfies the following $\bar\partial$-RH problem with the jump contour depicted in FIGURE \ref{fs2}.

\begin{figure}
	\centering\includegraphics[width=0.4\textwidth]{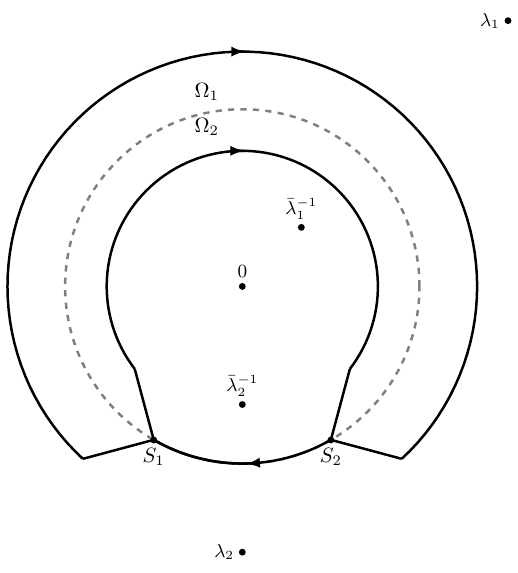}
	\caption{The jump contour $\Sigma^{(2)}=\Sigma^{(2)}_1\cup\Sigma^{(2)}_2$ and regions $\Omega_1$ and $\Omega_2$, where $\Sigma^{(2)}_1=\partial\Omega_1\cap D_+$  and $\Sigma^{(2)}_2=\partial\Omega_2\cap D_-$.}\label{fs2}
\end{figure}

\begin{drhp}\label{pm2}
	\
	\begin{itemize}
		\item $M^{(2)}(\lambda)$ is continuous in $\mathbb{C}\setminus\{\Sigma^{(2)}\cup\mathcal{Z}\cup\overline{\mathcal{Z}^{-1}}\}$.
		\item As $\lambda\to\infty$, we have $	M^{(2)}(\lambda)= I+\oo(\lambda^{-1})$.
		\item For $\lambda \in \Sigma^{(2)}$, we have $M^{(2)}_+(\lambda)=M^{(2)}_-(\lambda)V^{(2)}(\lambda)$,
		where
		\begin{align}\label{pv2}
			V^{(2)}(\lambda)=\begin{cases}
				\left(\begin{matrix}
					1&0\\R_1(\lambda)T^{2}(\lambda)e^{-\phi(\lambda)}&1
				\end{matrix}\right)&\lambda\in\Sigma^{(2)}_1, \\
				\left(\begin{matrix}
					1&R_2(\lambda)T^{-2}(\lambda)e^{\phi(\lambda)}\\0&1
				\end{matrix}\right)&\lambda\in\Sigma^{(2)}_2.		
			\end{cases}
		\end{align}
		\item  The residue condition is the same as that of RH problem \ref{m1} with $M^{(1)}(\lambda)$ replaced by $M^{(2)}(\lambda)$.
		\item  For $\lambda \in \mathbb{C}$,  we  have the $\bar\partial$-derivative relation $\bar\partial M^{(2)}(\lambda)=M^{(2)}(\lambda)\bar\partial\mathcal{R}(\lambda)$, where
		\begin{align}
			\bar\partial \mathcal{R}(\lambda)=\begin{cases}
				\left(\begin{matrix}
					0&0\\-\bar\partial R_1(\lambda)T^{2}(\lambda)e^{-\phi(\lambda)}&0
				\end{matrix}\right)&\lambda\in\Omega_1, \\
				\left(\begin{matrix}
					0&\bar\partial R_2(\lambda)T^{-2}(\lambda)e^{\phi(\lambda)}\\0&0
				\end{matrix}\right)&\lambda\in\Omega_2, \\
				\mathbf{0}& \text{elsewhere},\\
			\end{cases}
		\end{align}
		
	\end{itemize}
\end{drhp}

Rewriting $\lambda\in\mathbb{C}$ in the polar coordinates: $\lambda=\rho e^{\ii\theta}$ and
utilizing the formula $\bar\partial=e^{i\theta}(\partial_\rho+\ii\rho^{-1}\partial_\theta)/2$, we have
\begin{align}\label{pr1}
	\bar\partial R_1(\lambda)=
	\frac{\ii}{2}\rho^{-1}e^{\ii\theta}\partial_\theta \left(r(e^{\ii\theta})\right),
\end{align}
which can be readily seen that the $L^1$-norm of $\bar\partial R_1(\lambda)$ on any circle $\{\lambda: \abs{\lambda}=C\, \text{with}\, C>0\}$ is well dominated by the norm of $r\in H^1$.%贺定义

\subsection{Factorization of the $\bar\partial$-RH problem}
The above $\bar\partial$-RH problem can be decomposed into a pure RH problem and a pure  $\bar\partial$-problem, that is,
\begin{align}\label{m2mdmrh}
	M^{(2)}(\lambda)=M^{(2)}_D(\lambda)M^{(2)}_{RH}(\lambda).
\end{align}

\begin{rhp}
	$M^{(2)}_{RH}(\lambda)$ solves the $\bar\partial$-RH problem \ref{pm2} with $\bar\partial\mathcal{R}\equiv 0$.
\end{rhp}
\begin{dbarproblem}\label{m2d}
	\
	\begin{itemize}
		\item $M^{(2)}_{D}(\lambda)$ is continuous in $\mathbb{C}$.
		\item As $z\to \infty$, we have $M^{(2)}_{D}(\lambda)=I+\oo(\lambda^{-1})$.
		\item For $\lambda\in\mathbb{C}$, we have $M^{(2)}_{D}(\lambda)$ satisfies $\bar\partial M^{(2)}_D(\lambda)=M^{(2)}_D(\lambda)\tilde{\mathcal{R}}(\lambda)$, where $\tilde{\mathcal{R}}(\lambda)=M^{(2)}_{RH}(\lambda)\bar\partial\mathcal{R}(\lambda)(M^{(2)}_{RH}(\lambda))^{-1}$.
		
	\end{itemize}
	
\end{dbarproblem}

\subsubsection{Solvability of $M^{(2)}_{RH}(\lambda)$}

In this part, we discuss the solvability of $M^{(2)}_{RH}(\lambda)$.
Let $U = \{\lambda: |\lambda+\ii| \le c_0 \}$
be a small disk around $\lambda =\ii$,  where
\begin{align}\label{c0}
	c_0 := \min \{\frac{1}{2}, 2 |S_j+\ii|t^{\frac{1}{30}} \}.
\end{align}
We intend to construct the solution $M^{(2)}_{RH}(\lambda)$ in the following form
\begin{align}\label{pm2rh}
	M^{(2)}_{RH}(\lambda)=\begin{cases}
		E(\lambda)M^{(\mathcal{Z})}(\lambda)&\lambda\in\mathbb{C}\setminus  U,\\
		E(\lambda)M^{(\mathcal{Z})}(\lambda)M^{(U)}(\lambda)&\lambda\in  U,
	\end{cases}
\end{align}
where $M^{(\mathcal{Z})}(\lambda)$ is the solution of RH problem \ref{r5.6}, $M^{(U)}(\lambda)$ solves RH problem \ref{pmu}, and $E(\lambda)$ admits the RH problem \ref{E}.

\begin{rhp}\label{pmu}
	\
	\begin{itemize}
		\item $M^{(U)}(\lambda)$ is analytic in $\mathbb{C}\setminus\Sigma^{(U)}$, where $\Sigma^{(U)} =\bigcup_{j=1}^4\Sigma^{(U)}_j$ with $$\Sigma^{(U)}_1=S_1(1+e^{\frac{\ii\pi}{4}}), \
		\Sigma^{(U)}_2=S_2(1+e^{\ii(\pi-\frac{\pi}{4})}),\
		\Sigma^{(U)}_3=S_1(1+e^{-\frac{\ii\pi}{4}}),\  \Sigma^{(U)}_4=S_2(1+e^{-\ii(\pi-\frac{\pi}{4})}).$$
		\item As $\lambda\to\infty$, we have $M^{(U)}(\lambda)= I+\oo(\lambda^{-1})$.
		\item For $\lambda \in \Sigma^{(U)}$, we have  $M^{(U)}_+(\lambda)=M^{(U)}_-(\lambda)V^{(U)}(\lambda)$,
		where
		\begin{align}\label{pmmu}
			&V^{(U)}(\lambda)=\begin{cases}
				\left(\begin{matrix}
					1&0\\r(-\ii) T^{2}(-\ii) e^{-\phi(-\ii)+\tilde{\phi}(\lambda)}&1
				\end{matrix}\right)&\lambda\in\Sigma^{(U)}_j, \ j=1,3,\\
				\left(\begin{matrix}
					1&\overline{r(-\ii)} T^{-2}(-\ii) e^{\phi(-\ii)-\tilde{\phi}(\lambda)}\\0&1
				\end{matrix}\right)&\lambda\in\Sigma^{(U)}_{j+2},\ j =2,4,
			\end{cases}
		\end{align}
		with
		\begin{align}\label{tildephi}		
			\tilde{\phi}	(\lambda)=\tilde\phi(\lambda,n,t)=
			-2 \ii t (\xi-1) (\lambda+\ii) -\frac{\ii}{3}t(\lambda+\ii)^3.
		\end{align}
	\end{itemize}
\end{rhp}

\begin{rhp}\label{E}
	\
	\begin{itemize}
		\item $E(\lambda)$ is analytic for $\lambda \in \mathbb{C}\setminus\Sigma^E$, where
		$\Sigma^E=\partial U\cup\Sigma^{(2)}$.
		\item As $\lambda\to\infty$, $E(\lambda)= I+\oo(\lambda^{-1})$.
		\item For  $\lambda \in \Sigma^E$, we have $	E_+(\lambda)=E_-(\lambda)V^E(\lambda)$,
		where
		\begin{align}\label{ve}
			V^E(\lambda)=\begin{cases}
				M^{(\mathcal{Z})}(\lambda)V^{(2)}(\lambda)(M^{(\mathcal{Z})}(\lambda))^{-1}&\lambda\in\Sigma^{(2)}\setminus U,\\
				M^{(\mathcal{Z})}(\lambda)M^{(U)}(\lambda)V^{(2)}(\lambda)(V^{(U)}(\lambda))^{-1}(M^{(\mathcal{Z})}(\lambda)M^{(U)}(\lambda))^{-1}&\lambda\in\Sigma^{(2)}\cap U,\\
				M^{(\mathcal{Z})}(\lambda)M^{(U)}(\lambda)(M^{(\mathcal{Z})}(\lambda))^{-1}&\lambda\in\partial U.
			\end{cases}
		\end{align}
	\end{itemize}
\end{rhp}

From \eqref{pm2rh}, in order to obtain the solution $M^{(2)}_{RH}(\lambda)$, we need to construct the solution  $	M^{(\mathcal{Z})}(\lambda)$, $M^{(U)}(\lambda)$ and $E(\lambda)$.
Referring back to the analysis in Section \ref{s531}, we observe that the solution  $M^{(\mathcal{Z})}(\lambda)$ exists and is well-defined for $t\gg0$.

To construct the corresponding RH problem for $M^{(U)}(\lambda)$,   we observe that for $\lambda\in  U$ and $t$ large enough,
\begin{align}\label{tiks}
	\tilde{\phi}(\lambda) =  \frac{8}{3} \ii k^3 + 2 \ii s k,
\end{align}
where
\begin{align} \label{ss}
s =
-2  (\xi-1)t^{\frac{2}{3}}
\end{align}
parameterizes the space-time sectors, and
\begin{align} \label{k}
k=- \frac{1}{2 }t^{\frac{1}{3}} (\lambda-\ii)
\end{align}
is a scaled spectral parameter.  \eqref{tiks} implies us to work in the $k$-plane.
Then,  RH problem \ref{pmu} can be given by the Painlev\'{e} \uppercase\expandafter{\romannumeral2} RH model in Appendix \ref{appx}
\begin{align}
	M^{(U)}(k)=
	e^{-\ii (\frac{\varphi}{2}-\frac{\pi}{4})\sigma_3} M^P(k;|r( -\ii )|)e^{\ii (\frac{\varphi}{2}-\frac{\pi}{4})\sigma_3},
\end{align}
where
\begin{align}
	&\varphi=
\ii	\phi(-\ii ) +\arg  r(-\ii) +2  \alpha_{l} \sum_{\lambda_{l}\in\mathcal{Z}_\xi^-} \arg(\ii+\lambda_{l}),\label{varphi}
\end{align}
and $M^P$ is the solution of RH problem \ref{1modp2}. By \eqref{posee}, as $k \to \infty$, we have
\begin{equation*}
	M^{(U)}(k) = I + \frac{M_1^{(U)}}{k} + \oo(k^{-2}),
\end{equation*}
where
\begin{align}\label{m1u}
	M_1^{(U)} =
   \left(\begin{array}{ccc} 	-\frac{\ii}{2} \int_s^\infty v^2(\varsigma) \ddddd  \varsigma & \frac{v(s)}{2}e^{-\ii \varphi} \\ \frac{v(s)}{2}e^{\ii \varphi} &	\frac{\ii}{2} \int_s^\infty v^2(\varsigma) \ddddd \varsigma \end{array}  \right).
\end{align}

Based on the solvability of $M^{(\mathcal{Z})}(\lambda)$ and $M^{(U)}(\lambda)$, we come to the error part $E(\lambda)$.
From \eqref{pv2} and the signature table of $\re \phi(\lambda)$ in Figure \ref{f6}, it is readily seen that for some positive constant $\epsilon$
\begin{align}\label{vout}
	\abs{V^{(2)}(\lambda)-I}\le e^{-\epsilon t},\quad \lambda\in\Sigma^{(2)}\setminus U.
\end{align}
For $\lambda \in \Sigma^{(2)}_1\cap U$, by \eqref{pm2} and \eqref{pmmu}, we have
\begin{align}\label{v2vu}
	\abs{V^{(2)}(V^{(U)})^{-1}-I}\lesssim\abs{V^{(2)}-V^{(U)}} \lesssim \abs{R_1 T^{2} e^{-\phi} - r(-\ii)T^{2}(-\ii)e^{-\phi(-\ii)+\tilde{\phi}}  )}.
\end{align}
Then, after some calculations, we obtain
\begin{align}\label{pestr}
	\abs{R_1 T^{2} e^{-\phi} - r(-\ii)T^{2}(-\ii)e^{-\phi(-\ii)+\tilde{\phi}} )} \lesssim \abs{R_1- r(-\ii)} + \abs{e^{-\phi}-e^{-\phi(-\ii)+\tilde{\phi}}}.
\end{align}
To proceed, we observe that for $\lambda \in U$ and $t$ large enough,
\begin{align}\label{pexpphi}
	\phi = \phi(-\ii) -\tilde{\phi}+ \oo (t(\lambda+\ii)^4)=
	\phi(-\ii) -\frac{8}{3} \ii k^3 - 2 \ii s k +\oo(t^{-\frac{1}{3}}k^4 ),
\end{align}
where $s$ and $k$  are defined in \eqref{ss}-\eqref{k}.
The expansion of  $\phi$ in \eqref{pexpphi} invokes us to work in the $k$-plane. Under the change of variable \eqref{k}, it is easily seen that the jump line under the variable $k$ is
\begin{align}
	\tilde{\Sigma} := \bigcup_{j=1}^4\tilde{\Sigma}_j \cup \wideparen{k_1k_2},
\end{align}
where $k_j=-\frac{1}{2} t^{\frac{1}{3}} (S_j+\ii) $, $\wideparen{k_1k_2} = \{ k: 	k=
-\frac{1}{2} t^{\frac{1}{3}} (\lambda+\ii), \ \lambda \in  \wideparen{S_1S_2} \}$, and
\begin{align*}
	&\tilde{\Sigma}_1 = \left\{k: k=k_1(1+l e^{\frac{\ii\pi}{4}}), 0 \le l \le c_0 \left( \frac{ t}{8}\right)^{\frac{1}{3}} \right\},  \tilde{\Sigma}_2 = \left\{k: k=k_2(1+l e^{\ii(\pi-\frac{\pi}{4})}), 0 \le l \le c_0 \left( \frac{ t}{8}\right)^{\frac{1}{3}} \right\},\\
	& \tilde{\Sigma}_3 = \left\{k: k=k_1(1+l e^{-\ii(\pi-\frac{\pi}{4})}), 0 \le l \le c_0 \left( \frac{ t}{8}\right)^{\frac{1}{3}} \right\},  \tilde{\Sigma}_4 = \left\{k: k=k_2(1+l e^{-\ii\frac{\pi}{4}}), 0 \le l \le c_0 \left( \frac{ t}{8}\right)^{\frac{1}{3}} \right\}.
\end{align*}

%First, from \eqref{pexpphi} and \eqref{k},
%take
%\begin{align}\label{lambda}
%	\lambda=
%	S_j+\left( \frac{\ii t}{8}\right)^{-\frac{1}{3}} k, \quad j=1,2.
%\end{align}
Then, recalling that $r \in H^1$,
by \eqref{R1} and \eqref{tildephi}, for $k \in \tilde{\Sigma}_1$, we have
\begin{align}\label{R1r}
	\left|R_1\left(-\frac{1}{2}t^{-\frac{1}{3}}k+\ii\right) - r(-\ii)\right|  \lesssim \left|-\frac{1}{2}t^{-\frac{1}{3}} k\right|^{\frac{1}{2}} \lesssim t^{-\frac{1}{6}}\abs{k}^{\frac{1}{2}}.
\end{align}
Second, from  \eqref{pexpphi} and \eqref{k}, it follows that
\begin{align}\label{eee}
	\abs{e^{-\phi}-e^{-\phi(-\ii)+\tilde{\phi}}} \lesssim \abs{e^{\oo(t^{-1/3} k^4)}-1} \lesssim t^{-\frac{1}{3}}|k|^4.
\end{align}
Finally, since \eqref{c0}  holds,  we have
$$|k| \lesssim t^{\frac{1}{30}},$$
which together with \eqref{pestr},\eqref{R1r} and \eqref{eee}, gives us
\begin{equation}
	\abs{R_1 T^{2} e^{-\phi} - r(-\ii)T^{2}(-\ii)e^{-\phi(-\ii)+\tilde{\phi}} )} \lesssim t^{-\frac{3}{20}}.
\end{equation}
The similar estimates hold for the others jump line $\tilde{\Sigma}_j,\ j=2,3,4$, and $\wideparen{k_1k_2}$.
Thus, combining the above estimates with \eqref{v2vu}, it follows that
\begin{align}\label{v2vue}
	\abs{V^{(2)}(V^{(U)})^{-1}-I}\lesssim  t^{-\frac{3}{20}}.
\end{align}

Moreover,
using \eqref{stanp}, it implies that
\begin{align}\label{muu}
	\abs{(M^{(U)})^{-1}-I}\lesssim t^{-\frac{1}{30}},\quad \lambda\in\partial U.
\end{align}
Considering \eqref{ve}, \eqref{vout}, \eqref{v2vue} and \eqref{muu}, by the boundedness of $M^{(\mathcal{Z})}$ on $\Sigma^E$, we conclude that
\begin{align}
	\abs{V^E-I}\lesssim t^{-\frac{1}{30}}.
\end{align}
It then follows from the small-norm RH problem theory that there exists a unique solution to RH problem \ref{E} for large positive $t$. Moreover, according to \cite{beals1998scattering}, we have
\begin{align}\label{El}
	E(\lambda)=I+\int_{\Sigma^E}\frac{ w(\varsigma)(V^E(\varsigma)-I)}{\varsigma-\lambda}\frac{\ddddd \varsigma}{2\pi \ii},
\end{align}
where $w \in I + L^2(\Sigma^E)$ is the unique solution of the Fredholm-type equation
\begin{align}\label{w}
	w = I + \mathcal{C}_E w.
\end{align}
Here, $\mathcal{C}_E: L^2(\Sigma^E) \to L^2(\Sigma^E)$ is an integral operator defined by
$\mathcal{C}_E(f)(\lambda) = \mathcal{C}_- (f(V^E(\lambda)-I))$ with $\mathcal{C}_-$ being the Cauchy projection operator on $\Sigma^E$. Thus,
\begin{align}
	||\mathcal{C}_E||_{L^2_{op}(\Sigma^E)}\le ||\mathcal{C}_-||_{L^2_{op}(\Sigma^E)} ||V^E-I||_{L^\infty(\Sigma^E)} \lesssim t^{-\frac{1}{30}},
\end{align}
which implies that $\identity-\mathcal{C}_E$ is invertible for sufficiently large $t$. Thus, $w$ exists uniquely with
\begin{equation*}
	w = I + (\identity-\mathcal{C}_E)^{-1}( \mathcal{C}_E I).
\end{equation*}
On the other hand,  \eqref{w} can be rewritten as
\begin{equation*}
	w = I + \sum_{j=1}^4 \mathcal{C}_E^j I + (\identity-\mathcal{C}_E)^{-1} ( \mathcal{C}_E^5 I),
\end{equation*}
where for $j=1,\cdots,4$, we have the estimates
\begin{align}
	|| \mathcal{C}_E^j I ||_{L^2(\Sigma^E)} \lesssim t^{-(\frac{3}{20}+\frac{j}{30})}, \quad ||w-I - \sum_{j=1}^4 \mathcal{C}_E^j I ||_{L^2(\Sigma^E)} \lesssim t^{-(\frac{1}{6}+\frac{3}{20})}.
\end{align}
For later use, we focus on the estimate of $E(0)$. Substituting $\lambda=0$ into \eqref{El}, we have
\begin{align*}
	E(0)
	&= I + \oint_{\partial U}\frac{ w(\varsigma)(V^E(\varsigma)-I)}{\varsigma}\frac{\ddddd \varsigma}{2\pi \ii}+\int_{\Sigma^{(2)}}\frac{ w(\varsigma)(V^E(\varsigma)-I)}{\varsigma}\frac{\ddddd \varsigma}{2\pi \ii} \\
	&= I + \oint_{\partial U}\frac{V^E(\varsigma)-I}{\varsigma}\frac{\ddddd\varsigma}{2\pi \ii} +\int_{\Sigma^{(2)}}\frac{ V^E(\varsigma)-I}{\varsigma}\frac{\ddddd \varsigma}{2\pi \ii}+ \int_{\Sigma^E}\frac{(w(\varsigma)-I)(V^E(\varsigma)-I)}{\varsigma}\frac{\ddddd \varsigma}{2\pi \ii}\\
	&= I +  \oint_{\partial U}\frac{V^E(\varsigma)-I}{\varsigma}\frac{\ddddd \varsigma}{2\pi \ii} + \oo(t^{-\frac{1}{2}}).
\end{align*}
Substituting \eqref{m1u} into the above formula, we obtain
\begin{align}\label{e0}
	E(0) =
	I - 2 t^{-\frac{1}{3}}M^{(\mathcal{Z})}(-\ii) M_1^{(U)} M^{(\mathcal{Z})}(-\ii)^{-1}+ \oo (t^{  -\frac{1}{2}}).
\end{align}

Then, from \eqref{m2mdmrh}, we consider the function $M_D^{(2)}$.

\subsubsection{Estimate on the $\bar\partial$ problem}
Since we have already established the existence of $M^{(\mathcal{Z})}(\lambda)$, $M^{(U)}(\lambda)$ and $E(\lambda)$, the existence of  $M^{(2)}_{RH}(\lambda)$ follows naturally. In this part, we concentrate on  $M^{(2)}_D(\lambda)$.
The solution of this pure $\bar\partial$ problem \ref{m2d}  can be expressed in terms of
the integral equation
\begin{align}\label{pmd}
	M^{(2)}_D(\lambda)=I+\iint_\mathbb{C}\frac{M^{(2)}_D(\varsigma)\tilde{\mathcal{R}}(\varsigma)}{\varsigma-\lambda}\frac{\ddddd L(\varsigma)}{\pi},
\end{align}
where $L(\varsigma)$ is the  Lebesgue's measure on the complex plane.
Rewrite \eqref{pmd} as
\begin{align}\label{psf}
	(\identity-S)M^{(2)}_D=I,\quad Sf(\lambda)=\iint_{\mathbb{C}}\frac{f(\varsigma)\tilde{\mathcal{R}}(\varsigma)}{\varsigma-\lambda}\frac{\ddddd L(\varsigma)}{\pi}.
\end{align}
The following proposition shows that for sufficiently large $t$ the operator $S$ is small-norm, so that the resolvent operator $(\identity-S)^{-1}$ exists.
\begin{proposition}
	Define the integral operator $S$ as \eqref{psf}.
	If $r\in H^1$, then $S$ is bounded on $L^\infty(\mathbb{C})$, and as $t\to+\infty$, the bound of $S$ on $L^\infty(\mathbb{C})$ is dominated by $t^{-\frac{1}{6}}$, that is,
	\begin{align}\label{s}
		\norm{S}_{L^\infty \to L^\infty}\lesssim t^{-\frac{1}{6}}.
	\end{align}
\end{proposition}
\begin{proof}
	For any $f \in L^\infty$,
	\begin{equation*}
		\abs{Sf(\lambda)} \lesssim\norm{f}_{L^\infty(\mathbb{C})}\iint_{\mathbb{C}}\frac{|\tilde{R}(\varsigma)|}{|\varsigma-\lambda|}\frac{\ddddd L(\varsigma)}{\pi}.
	\end{equation*}
	Notice that $M^{(2)}_{RH}$, $(M^{(2)}_{RH})^{-1}$, $T$ and $T^{-1}$ are all bounded on $\Omega=\Omega_1 \cup \Omega_2$ as well as $\tilde{R}(\varsigma)$ is supported on $\varsigma \in \overline{\Omega}$. In what follows, we give a detailed estimate of the integral over $\Omega_1$.
	For the region $\Omega_1$, we write $\varsigma\in\Omega_1$ in the polar coordinates: $\varsigma=\rho e^{i\theta}$, $(\rho,\theta)\in(1,\rho_0)\times(\pi-\theta_\rho,\theta_\rho)$, where $\rho_0=\abs{1+\epsilon_0e^{-\ii\pi/4}}$, and $\rho e^{\ii\theta_\rho}$, $\rho e^{\ii(\pi-\theta_\rho)}$ are on the upper boundaries of $\Omega_1$: $S_1(1+e^{-\ii\pi/4}(0,\epsilon_0))$, $S_2(1+e^{\ii\pi/4}(0,\epsilon_0))$, respectively.
	For $\xi \in \mathcal{R}_-$, we have $S_1,\; S_2 \to -\ii$.
	Taking Taylor's expansion of \eqref{e2.18} at $\lambda= -\ii$, we see that
		\begin{align}\label{rephi}
	\re\phi(\rho e^{\ii\theta_\rho})=
	2t(\xi-1)(\rho-1)-t(\xi-1)(\rho-1)^2 +\frac{1}{3} \left(2\xi-3 \right)t (\rho-1)^3  + \oo(t(\rho-1)^4), \quad \rho e^{\ii\theta_\rho}\to -\ii.
	\end{align}
%	\begin{align}\label{rephi}
%		\re\phi(\rho e^{\ii\theta_\rho})=
%		2t(\xi+1)(\rho-1)-t(\xi+1)(\rho-1)^2 +\frac{2}{3} t\xi (\rho-1)^3  + \oo(t(\rho-1)^4), \quad \rho e^{\ii\theta_\rho}\to \ii.
%	\end{align}
	Utilizing  \eqref{pr1} and the boundedness of $T$ on $\Omega_1$, we estimate the integral on $\Omega_1$ for $\lambda\ne0$,
	\begin{equation}\label{omega1}
		\begin{split}
			&\iint_{\Omega_1}\abs{\frac{\bar\partial R_1(\varsigma)e^{-\phi}}{\varsigma-\lambda}}\frac{\ddddd L(\varsigma)}{\pi} =\int_{1}^{\rho_0}\int_{\pi-\theta_\rho}^{\theta_\rho}\abs{\frac{\rho^{-1}\partial_\theta(r(e^{\ii\theta}))e^{-\phi(\rho e^{\ii\theta})}}{2\pi(\rho e^{\ii\theta}-\lambda)}}\rho\ddddd\theta\ddddd\rho\\
			&\quad \le\int_{1}^{\rho_0}\int_{-\frac{\pi}{2}}^{\frac{3\pi}{2}}\frac{\abs{\partial_\theta (r(e^{\ii\theta}))}e^{-\re\phi(\rho e^{\ii\theta_\rho})}}{2\pi\abs{\rho e^{\ii\theta}-\lambda}}\ddddd\theta\ddddd\rho\\
			&\quad \le\int_{1}^{\rho_0}\ddddd\rho\frac{e^{-\re\phi(\rho e^{\ii\theta_\rho})}}{2\pi}\norm{r}_{H^1}\left(\int_{-\pi}^{\pi}\abs{\rho e^{\ii\theta}-\lambda}^{-2}\ddddd\theta\right)^\frac{1}{2}\\
			&\quad\lesssim\int_{1}^{\rho_0}\ddddd\rho\frac{e^{-\re\phi(\rho e^{\ii\theta_\rho})}}{2\pi\abs{\rho-\abs{\lambda}}^{1/2}}.
		\end{split}
	\end{equation}
	By \eqref{rephi} and \eqref{omega1}, since we have assumed $\epsilon_0$ small enough, it is readily seen that
	\begin{align}
		\iint_{\Omega_1}\abs{\frac{\bar\partial R_1(\varsigma)e^{-\phi}(\varsigma)}{\varsigma-\lambda}}\frac{\ddddd L(\varsigma)}{\pi} \lesssim \int_{0}^{\rho_0}\ddddd\rho\frac{e^{-\frac{1}{3}t(\rho-1)^3  }}{2\pi\abs{\rho-\abs{\lambda}}^{1/2}}\lesssim t^{-\frac{1}{6}}.
	\end{align}
	For $\Omega_2$, the similar estimate is obtained. Thus, we	prove the result \eqref{s}.
	
\end{proof}

Now, we focus on the estimate of $M^{(2)}_D(0)$
\begin{equation}\label{em2d}
	\begin{split}
		M^{(2)}_D(0)=I+\iint_{\mathbb{C}}M^{(2)}_D(\varsigma)\tilde{\mathcal{R}}(\varsigma)\varsigma^{-1}\frac{\ddddd L(\varsigma)}{\pi}.
	\end{split}
\end{equation}
Again taking $s=\rho e^{\ii\theta}\in\Omega_1$,
\begin{equation}\label{io1}
	\begin{split}
		&\abs{\iint_{\Omega_1}M^{(2)}_D(\varsigma)\tilde{\mathcal{R}}(\varsigma)\varsigma^{-1}\frac{\ddddd L(\varsigma)}{\pi}}\lesssim \iint_{\Omega_1}\abs{\bar\partial R_1(\varsigma)e^{-\phi(\varsigma)}\varsigma^{-1}}\frac{\ddddd L(\varsigma)}{\pi}\\
		&\quad=\int_1^{\rho_0}\frac{e^{-\re\phi(\rho e^{\ii\theta_\rho})}}{2\pi\rho}\ddddd\rho\int_{\pi-\theta_\rho}^{\theta_\rho}\abs{\partial_\theta(r(e^{\ii\theta}))}e^{-\re\phi(\rho e^{\ii\theta})+\re\phi(\rho e^{\ii\theta_\rho})}\ddddd\theta \\
		& \quad \lesssim \int_1^{\rho_0}\frac{ e^{\re(\rho e^{-\ii\theta_\rho})}}{2\pi\rho}\ddddd\rho
	\left(	\left(\int_{-\frac{\pi}{2}}^{\theta_\rho}e^{	-t(1 -\rho^{-2})(\theta-\theta_\rho)^2}\ddddd\theta \right)^{\frac{1}{2}} + \left(\int_{\pi-\theta_\rho}^{\frac{\pi}{2}}e^{	-t(1 -\rho^{-2})(\theta-\pi +\theta_\rho)^2}\ddddd\theta \right)^{\frac{1}{2}} \right)
	\end{split}
\end{equation}
\begin{comment}
By \eqref{e2.18},  it follows that for $\theta \in (\frac{\pi}{2}, \theta_\rho)$,
\begin{align}
	\re\phi(\rho e^{\ii\theta}) - 	\re\phi(\rho e^{\ii\theta_\rho})  =	t(1 -\rho^{-2}) \rho(\sin \theta-\sin \theta_\rho).
\end{align}
Moreover,  it is readily seen that
\begin{align}
	\rho (\sin \theta-\sin \theta_\rho)<
	-\frac{1}{2} \rho \sin \theta_\rho (\theta-\theta_\rho)^2 \le \frac{\xi}{2} (\theta-\theta_\rho)^2 \quad \theta \in  (\frac{\pi}{2}, \theta_\rho),
\end{align}
from which give us that
\begin{align}\label{pe2}
	\re\phi(\rho e^{\ii\theta}) - 	\re\phi(\rho e^{\ii\theta_\rho})< \frac{\xi}{2} 	t(1 -\rho^{-2})(\theta-\theta_\rho)^2.
\end{align}
Substituting the estimate \eqref{pe2} into \eqref{io1}, we obtain
\begin{align*}
	\int_1^{\rho_0}\frac{ e^{\re(\rho e^{\ii\theta_\rho})}}{2\pi\rho}\ddddd\rho
	\left(\int_{-\frac{\pi}{2}}^{\theta_\rho}e^{	t\xi(1 -\rho^{-2})(\theta-\theta_\rho)^2}\ddddd\theta \right)^{\frac{1}{2}}
	\lesssim t^{-\frac{1}{4}}\int_1^{\rho_0} \rho^{-\frac{1}{2}} (\rho^2-1)^{\frac{1}{4}}\frac{e^{-\frac{1}{3}t (\rho-1)^3} }{2 \pi}
	\ddddd\rho\lesssim t^{-\frac{1}{2}}.
\end{align*}
For $\theta \in \left(-\pi-\theta_\rho, -\frac{\pi}{2}\right)$, the estimate is also of the order $\oo(t^{-\frac{1}{2}})$.  It is thereby inferred that
\begin{align}\label{io11}
	\abs{\iint_{\Omega_1}M^{(2)}_D(\varsigma)\tilde{\mathcal{R}}(\varsigma)\varsigma^{-1}\frac{\ddddd L(\varsigma)}{\pi}}\lesssim  t^{-\frac{1}{2}}.
\end{align}
\end{comment}
Through direct calculations, we have
\begin{align}\label{io2}
	\abs{\iint_{\Omega_1}M^{(2)}_D(\varsigma)\tilde{\mathcal{R}}(\varsigma)\varsigma^{-1}\frac{\ddddd L(\varsigma)}{\pi}}\lesssim t^{-\frac{1}{2}}.
\end{align}
The similar result also holds for $\Omega_2$. Then,
substituting \eqref{io2} for $\Omega_1$ and $\Omega_2$ into \eqref{em2d}, it is readily seen that as $t\to+\infty$,
\begin{align}\label{em2de}
	M^{(2)}_D(0)=I+\oo(t^{-\frac{1}{2}}).
\end{align}

\subsection{Asymptotic formula}

Tracking back the transformations \eqref{t1}, \eqref{m2m1}, \eqref{m2mdmrh},  and \eqref{em2de}, we conclude that, in the neighborhood of $\lambda=0$,
\begin{equation*}
	M(\lambda) = 	M^{(2)}_D(\lambda) E(\lambda) M^{(\mathcal{Z})}(\lambda) T(\lambda)^{-\sigma_3}.
\end{equation*}
%Taking $\lambda=0$ in the above equation, we then obtain from \eqref{T}, \eqref{pm2rh}, and \eqref{e0} that
%\begin{align}
%	M(0) = 	M^{(\mathcal{Z})}(0) T(0)^{-\sigma_3}-2  t^{-\frac{1}{3}} M^{(\mathcal{Z})}(-\ii)M_1^{(U)}M^{(\mathcal{Z})}(-\ii)^{-1}M^{(\mathcal{Z})}(0) T(0)^{-\sigma_3} + \oo(t^{-\frac{1}{2}}),
%\end{align}
%\hl{where}%这里是不是写成T_0?
%\begin{align*}
%	T_0 = \prod_{\lambda_{l}\in\mathcal{Z}_\xi^-}\left| \lambda_{l} \right|^{2\alpha_{l}}.
%\end{align*}
%Together with \eqref{e2.46}, the asymptotic formula  \eqref{qn} for $\xi \in $ I is established. % Moreover, by employing a similar approach, we can derive the asymptotic behavior for $\xi \in \mathcal{R}_+$.
Then, it follows by the reconstruction formula \eqref{e2.46}  that
\begin{align}
q_n(t)=T_0\left(M^{(2)}_DEM^{(\mathcal{Z})}\right)_{1,2}(0,n+1,t),
\end{align}
where
\begin{align*}
T_0=
\prod_{\lambda_l\in\mathcal{Z^-_\xi}}\abs{\lambda_l}^{2\alpha_l}.
\end{align*}
The asymptotic formula  \eqref{qn} in the 1st transition zone is established.  Moreover, by employing a similar approach, we can derive the asymptotic behavior in the 2nd transition zone,
 where
\begin{align}
T_0=e^{-\int_{0}^{2\pi}\ln(1+\abs{r(e^{\ii\theta})}^2)\frac{\ddddd\theta}{2\pi}}\prod_{\lambda_l\in\mathcal{Z^-_\xi}}\abs{\lambda_l}^{2\alpha_l}.
\end{align}

\appendix

\section{Parabolic cylinder model}

Here, describe the parabolic cylinder model that is first introduced by \cite{its1981asymptotics} and further by \cite{deift1993steepest}.
And it has been frequently applied to the literature of long-time asymptotics \cite{dieng2008long,Borghese2018long, grunert2009long,kruger2009long,jenkins2018soliton}.
Set $\Sigma^{(PC)}=\bigcup_{k=1}^4(e^{\frac{\pi\ii}{4}(2k-1)}\mathbb{R}^+)$ oriented outwards from the origin.
Consider the parabolic-cylinder RH problem below.
\begin{rhp}\label{rhpc}
	\
	\begin{itemize}
		\item $M^{(PC)}=M^{(PC)}(\zeta,\tau)$ is analytic in $\mathbb{C}\setminus\Sigma^{(PC)}$.
		\item As $\zeta\to\infty$, $M^{(PC)}= I+\oo(\zeta^{-1})$.
		\item For  $\zeta \in \Sigma^{(PC)}$,
		\begin{align}
			&M^{(PC)}_+=M^{(PC)}_-V^{(PC)},\quad
			V^{(PC)}=V^{(PC)}(\zeta,\tau)=
			\begin{cases}
				\left(\begin{matrix}
					1&0\\\tau\zeta^{-2\ii\nu}e^{\frac{\ii\zeta^2}{2}}&1
				\end{matrix}\right)&\zeta\in e^{\frac{\ii\pi}{4}}\mathbb{R}^+,\\
				\left(\begin{matrix}
					1&\frac{\overline{\tau}}{1+\abs{\tau}^2}\zeta^{2\ii\nu}e^{-\frac{\ii\zeta^2}{2}}\\0&1
				\end{matrix}\right)&\zeta\in e^{\frac{3\ii\pi}{4}}\mathbb{R}^+,\\
				\left(\begin{matrix}
					1&0\\\frac{\tau}{1+\abs{\tau}^2}\zeta^{-2\ii\nu}e^{\frac{\ii\zeta^2}{2}}&1
				\end{matrix}\right)&\zeta\in e^{-\frac{3\ii\pi}{4}}\mathbb{R}^+,\\
				\left(\begin{matrix}
					1&\overline{\tau}\zeta^{2\ii\nu}e^{-\frac{\ii\zeta^2}{2}}\\0&1
				\end{matrix}\right)&\zeta\in e^{-\frac{\ii\pi}{4}}\mathbb{R}^+.
			\end{cases}
		\end{align}
		
	\end{itemize}
\end{rhp}
It is well acknowledged that $M^{(PC)}$ admits the asymptotic property at $\zeta\to\infty$ \cite{Borghese2018long,deift1993steepest},
\begin{align}\label{e4.27}
	M^{(PC)}(\zeta,\tau)= I+\frac{1}{\zeta}\left(\begin{matrix}
		0&-\ii\gamma_1(\tau)\\\ii\gamma_2(\tau)&0
	\end{matrix}\right)+\oo(\zeta^{-2}),
\end{align}
where
\begin{equation}
	\begin{split}
		&\gamma_1=\gamma_1(\tau)=\frac{\sqrt{2\pi}e^{\ii\pi/4-\pi\nu(\tau)/2}}{\tau\Gamma(-\ii\nu(\tau))},\quad \gamma_2=\gamma_2(\tau)=\frac{\sqrt{2\pi}e^{-\ii\pi/4-\pi\nu(\tau)/2}}{\bar\tau\Gamma(\ii\nu(\tau))},\quad
		\nu=\nu(\tau)=-\frac{\ln(1+\abs{\tau}^2)}{2\pi}.
	\end{split}
\end{equation}

\section{Painlev\'{e} \uppercase\expandafter{\romannumeral2} RH model} \label{appx}
The   Painlev\'{e} \uppercase\expandafter{\romannumeral2} equation takes the form
\begin{equation}\label{p23}
	v_{ss} = 2 v^3 +s v, \quad s \in \mathbb{R},
\end{equation}
which is generally related to a $2 \times 2$ matrix-valued RH problem as follows \cite{ deift1993steepest, Charlier2020, Deiftzhoup2, FokasAblop2}.
Let $\Sigma^P$ denote the oriented contour consisting of six rays,  $\Sigma^P = \bigcup_{n=1}^6\left\{  \Sigma_n^P = e^{\ii(n-1)\frac{\pi}{3}} \mathbb{R}_+ \right\}$, with associated jump matrix $V^P: \Sigma^P \to M_2(\mathbb{C})$, $V^P|_{\Sigma_1^P} =\begin{pmatrix} 1 & 0 \\ p & 1 \end{pmatrix} $, etc., as depicted in  FIGURE \ref{Sixrays}, where $p$, $q$ and $r$ are complex numbers satisfying the relation
\begin{equation}
p+q+r+pqr =0.
\end{equation}

\begin{figure}
	\begin{center}
		\begin{tikzpicture}[scale=1.2]
			\node[shape=circle,fill=black,scale=0.15] at (0,0) {0};
			\node[below] at (0.3,0.25) {\footnotesize $0$};
			\draw [] (-2.5,0 )--(2.5,0);
			\draw [-latex] (0,0)--(1.25,0);
			\draw [-latex] (0,0 )--(-1.25,0);
			\draw [] (0,0 )--(2.5,2);
			\draw [-latex] (0,0)--(1.25,1);
			\draw [] (0,0 )--(2.5,-2);
			\draw [-latex] (0,0)--(1.25,-1);
			\draw [] (0,0 )--(-2.5,2);
			\draw [-latex] (0,0)--(-1.25,1);
			\draw [] (0,0 )--(-2.5,-2);
			\draw [-latex] (0,0)--(-1.25,-1);

			\node at (1.6,0.2) {\footnotesize$\Sigma^P_1$};
			\node at (1.1,-1.2) {\footnotesize$\Sigma^P_6 $};
			\node at (-1.2,1.2) {\footnotesize$\Sigma^P_3$};
			\node at (-1.6,-0.2) {\footnotesize$\Sigma^P_4$};
			\node at (1.2,1.2) {\footnotesize$\Sigma^P_2$};
			\node at ( -1.2,-1.2) {\footnotesize$\Sigma^P_5$};

					\node  at (3,0) {\footnotesize $\begin{pmatrix} 1 & p \\ 0 & 1 \end{pmatrix}$};
					\node  at (2.5,2.5) {\footnotesize $\begin{pmatrix} 1 & 0\\ q& 1 \end{pmatrix}$};
				\node  at (-2.5,2.5)  {\footnotesize $\begin{pmatrix} 1 & r\\ 0& 1 \end{pmatrix}$};
				\node  at (2.5,-2.5) {\footnotesize $\begin{pmatrix} 1 & 0 \\ r& 1 \end{pmatrix}$};
				\node  at (-2.5,-2.5) {\footnotesize $\begin{pmatrix} 1 & q\\ 0& 1 \end{pmatrix}$};
				\node  at (-3,0) {\footnotesize $\begin{pmatrix} 1 & 0 \\ p & 1 \end{pmatrix}$};
			
		\end{tikzpicture}
		\caption{ \footnotesize { The jump contour $\Sigma^P$.}}
		\label{Sixrays}
	\end{center}
\end{figure}

%The Painlev\'{e} \uppercase\expandafter{\romannumeral2} RH problem satisfies the following properties:
\begin{rhp}\label{1modp2}
	%	Find   $M^{P}(k)=M^{\mathrm{P}}(k,s)$ with properties
	\
	\begin{itemize}
		\item $M^{P}=M^{P}(k;s)$ is analytical in $\mathbb{C}\setminus \Sigma^{P}$.
		\item $M^{P}$ satisfies the jump condition
		\begin{equation*}
			M^{P}_+(k)=M^{P}_-(k)e^{-\ii (\frac{4k^3}{3}+sk)\sigma_3} V^P(k) e^{\ii (\frac{4k^3}{3}+sk)\sigma_3}, \quad k \in \Sigma^P.
		\end{equation*}

		\item As $z\to \infty$ in $\mathbb{C}\setminus \Sigma^{P}$, we have $M^{P}(k)=I+\mathcal{O}(k^{-1})$.

	\end{itemize}
\end{rhp}
%\noindent has a unique solution $M^P(k)$ for each $s \in \mathbb{C} \setminus  \mathcal{S}_\mathcal{C}$. For each $n$, the restriction of $M^P(\hat{k})$ to $\arg \hat{k} \in \left(\frac{\pi(2n-3)}{6}, \frac{\pi(2n-1)}{6}\right)$ admits an analytic continuation to $\left( \mathbb{C} \setminus  \mathcal{S}_\mathcal{C} \right) \times \mathbb{C}$ and  there are smooth function $\{M_j^P(s)\}_{j=1}^\infty$ of $s \in \mathbb{C} \setminus  \mathcal{S}_\mathcal{C}$ such that, for each integer $N \ge 0$,
Then,
\begin{align}\label{up2}
v(s)=2	\ii \left(M_1^\mathrm{P}(s)\right)_{12} =-2 \ii  \left(M_1^\mathrm{P}(s)\right)_{21},
\end{align}
solves the Painlev\'{e} \uppercase\expandafter{\romannumeral2} equation \eqref{p23}, where
\begin{equation}\label{stanp}
M^P(k) = I +\frac{M_1^P(s)}{k} + \mathcal{O} \left(\frac{1}{k^2}\right)
\end{equation}
as $k\to \infty$.
Furthermore, for any $q \in \mathbb{C}$
\begin{equation}
q, \quad p=\bar{q}, \quad r = -\frac{q+\bar{q}}{1+|q|^2} \in \mathbb{R},
\end{equation}
formula \eqref{up2} leads to a global, pure imaginary solution of the Painlev\'{e} \uppercase\expandafter{\romannumeral2} equation. In this case, the leading coefficient $M_1^P(s)$ is given by
\begin{align}\label{posee}
M_1^P(s) = \frac{\ii}{2} \begin{pmatrix} -\int_{s}^\infty v(\zeta)^2\ddddd \zeta &- v(s)  \\ v(s) & \int_{s}^\infty v(\zeta)^2\ddddd \zeta  \end{pmatrix}.
\end{align}

\begin{comment}
%The above RH problem admits a unique solution. Then,  there exists a smooth function $\{M_1^P(s)\}_{j=1}^\infty$ such that, for each integer $N \ge 0$, we have
\begin{equation}\label{stanp}
	M^P(k) = I + \sum_{j=1}^N\frac{M_j^P(s)}{k^j} + \mathcal{O}(k^{-N-1}),\quad k\to \infty.
\end{equation}
%uniformly for $s$ in compact subsets of  $\mathbb{C} \setminus  \mathcal{S}_\mathcal{C}$ and for $\arg \hat{k} \in [0, 2\pi]$.
Moreover,
\begin{align}\label{up2}
	v(s)=2	\left(M_1^\mathrm{P}(s)\right)_{12} =2  \left(M_1^\mathrm{P}(s)\right)_{21}
\end{align}
solves the Painlev\'{e} \uppercase\expandafter{\romannumeral2} equation \eqref{p23}.
Further, \hl{if} $\mathcal{C} = (c_1,0,-c_1)$ where $c_1 \in \ii \mathbb{R}$ with $|c_1| <1$, then the leading coefficient $M_1^P(s)$ is given by%符号和前的cauchy算子符号重复
\begin{align}\label{posee}
	M_1^P(s) = \frac{1}{2} \begin{pmatrix} -\ii\int_{s}^\infty v(\zeta)^2\ddddd \zeta & v(s)  \\ v(s) & \ii\int_{s}^\infty v(\zeta)^2\ddddd \zeta  \end{pmatrix},
\end{align}
and for each $c_1 > 0$,
\begin{align}\label{mPbounded}
	\sup_{k \in \mathbb{C}\setminus \Sigma^P} \sup_{s \geq -c_1} |M^P(k)|  < \infty.
\end{align}
The solution of the Painlev\'{e} \uppercase\expandafter{\romannumeral2} equation \eqref{p23} is specified by
\begin{equation}
	v(s) \sim -\I c_1 \mathrm{Ai}(s) \sim - \frac{\I c_1}{2\sqrt{\pi}} s^{-\frac{1}{4}}  e^{-\frac{2}{3}s^{\frac{3}{2}}},\ s \to +\infty,
\end{equation}
where $\mathrm{Ai}(s)$ denotes the classical Airy function.
\end{comment}

\bibliography{wileyNJD-AMS}
\bibliographystyle{siam}

\end{document}